\documentclass[11pt]{article}
\usepackage[utf8]{inputenc}
\usepackage[usenames,dvipsnames]{xcolor}
\usepackage{mathtools}
\usepackage{mathrsfs}
\usepackage{bbm}
\usepackage[a4paper, margin =2cm]{geometry}

\usepackage{amsmath,amsfonts,amssymb,graphics,amsthm, comment}
\usepackage{hyperref}
\hypersetup{
 pdftitle=Brownian paths as loop-decorated SLEs,
 pdfauthor=Nathana\"el Berestycki Isao Sauzedde,
    colorlinks=true,
    linkcolor=blue,
    citecolor=RedViolet,
    urlcolor=blue,
    pdfborder={0 0 0}
}

\usepackage[font=small]{caption}
\usepackage{enumerate}

\newtheorem*{lemma*}{Lemma}

\newtheorem{theorem}{Theorem}[section]
\newtheorem*{theorem*}{Theorem}

\newtheorem{lemma}[theorem]{Lemma}

\newtheorem{proposition}[theorem]{Proposition}
\newtheorem{corollary}[theorem]{Corollary}
\newtheorem{definition}[theorem]{Definition}

\newtheorem{remark}[theorem]{Remark}

\newtheorem{remarks}[theorem]{Remarks}

\def\d{\,{\rm d}}
\def\R{{\mathbb R}}

\def\P{{\mathbb P}}
\def\Z{{\mathbb Z}}

\def\cL{{\mathcal L}}
\def\cF{{\mathcal F}}
\def\cX{{\mathcal X}}
\def\eps{\varepsilon}
\DeclareMathOperator{\Range}{Range}
\def\Conf{{\mathscr{R}}}

\def\rhos{\mathbf{r}}
\def\Att{{\Xi}}

\title{Brownian paths as loop-decorated SLEs
}
\author{Nathana\"el Berestycki\thanks{University of Vienna, Austria, \texttt{nathanael.berestycki@univie.ac.at}}\ \ and Isao Sauzedde\thanks{\'Ecole Normale Supérieure de  Lyon, France, \texttt{isao.sauzedde@ens-lyon.fr} } }
\date{\today}

\begin{document}

\maketitle
\abstract{We construct an application, which takes as input a simple path and a possibly infinite collection of loops, and outputs a continuous path by adding the loops chronologically to the simple path as the simple path encounters them. By studying the regularity properties of this application and using lattice discretisations, we prove that chronologically adding the loops from a Brownian loop soup encountered by an independent radial SLE$_2$ path produces a continuous path which has the law of a planar Brownian motion. This resolves a conjecture of Lawler and Werner \cite[Conjecture 1]{LawlerWerner}.

This construction produces a coupling between SLE$_2$ and Brownian motion, and we further show that this joint law is the scaling limit of the loop-erased random walk and the random walk itself. The arguments are robust and can be applied for instance in the off-critical setup, where the scaling limit of loop-erased random walk is Makarov and Smirnov's massive SLE$_2$ \cite{MakarovSmirnov}.

}

\medskip \noindent \small{\textbf{Keywords:} Loop-erased random walk, Schramm--Loewner Evolution (SLE), Brownian motion, Brownian loop soup, random walk loop soup, scaling limits. \\
\textbf{MSC 2020 classification}: 60J67, 60J65, 82B41.}

\tableofcontents

\section{Introduction}
\label{sec:introduction}
\subsection{Main results}

Let $D \subset \R^2$ be a proper simply connected domain. Let $o\in D$. Let $(W_t, 0 \le t \le \tau_D)$ be a Brownian motion starting from $o$ and run until the time $\tau_D$ where it leaves $D$ for the first time. A classical question in the area of conformally invariant random processes is whether it is possible to meaningfully erase the loops created by $W$ chronologically and obtain an SLE$_2$ path. (We recall that, following the landmark result of Lawler, Schramm and Werner \cite{LSW}, the scaling limit of the loop-erased random walk starting from a vertex converging to $o$, is the so-called radial SLE$_2$ path; see below for more precise definitions).

The inverse of loop-erasure is loop addition and in this paper we show that it is possible to \textbf{chronologically} add loops to a radial SLE$_2$ path $\gamma$ in $D$ in such a way as to obtain a Brownian motion $W$. These loops are obtained as the set of loops $\cL_\gamma$ which intersect $\gamma$ in a Brownian loop soup $\cL$ with intensity $c = 1$ (again, see below for precise definitions).

Results in this vein go back at least to the work of Lawler, Schramm and Werner \cite{LSWrestriction} who showed as a consequence of their study of the conformal restriction property that the \textbf{hull} generated by a chordal SLE$_2$ path $\gamma$ in the upper half plane and the loops $\cL_\gamma$ coincides with the hull of a Brownian excursion in the upper half plane: in other words, ``filling-in the holes'', the reunion of $\gamma$ and $\cL_\gamma$ has the same distribution as a Brownian excursion.  A more precise result was shown by Zhan~\cite{Zhan}, who proved (among other things) that there is a random nondecreasing function $\psi: [ 0, \tau_D] \to \R$ such that $W_{\psi(t)}$ is a radial SLE$_2$ path. Even more recently,
Sapozhnikov and Shiraishi \cite{SapozhnikovShiraishi} showed that the \textbf{range} $\Range (W) = W ([0, \tau_D]) = \{ W_t: 0\le 0 \le t \le \tau_D\} \subset \R^2$ of the Brownian motion $W$, viewed as a random compact set, has the same law as the range of $\gamma$ together with the ranges of loops in  $\cL_\gamma$. More precisely, they showed

\begin{theorem*}[{\cite[Theorem 1.1]{SapozhnikovShiraishi}}]
	In distribution,
	$
	\Range( W)=\Range(\gamma)\cup \bigcup_{\ell \in \mathcal{L}_\gamma } \Range(\ell).
	$
\end{theorem*}
(Remarkably, their result is also true in dimension $d = 3$.) Here both sides are considered in the space of compact subsets of $D$, endowed with the Hausdorff metric and with the associated Borel $\sigma$-algebra. As this is an identity in distribution for random sets, it does not convey any information about the order with which these sets are traversed.

Our main contribution is to show that one can \textbf{chronologically} add loops of $\cL_\gamma$ to $\gamma$ to obtain a parametrized path, which has the law of Brownian motion.

The result below, which resolves a conjecture of Lawler and Werner \cite[Conjecture 1]{LawlerWerner} is the main theorem of this article, stated in a somewhat informal manner, which will be made precise below. (As we will explain further below, our arguments are robust and function equally well in the chordal setup, which is in fact the original setting of the conjecture of Lawler and Werner; interestingly they  apply to the near-critical or \textbf{massive} setup which we will introduce later on.)

\begin{theorem}\label{thm:main_intro}
  Let $\gamma$ be a radial SLE$_2$ and let $\cL$ be an independent Brownian loop soup with intensity $c = 1$. Let $\cL_\gamma$ denote the set of loops in $\cL$ intersected by $\gamma$.

  The loops of $\cL_\gamma$ can be totally ordered by the first time $\gamma$ intersects them, starting from the root~$o$. The chronological concatenation of all the loops in $\cL_\gamma$ gives a continuous parametrized path $\Att(\gamma,\mathcal{L})$, which furthermore has the same law as $(W_t, 0 \le t \le \tau_D$).
\end{theorem}

The process of adding loops chronologically with respect to the order in which they are encountered by $\gamma$ is explicitly suggested in \cite{LawlerWerner}. We emphasise that in any interval of time with positive length, $\gamma$ will encounter infinitely many loops from $\cL_\gamma$ and thus the chronological concatenation of these loops is not obviously defined. This was however justified in \cite[Lemma 13]{LawlerWerner} and the paragraph below it (who also state that such a concatenation is necessarily continuous). As part of our proof we will give a somewhat simpler proof of this first fact and a complete proof of the latter. Our construction also applies to other values of $\kappa$, see the open problems in Section \ref{sec:open} for a discussion.

\medskip The construction above provides in particular a coupling $(\gamma, W)$ between an SLE$_2$ path $\gamma$ and a Brownian motion $W$. To further illustrate the sense in which $\gamma$ is the ``loop-erasure'' of $W$ in this construction (and as a main step of the proof of the theorem), we show that actually the law of the pair $(\gamma, W)$ in the above coupling is the scaling limit as $n\to \infty$ of the pair $(\gamma^n, W^n)$. Here $W^n$ is a random walk on the scaled square lattice $n^{-1} \Z^2$: that is, if $(S^n(0), S^n(1), S^n(2), \ldots)$ is a random walk in discrete time on $\Z^2$, and $(S^n (t), t \ge 0)$ is its linear interpolation, then for $t\ge 0$ we let $W^n(t) = o^n + n^{-1} S^n (  2n^2 t )$, where the vertices~$o^n\in n^{-1} \Z^2$  converge to $o$. To state the result precisely,
consider the discretised domain $D_n $ which is the component containing $o^n$ of $\{w \in n^{-1} \Z^2: B(w, 1/(2n)) \subset D\}$, and let $\tau_D^n = \inf \{ t\ge 0: W^n(t) \notin D_n\}$ be the first time that $W^n$ leaves $D_n$. Finally, $\gamma^n = LE (W^n[0, \tau^n_D]) $ is the chronological loop-erasure of $W[0, \tau_D^n]$ (we refer for instance to \cite{LawlerLimic} for a definition of the chronological loop-erasure of a discrete path).

\begin{theorem}
There exists a sequence $\epsilon_n\to 0$ which satisfies the following properties.
For each $n \ge 1$, there is a coupling between $(\gamma^n, W^n)$ on the one hand, and $(\gamma, W)$ on the other hand, such that with probability at least $1-\epsilon_n$,
\begin{align*}
&\| W^n (\ \cdot \ \wedge \tau^n_D) - W
 (\  \cdot \ \wedge \tau_D)\|_\infty  \le \epsilon_n, \text{ and }\\
& \mathrm{d}_{\mathrm{Hausdorff}} (\gamma^n , \gamma)  \le \epsilon_n.
\end{align*}
Here $\mathrm{d}_{\mathrm{Hausdorff}}$ refers to the Hausdorff distance between two compact sets. In fact, if we view $\gamma$ as parametrized by its so-called ``natural parameterisation'', one also has $\| \gamma^n - \gamma\|_\infty \le \epsilon_n$.
\end{theorem}

\begin{remark} \label{R:coupling3}
In fact, the argument shows that we can not only couple the pair $(\gamma^n, W^n)$ with $(\gamma, W)$ such that $(\gamma^n,W^n)$ is close to $(\gamma,W)$ in $\| \cdot \|_\infty$ norm, but also that the triplet $(\gamma^n, \cL^n, W^n) $ is close to $(\gamma, \cL, W)$ (see Theorem~\ref{th:main} below for the full statement).
 Here $\cL^n$ is a random walk loop soup (independent of $\gamma^n)$, and all the loops in $\cL^n$ beyond a certain mesoscopic scale are close to all the loops in $\cL$. 

Note furthermore that, by a result of Lawler and Trujillo-Ferreras \cite{LawlerFerreras}, the loops in $\cL^n$ beyond a certain mesoscopic scale are in one-to-one correspondence and close to the loops of the corresponding sizes in $\cL$. (In fact we mention that, very recently, Qian \cite{Qian} announced that it is possible to remove the polynomial restriction on the range of the loops that are coupled together).   \end{remark}

The way we obtain such a result relies on both probabilistic estimates and topological arguments. We construct a deterministic application $\Att$ which takes as input a simple path and a collection of loops and chronologically adds the loops to the path. Although it is unrealistic to find a topology that would make the map $\Att: ({\textsf{path, loops}})\mapsto \Att({\textsf{path, loops}})$ continuous, the pairs $({\textsf{path, loops}})$ around which $\Att$ has a singular behavior are somewhat atypical and essentially fall into one of three sorts of specific issues (see Figure~\ref{fig:issues} below). Indeed, we are able to define a subspace $\Conf$ of pairs $({\textsf{path, loops}})$ which is sufficiently large to ensure that the pair formed by the SLE$_2$ path $\gamma$ and the Brownian loop soup $\mathcal{L}$ almost surely lies in $\Conf$, and a topology sufficiently strong to ensure that the application
$\Att$ is continuous on $\Conf$.\footnote{We emphasize that it is not only the restriction of $\Att$ to $\Conf$ which is continuous, but the globally defined application $\Att$ which is continuous on $\Conf$.} Yet, this topology is sufficiently weak that the lattice approximation of
$(\gamma,\mathcal{L})$ by the pair of a loop erased random walk and a random walk loop soup converge in distribution for this topology, and the initial and analytically difficult question of whether $\Att(\gamma,\mathcal{L})$ is distributed as a Brownian motion is reduced to its discrete counterpart, which is of purely combinatorics nature.

The main contributions of this paper can be summarized as follow:
\begin{itemize}
\item The introduction and rigorous construction of the `right' objects to work with, namely the space of configurations on which the map $\Att$ is defined (some simplifications were made in this introduction for the sake of clarity), the map $\Att$ itself, and the subset $\Conf$ (Section~\ref{sec:construction})
, and the adequate distance $d_{\Conf_0}$ (Section~\ref{sec:continuity}).
\item In Sections~\ref{sec:attaching} and~\ref{sec:continuity}, the proof of the continuity of $\Att$ on the subset $\Conf$ for the distance $d_{\Conf_0}$ (Theorem~\ref{th:continuityAtt}), together with tools which reduce the question of convergence toward an element of $\Conf$ to conditions simpler to check (Lemma~\ref{le:Tbound} and Lemma~\ref{le:omegabound}).
\item The proof that the pair $(\gamma,\mathcal{L})$ almost surely lies in $\Conf$ (Section~\ref{sec:probabilistic1}, Proposition~\ref{prop:conf}), and that the lattice approximations converge in distribution, for the topology of $d_{\Conf_0}$ (Section~\ref{sec:probabilistic2}, Proposition~\ref{prop:converg}).
\end{itemize}

\begin{figure}[ht]
	\centering
	\includegraphics[height=3.5cm]{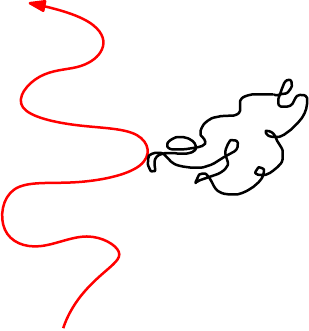}\hspace{1cm}
\includegraphics[height=3.5cm]{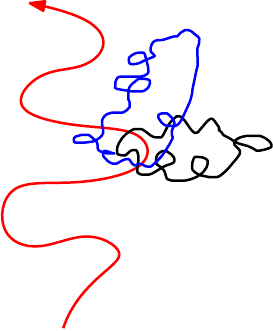}\hspace{1cm}
\includegraphics[height=3.5cm]{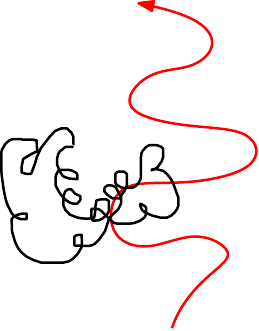}

  \begin{minipage}[t]{0.8\linewidth}

	\caption{The three types of issues that prevent continuity of the map $\Att$. In red, the simple path. In black and blue, loops $\ell,\ell'\in \mathcal{L}$. Around the root point $x_\ell$ where $\gamma$ meets $\ell$ for the first time:  {\it Left:}  the intersection is one-sided. Small continuous deformations of $\gamma$ or $\ell$ can remove this intersection entirely,  rather than moving it continuously. {\it Middle:} $\gamma$ also meets $\ell'$ for the first time at the same place as $\ell$. Small continuous deformations can make $\gamma$ meet either $\ell$ or $\ell'$ first. The resulting path are very different, although their range are about the same. {\it Right:} $\ell$ passes through $x_\ell$ at two distinct times $s,t$. Some small deformations lead to rerooting $\ell$  around $s$, while others lead to rerooting $\ell$ around~$t$, again leading to two fairly distinct paths.
	}
	\label{fig:issues}
  \end{minipage}
\end{figure}


The (deterministic) topological arguments which underlie our approach have the advantage of being very robust. For instance, these tools apply to different topologies (e.g., on Riemann surfaces) or to cover the 
\textbf{chordal setting} (which, as already mentioned, was actually the original setting for Lawler and Werner's conjecture \cite{LawlerWerner}. In this setting, $D = \mathbb{H}$ is the upper half plane; $\gamma$ is now a \emph{chordal} (instead of radial) SLE$_2$ path in $\mathbb{H}$ from 0 to $\infty$, and $W$ is a Brownian excursion from 0 to $\infty$ in $\mathbb{H}$ (that is, a Brownian motion starting from 0, and conditioned to stay in the upper half plane for all positive times, and to leave via $\infty$ -- equivalently, its real and imaginary parts are independent Brownian motion and three-dimensional Bessel processes respectively; see \cite{BN}). The corresponding loop soup is unchanged. The analogous result of Theorem \ref{thm:main_intro} to the chordal follows in the exact same manner as the radial case, since:
\begin{itemize}
    \item It is known how to make sense of a random walk $W^n$ in $(1/n)\mathbb{Z}^2 \cap \mathbb{H}$, starting from $i/n$, and conditioned to stay in $\mathbb{H}$ forever (and leaving $\mathbb{H}$ through infinity).
    \item As it is a transient path, its chronological loop-erasure $\gamma^n$ is well defined a.s., and converges to chordal SLE$_2$ by the landmark result of \cite{LSW} as $n\to \infty$ (say, in the Hausdorff sense, or with respect to uniform convergence up to reparametrisation). 
    \item Adding the loops from the loop soup $\cL^n$ (the same as in Remark \ref{R:coupling3}) encountered by $\gamma^n$ produces a path with the same law as $W^n$.
    \item To apply the arguments in this paper, the only estimate that remains to be proved is convergence of $(\gamma^n, \cL^n)$ to $(\gamma, \cL)$ in the sense of the topology defined in Section \ref{sec:continuity}. This boils down to proving that the total time spent on small loops is negligible. For ordinary random walk, these estimates are obtained in Section \ref{SS:timesmallloops}. One way to argue is by absolute continuity: indeed, for any fixed $\eps>0$, if $\tau_\eps = \inf\{ t: \Im ( W^n(t)) \ge \eps\}$, then the law of $W^n[\tau_\eps, \tau_{\eps^{-1}}]$ under the chordal law is absolutely continuous (uniformly in $n$, for each fixed $\eps$) with respect to that of the radial law in an appropriate domain. This can be used to prove the corresponding estimates on the time interval $[\tau_\eps, \tau_{\eps^{-1}}]$ and thus the validity of the desired identity on this time-interval. Since $\eps$ is arbitrary, this implies the claim.   
\end{itemize}


\medskip To illustrate further the robustness of this approach let us mention another setting in which our proof applies, which is the \textbf{massive} or \textbf{off-critical} setting, which has attracted considerable renewed attention since the inspiring work of Makarov and Smirnov \cite{MakarovSmirnov} where,  motivated by potential connections to (non-conformal) quantum field theories, they formulated a programme of studying near-critical scaling limits. 

Let us explain this briefly by focusing on the case of interest, namely loop-erased random walk. Let us fix a bounded simply connected domain $D$ and real parameter $m \in \R$. Suppose that the random walk $W^n$, on the lattice $ (1/n) \mathbb{Z}^2 \cap D$, has a fixed probability of dying at every step, given by $m^2/(2n^2)$, but condition it to leave the domain $D$ before dying (by convention, the mass is $m^2$, and this corresponds to the rate at which the corresponding Brownian motion is killed in the continuum scaling limit). 

Makarov and Smirnov sketched a proof in 
\cite{MakarovSmirnov} that its loop-erasure $\gamma^n$ converges in the scaling limit to a certain random curve, which they called massive SLE$_2$, and which is locally (i.e., on bounded intervals of time) absolutely continuous with respect ordinary SLE$_2$. A complete (and nontrivial) proof of this fact was subsequently given by Chelkak and Wan \cite{ChelkakWan} in the slightly harder chordal setting. This was extended in \cite{BHS} to random walk with fixed drift (i.e., drift of a fixed intensity in a fixed direction). Indeed it turns out that the loop-erasure $\gamma^n$ of this walk $W^n$, conditional on its endpoint, converges to the \emph{same} massive SLE$_2$ (where the mass $m^2$ is simply the Euclidean norm of the limiting drift). In fact, that result was generalised to random walk with variable drift, provided that the drift derives from a log-convex potential. In that case, the scaling limit of $\gamma^n$ is a suitably generalised notion of massive SLE$_2$; see \cite{BHS}. (These ideas have played a crucial role in recent developments connecting the off-critical dimer model of statistical mechanics to a quantum field theory known as the sine-Gordon model; see \cite{BHS, BMR}.)

\medskip Given the above works, it is natural
to guess that massive SLE$_2$ (say, with fixed mass $m\in \R$, as in the setting of Makarov and Smirnov \cite{MakarovSmirnov} for simplicity) is the ``loop-erasure'' of Brownian motion killed at rate $m^2$, conditioned to leave the domain $D$ before dying. In fact, this can be proved as a straightforward consequence of our approach, in the same sense as Theorem \ref{thm:main_intro}. We state this as a corollary:

\begin{corollary}
    \label{cor:massive}
Let $D$ be a bounded, simply connected domain and let $o \in D$. Let $\gamma$ be a radial massive SLE$_2$ from $o$, and let $\cL$ be an independent Brownian loop soup with intensity $c = 1$, where each loop $
\ell \in \cL$ of duration $t_\ell$ is independently retained with probability $\exp ( - m^2 t_\ell)$. Let $\cL_\gamma$ denote the set of loops in $\cL$ intersected by $\gamma$. 

The loops of $\cL_\gamma$ can be totally ordered by the first time $\gamma$ intersects them, starting from the root~$o$. The chronological concatenation of all the loops in $\cL_\gamma$ gives a continuous, parametrized path, which furthermore has the same law as $(W_t, 0 \le t \le \tau_D$) conditional on $\tau > \tau_D$, where $\tau$ is an independent exponential random variable with rate $m^2$, and $\tau_D$ is the first exit time of $D$.
\end{corollary}

\subsection{Precise statements of Theorem~\ref{thm:main_intro}}

We view the space of rooted and oriented loops as taking values in the space
\[
\mathcal{C}([0,\cdot], \mathbb{R}^2) = \bigcup_{t\ge 0} \mathcal{C}([0,t], \mathbb{R}^2).
\]
In the rest of the article, for $f\in \mathcal{C}([0,\cdot], \mathbb{R}^2)$ we write $t_f$ for the unique $t$ such that $f\in\mathcal{C}([0,t], \mathbb{R}^2) $.
The space $\mathcal{C}([0,\cdot], \mathbb{R}^2)$ is a Polish space when endowed with the distance $d_\infty$ given by
\[d_\infty(f,g)\coloneqq |t_f-t_g| + \sup_{t \geq 0} |f(t\wedge t_f)-g(t\wedge t_g)|. \]

By definition \cite{LawlerWerner}, the loop measure is the Borel measure on $\mathcal{C}([0,\cdot], \mathbb{R}^2)$ given by
\[
\mu ( \mathrm{d} \ell ) =\int_0^\infty \int_{\mathbb{R}^2} \frac{1}{2 \pi t^2} \mathbb{P}_{t,x,x} ( \mathrm{d} \ell) \d x \d t,
\]
where $\mathbb{P}_{t,x,y}$ is the distribution of a Brownian bridge from $x$ to $y$ with duration $t$. We fix a simply connected planar domain $D\neq \mathbb{C}$, and we let $\mu^D$ be the restriction of $\mu$ to loops $\ell \in \Omega$ which remain in $D$, that is, $\Range(\ell) \subset D$.
\footnote{
In fact, the loop measure considered in \cite{LawlerWerner} corresponds to a slightly different state space and topology, and it is usually the projection of $\mu^D$ onto the space of \textbf{un}rooted loops which is of interest. (It is however convenient for us to consider the loops as being rooted, though the precise choice of the roots is irrelevant for our construction below).}

Let $\cL$ denote the \textbf{Brownian loop soup} in $D$ with intensity $c = 1$, i.e., a Poisson point process on $\mathcal{C}([0,\cdot], \mathbb{R}^2)$ with intensity $\mu^D$. The probability distribution of $\cL$ is written $\mu^{\mathcal{L}}$.
\medskip

We also consider a (time-reversed) radial SLE$_2$ path in $D$ from $o$ to a random point distributed according to the harmonic measure seen from $o$ on $\partial D$. The range of this path has Hausdorff dimension $5/4$ and has a continuous parametrisation by $5/4$-Minkowski content~\cite[Theorem 1.1]{LawlerRezaei}, which we refer to as its natural parameterisation. 
We let $\gamma$ be the SLE$_2$ path with this parameterisation, and let $\mu^\gamma$ be its probability distribution. In particular, $\gamma(0)=o$ and $\gamma(t_\gamma)\in \partial D$.  The theorem below is a more precise formulation of Theorem~\ref{thm:main_intro}.
\begin{theorem}\label{th:main_precise}
    Let $(\gamma,\mathcal{L})$ be distributed as $\mu^\gamma\otimes \mu^{\mathcal{L}}$, and let $\mathcal{L}_\gamma\coloneqq \{ \ell \in \mathcal{L}: \Range(\ell)\cap \Range(\gamma)\neq \emptyset\}$.
    For $\ell \in \mathcal{L}_\gamma$, let $\sigma_\ell\coloneqq \inf\{\sigma: \gamma(\sigma) \in \Range(\ell) \}$.

    Then, almost surely, for all $\ell\neq \ell'\in  \mathcal{L}_\gamma$, it holds $\sigma_\ell\neq \sigma_{\ell'}$, and there exists a unique $\theta_\ell\in (0, t_\ell)$ such that $\ell(\theta_\ell)= \gamma(\sigma_\ell)$.
    Define a total order on $\cL_\gamma$ by $\ell \prec \ell'\iff \sigma_\ell<\sigma_{\ell'}$, and let $T_\ell\coloneqq \sum_{\ell'\prec \ell} t_\ell$ (to be thought of as the time it takes to the path $X=\Att(\gamma,\mathcal{L})$ to reach the loop $\ell$), and $t_X\coloneqq \sum_{\cL_\gamma} t_\ell$.
    Almost surely, $t_X<\infty$, and there exists a unique continuous path $X\in \mathcal{C}([0,t_X],\bar{D})$ such that for all $\ell \in \mathcal{L}_\gamma$ and all $t\in [T_\ell,T_\ell+t_\ell]$, $X(t)= \ell(   t-T_\ell+\theta_\ell \mod t_\ell)$.

    Furthermore, $X$ is distributed as a Brownian motion started from $o$ and defined up to its first exit time from $D$.
\end{theorem}
The uniqueness of such a path $X$ follows directly from Proposition~\ref{prop:conf} and Lemma~\ref{le:dense} below, and the fact it is distributed as a Brownian motion is obtained in Theorem~\ref{th:main}.

\begin{figure}[ht]
	\centering
	\includegraphics[width=0.8\linewidth]{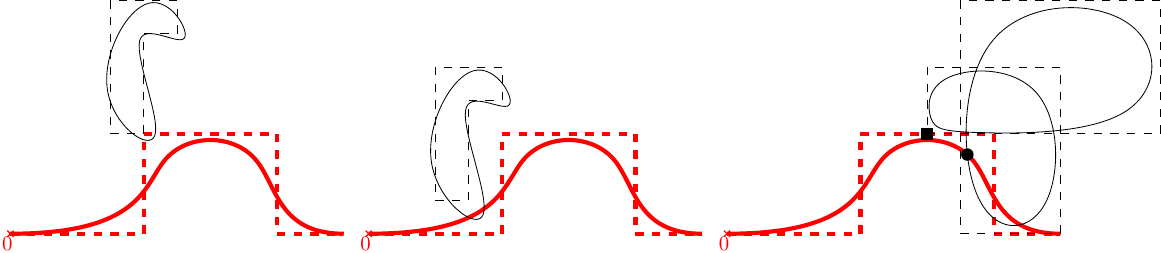}

  \begin{minipage}[t]{0.8\linewidth}

	\caption{Three examples of bad events. In red (thick lines), the $SLE_2$ process. In black (thin lines), one of the loop in $\mathcal{L}$. Dashed lines: their discrete counterparts.
		 {\it left:} Although $\ell$ and $\gamma$ do not intersect, the corresponding discrete processes intersect, so the reconstructed path $W(\mathbb{L}, \mathsf{S})$ will incorporate this loop.
		{\it middle:} The opposite event.  {\it right:} The paths $\ell$ and $\gamma$
		intersect, so does their discrete equivalent, the continuum and discrete roots (bullet and squared bullet) are close from each other spatially, but correspond to very different parts of $\gamma$. Remark it can also happen that the continuum and discrete roots both exist but are spatially far from each other. If the red curves now represent the boundary $\partial D$ rather than the $SLE_2$ process, these figures depict other atypical events that we need control over.
	}
	\label{fig:events}
  \end{minipage}
\end{figure}

\subsection{Strategy and organisation of the paper}

\label{SS:strategy}

\textbf{Approach and main ideas.} Before discussing some of the key points of the proof it is useful to first give a brief overview of the approach of Lawler and Werner in \cite{LawlerWerner}. Their viewpoint is to consider the set of loops as they are discovered chronologically along the (chordal) SLE$_2$ path $\gamma$. Since the (unrooted) Brownian loop soup is conformally invariant, one may map out at any point $s$ the portion $
\gamma[0,s]$ of the SLE$_2$ path discovered by time $s$ along with the loops that it has intersected so far. What remains is a Brownian loop soup in the complement of this domain, which is the conformal image of an independent Brownian loop soup in the upper-half plane. In this way, the sequence of loops encountered by the SLE path can be described as a Poisson point process whose intensity is an infinite measure on loops in the upper half plane rooted at the origin, the so-called bubble measure. A similar property holds at the discrete level as well. As mentioned by Lawler and Werner, to exploit this approach in order to prove their conjecture, the key difficulty is to prove a form of convergence of discrete bubble measures to their continuum counterpart. This is potentially tricky due to the fact that this would need to be proved in domains with very rough boundaries (namely, the complement of a piece of loop-erased random walk and loops intersected by it).

By contrast, our approach is more ``hands-on'' and is more closely related to the one used in~\cite{SapozhnikovShiraishi}. We consider a discrete analogue of the construction in Theorem~\ref{th:main_precise}, in which we add the loops of a discrete random walk loop soup encountered by an independent loop-erased random walk; it is already known that this produces a path with the same law as a random walk. We then check that as we take a scaling limit, if the loop-erased random walk is close to the SLE$_2$ path and the discrete loop soup is close to the continuum loop soup, then the resulting concatenation of discrete loops is close to its continuum analogue and must therefore be a Brownian motion.

The main difficulties relate to the fact that a small deformation of a path and/or a loop could result in missing a loop entirely, or on the contrary adding many new loops, or rooting it at a different place (see Fig.~\ref{fig:events}).
Luckily, we show that the rerooting map $\ell \mapsto \theta_\ell$ is discontinuous only at loops which visits $\ell(\theta_\ell)$ either twice or in a \textit{one-sided} way.  Another source of difficulty is the need to control the total time spend on microscopic and mesoscopic loops, uniformly over the deformation. We introduce the space $\Conf$ on which these atypical visits do not happen, and define a distance which forbids (among other thing) such accumulation of time. We then prove that the pair $(\gamma, \cL)$ do not have atypical visits, and that these accumulated errors upon discretisation are negligible, uniformly over the discretisation $n\ge 1$.

\bigskip
\noindent \textbf{Organisation of the paper.} The paper is divided as follows.

\emph{\it Section~\ref{sec:construction}:} In Subsection~\ref{sub:sets}, we construct the space $\Conf_0$ on which the map $\Att$ will be defined and the subset $\Conf$ on which we will prove the map $\Att$ is continuous (as a function of the simple path $\gamma$ and loops $\mathcal{L}$). We need to support the case of the SLE$_2$ path and the Brownian loop soup, but also their discrete analogue, which brings two complications: first, we must manage cases for which several loops are rooted at the exact same location along the simple path (Figure~\ref{fig:issues}, middle), as well as cases for which a loop goes twice or more through their root (Figure~\ref{fig:issues}, right). In these cases, one may think of $\Att$ as being multivalued, which leads us to the construction of the set of {\it tie-breaks} in Subsection~\ref{sub:tie} which resolves these  indeterminacies. Secondly, the set of roots is no longer necessarily dense on $\Range(\gamma)$, and to maintain continuity we must introduce a parameter $\lambda\geq 0$ which describes the inverse pace at which $X$ follows the path $\gamma$ (our explanation so far correspond to $\lambda=0$,
which we take in the continuum limit while $\lambda=\lambda_n>0$ for lattice approximations). Then, we define by~\eqref{eq:defX} the map $\Att$ itself, and we provide conditions on $(\gamma,\mathcal{L})$ under which we prove in Subsection~\ref{sub:cont} that the path $\Att(\gamma,\mathcal{L})$ is continuous in time.

%

\emph{\it Section~\ref{sec:attaching}:} We define one-sided intersections and rerooting operations on couples $(\gamma, \ell)$, where $\ell$ is a loop which intersects $\gamma$, and we prove that these operations
are continuous at $(\gamma,\ell)$, provided that $\ell$ visits $\gamma(\sigma_\ell)$ at a single time $\theta_\ell$ (recall the notation $\sigma_\ell$ from Theorem~\ref{th:main_precise}) and that this visit is two-sided (Proposition~\ref{prop:sigmaxcont} and Corollary~\ref{coro:thetacont}).

\emph{\it Section~\ref{sec:continuity}:} This section contains the core deterministic arguments. We introduce the distance $d_{\Conf_0}$ on $\Conf_0$ in Subsection~\ref{sub:d}, and ultimately prove the continuity  on $\Conf$ of the map $\Att$ in Subsection~\ref{sub:cont2}. To this end, we prove in Subsection~\ref{sub:strong} that the topology induced by
$d_{\Conf_0}$ is equivalent (in the vicinity of $\Conf$) to an {\it a priori} stronger topology (Lemma~\ref{le:strongsuited}).
Conversely, we prove in Subsection~\ref{sub:weak} that, on some subsets of $\Conf_0$ defined by the property that some function-valued functional of $(\gamma,\mathcal{L})$ is smaller than a fixed function, the convergence for $d_{\Conf_0}$ toward an element in $\Conf$ is equivalent to the convergence for {\it a priori} weaker topologies. The key technical element that ensures, in the end, the continuity of the map $\Att$, is given by Lemma~\ref{le:omega1bound}, which roughly states that for $(\gamma',\mathcal{L}')$ sufficiently close from  $(\gamma,\mathcal{L})$, the set of hitting times by $\gamma'$ of sufficiently large loops in $\mathcal{L}'$ is fairly dense on $[0,t_{\gamma'}]$, which prevents $\Att(\gamma',\mathcal{L}')$ from jumping too quickly along $\gamma$ without meeting any large loop.

\emph{\it Section~\ref{sec:probabilistic1}:} We prove that for any simple path $\gamma$ with dimension less than $2$, $\mu^{\mathcal{L}}$-almost surely, $(\gamma, \mathcal{L})$ lies in the smaller space $\Conf$. We also show that for $\gamma_n$ a loop-erased random walk and $\mathbb{L}_n$ a random walk loop soup on a lattice approximation of $D$, $\Att(\gamma_n, \mathbb{L}_n)$
is distributed as a simple random walk.\footnote{This is essentially an already-known result, but we could not find a clear statement in the litterature.}

\emph{\it Section~\ref{sec:probabilistic2}:} We prove that the pair $(\gamma_n, \mathbb{L}_n)$ converges, as $n\to \infty$ and for the distance $d_{\Conf_0}$, toward $(\gamma,\mathcal{L})$ the pair of the radial SLE$_2$ trace and the Brownian loop soup, indeed concluding the proof of Theorem~\ref{th:main_precise}.
The main difficulty is to prove that the total duration of all the loops in $\mathbb{L}_n$ which intersect $\gamma_n$ is a close approximation of the total duration of all the loops in $\mathcal{L}$ which intersect $\gamma$. This is proved in Theorem~\ref{th:tCV} by using again the deterministic Lemma~\ref{le:strongsuited}, combined with the technical probabilistic estimation~\eqref{eq:temp:Ebound}, which roughly states that the total time spend by $\mathbb{L}_n$ on loop with size less than $\delta$, is (in expectation) bounded by the sum of a function independent from $\delta$ which goes to
$0$ with the mesh size, and a  function independent from the mesh size which goes to $0$ with $\delta$.

%
%

\section{Construction of the path \texorpdfstring{$\Att(\gamma,\mathcal{L})$}{Xi(gamma,L)}}
\label{sec:construction}

\subsection{Preliminary notations}
For sets $X,Y$, let $\mathcal{F}(X,Y)\coloneqq Y^X$ the set of functions from $X$ to $Y$, $\mathcal{C}(X, Y) $ the subset of continuous function, and $ \mathcal{C}( [0,\cdot],Y)\coloneqq  \bigcup_{T\geq 0} \mathcal{C}( [0,T], Y)$. For $f\in \mathcal{F}( [0,\cdot],Y)\coloneqq \bigcup_{T\geq 0} \mathcal{F}( [0,T], Y)$, we let $t_f$ be the unique non-negative real number such that the domain of $f$ is $[0,t_f]$.
We define the set of \textbf{simple paths }
\[ \mathscr{S}\coloneqq
\{\gamma\in \mathcal{C}([0,\cdot],\mathbb{R}^2):  \ \gamma_{|[0,t_\gamma)} \mbox{ is injective}\}.
\]
We also let $\mathscr{L}$ be the set of \textbf{loops}
\[
\mathscr{L} \coloneqq
\{\ell\in \mathcal{C}([0,\cdot],\mathbb{R}^2):  \ell(0)=\ell(t_\ell)\}.
\]
Besides the distance $d_\infty$ defined sooner, we will use the distance $\rho$ on $ \mathcal{F}( [0,\cdot], \mathbb{R}^d)$ given by
\[
\rho (\gamma_0, \gamma_1)\coloneqq \inf_\varphi \{\|{\gamma}_0-{\gamma}_1\circ \varphi\|_{\infty,[0,t_{\gamma_0}] }+\frac{\|\varphi-\operatorname{Id}\|_{\infty,[0,t_{\gamma_0}] }+\|\varphi^{-1}-\operatorname{Id}\|_{\infty,[0,t_{\gamma_1}] }}{2} \},
\]
where $\varphi$ ranges over the set of continuous increasing bijections from $[0,t_{\gamma_0}]$  to $[0,t_{\gamma_1}]$.
We will prove (Lemma~\ref{le:sametopo}) that these two distances induce the same topology on $ \mathcal{C}( [0,\cdot], \mathbb{R}^d)$.
For $d,d'$ two distances on resp. $X,X'$, we write $d\times d'$ for the distance on $X\times X'$ given by $d\times d'((x,x'),(y,y'))\coloneqq d(x,y)+d'(x',y')$.

The $ \delta$-continuity modulus of a function $f$ is written
\[\omega_ \delta(f)\coloneqq \sup \{f(t)-f(s): s,t\in[0,t_f], |t-s|\leq\delta\}.\] If $f\in \mathscr{L}$, the $t_f$-periodic extension $\bar{f}$ of $f$ is uniformly continuous and we let $\omega^l_{\epsilon}(f)\coloneqq\omega_\epsilon(\bar{f})$, which is invariant by the operation of changing the root of $f$. Remark $ \omega_{\epsilon}(f)\leq \omega^l_{\epsilon}(f)\leq 2 \omega_{\epsilon}(f)$.
For a collection $\mathcal{L}$ of loops and $\epsilon>0$, we let $\omega_\epsilon(\mathcal{L})\coloneqq \sup_{\ell \in \mathcal{L}} \omega_\epsilon(\ell)\in [0,\infty]$ the maximum of the continuity moduli, and $\omega^l_{\epsilon}(\mathcal{L})=\sup_{\ell \in \mathcal{L}} \omega^l_\epsilon(\ell)$.

\smallskip

We let $\mathscr{L}^{lf}$ be the set of locally finite counting measures on $\mathscr{L}$, which we think of as multi-subsets of $\mathscr{L}$. We will use set-theoretic notations for multisets. Given a simple path $\gamma\in \mathscr{S}$ and a multiset of loops $\mathcal{L} \in \mathscr{L}^{lf}$, we define
\[
 \mathscr{L}_{\gamma}\coloneqq \{\ell \in \mathscr{L}:\Range(\ell)\cap\Range(\gamma)\neq \emptyset \},\qquad
 \mathcal{L}_{\gamma}\coloneqq \mathcal{L}\cap \mathscr{L}_\gamma,\qquad
t_{\gamma,\mathcal{L}} \coloneqq \sum_{\ell \in \mathcal{L}_{\gamma} } t_\ell.\]
In particular $\mathscr{L}_\gamma$ is the set of loops which intersect $\gamma$.
For a simple path $\gamma\in \mathscr{S}$ and $\ell\in \mathscr{L}_\gamma$, we let
\[ \sigma_{\gamma,\ell}\coloneqq \min \{ \sigma: \gamma(\sigma)\in \Range(\ell)\}, \quad
x_{\gamma,\ell}\coloneqq \gamma(\sigma_{\gamma,\ell}), \quad
 \Theta_{\gamma,\ell}\coloneqq \{ t\in [0,t_\ell]: \ell_t=x_\ell \},\]
which we will refer to as the {\bf first hitting time}  of $\ell$, its { \bf spatial root}, and its set of {\ temporal roots}.
We also set \[
\sigma_{\gamma,\mathcal{L}}\coloneqq \{ \sigma_{\gamma,\ell}: \ell\in \mathcal{L}_\gamma\},\quad
\theta^-_{\gamma,\ell}\coloneqq \min  \Theta_{\gamma,\ell}, \quad \theta^+_{\gamma,\ell}\coloneqq \max  \Theta_{\gamma,\ell}.\]
In the case $\Theta_{\gamma,\ell}$ consists of a single element, or when a specific element of $\Theta_{\gamma,\ell}$ is fixed, we will call $\theta_{\gamma,\ell}$ this element and refer to it as the {\bf temporal root} of $\ell$. Then, the rerooted loop $\hat{\ell}$ (which implicitly depends on the choice of $\theta_{\gamma,\ell}\in\Theta_{\gamma,\ell}$) is defined by $t_{\hat{\ell}}=t_\ell$ and for $t\in [0, t_{\hat{\ell}}]$,
\[\hat{\ell}(t)\coloneqq \ell ( t+\theta_{\gamma,\ell} \mod t_\ell   ).\]
We write $\prec_\gamma$ and $\preceq_\gamma$ for the strict partial order and the partial order given on $\mathscr{L}_{\gamma}$ by
\[ \ell\prec_\gamma \ell' \iff \sigma_{\gamma,\ell} < \sigma_{\gamma,\ell'}, \qquad
\ell\preceq_\gamma \ell' \iff \sigma_{\gamma,\ell} \leq \sigma_{\gamma,\ell'}
.\]
Later on, we will fix on $\mathcal{L}_\gamma$ a finer total order $\prec_{\gamma,B}$, which depends on the extra parameter $B$. Then,
for $\mathcal{L}\in  \mathscr{L}^{lf}$ and $\ell\in \mathcal{L}_\gamma$,
we let
\[
 \mathscr{L}_{\leq \epsilon}\coloneqq \{\ell \in \mathscr{L}: t_\ell\leq \epsilon\},\qquad
  \mathcal{L}_{\leq \epsilon}\coloneqq \mathcal{L}\cap \mathscr{L}_{\leq \epsilon} ,\qquad
 \mathcal{L}_{\prec \ell}\coloneqq \{\ell' \in \mathcal{L}: \ell'\prec_{\gamma,B}\ell\}.\]
We similarly define the spaces $\mathscr{L}_{< \epsilon},\ \mathscr{L}_{\geq \epsilon},\ \dots$ and the multisets $\mathcal{L}_{<\epsilon },\ \mathcal{L}_{\geq \epsilon },\  \mathcal{L}_{\preceq \ell},\dots$.
From $ \mathscr{L}_{\gamma}$, we will extract a subset  $\mathscr{L}_{reg,\gamma}$ which roughly speaking consists on the set of curves $\ell$ such that $\Theta_{\gamma,\ell}$ consists of a single element $\theta_{\gamma,\ell}$, and such that the intersection at this temporal root is {\textbf{bilateral}}, in the sense that for all $\epsilon>0$, the loop $\ell$ reaches both sides of $\gamma$ within the time interval
$[\theta_{\gamma,\ell}-\epsilon,\theta_{\gamma,\ell}+\epsilon]$. The precise definition will be given later. We combine the indices for intersections, e.g. $\mathcal{L}_{reg,\gamma,\leq \epsilon}\coloneqq \mathcal{L}_{reg,\gamma}\cap \mathcal{L}_{\leq \epsilon}$.

Often, $\gamma$ is fixed and it is then removed from the notations (e.g. $\theta_\ell=\theta_{\gamma,\ell}$).

 \subsection{Summary of the construction}

We shall define two sets $\Conf\subseteq \Conf_0$ of pairs $(\gamma,\mathcal{L})$, where $\gamma$ is a continuous simple path and $\mathcal{L}$ is a collection of loops, from which we wish to construct a continuous curve $X=\Att(\gamma,\mathcal{L})$ to be interpreted as {\it the curve obtained from attaching to $\gamma$ the loops in $\mathcal{L}$, as $\gamma$ visit them for the first time} (the loops that are not visited by $\gamma$ are discarded). For general elements of $\Conf_0$, there are some indeterminacy related e.g. to the fact that two different loops may have the same {\it first visit time by $\gamma$}.
Thus, for $(\gamma,\mathcal{L})\in \Conf_0$, we build a set $\mathcal{B}_{\gamma,\mathcal{L}}$ which describes this indeterminacy: an element $B$ of  $\mathcal{B}_{\gamma,\mathcal{L}}$ refines the order $\prec_\gamma$ into a total order $\prec_{\gamma,B}$, and selects for each $\ell\in \mathcal{L}_\gamma$ a temporal root $\theta_{\gamma,\ell}\in\Theta_{\gamma,\ell}$.
For $(\gamma,\mathcal{L})$ in the smaller set $\Conf$, this set $\mathcal{B}_{\gamma,\mathcal{L}}$  is reduced to a single element.
We further introduce a non-negative parameter $\lambda$ which describes the inverse pace at which $X$ follows $\gamma$. We are interested in $\lambda=0$ when we attach the Brownian loop soup along the $SLE_2$ path, but it is necessary to extend the construction to $\lambda>0$ in order to consider discrete approximations and continuity properties.
Thus, for $(\gamma,\mathcal{L},\lambda)\in \Conf_0\times[0,\infty)$  and $B\in \mathcal{B}_{\gamma,\mathcal{L}}$, we construct a curve
\[ X=\Att(\gamma,\mathcal{L},\lambda,B),\]
which we prove to be continuous as soon as $(\gamma,\mathcal{L})\in \Conf$ or $\lambda>0$. For $(\gamma,\mathcal{L})\in \Conf$, we let $\Att(\gamma,\mathcal{L})=\Att(\gamma, \mathcal{L},0,B)$ where $B$ is the unique element of $\mathcal{B}_{\gamma,\mathcal{L}}$.

This curve $X\coloneqq\Att(\gamma,\mathcal{L},\lambda,B)$ satisfies the following property, which describes it uniquely as a continuous function when  $(\gamma,\mathcal{L})\in \Conf$ and $\lambda=0$:
 \[\mbox{For all $\ell\in \mathcal{L}_\gamma$ and $t\in[0,t_\ell]$, }\quad  X(T_\ell+t )=\hat{\ell}(t), \quad\mbox{where}\quad T_\ell=T_\ell(\mathcal{X},B,\lambda)\coloneqq
\sum_{\substack{ \ell' \in \mathcal{L}_{\gamma},\\\ell'\prec_{\gamma,B} \ell} } t_{\ell'}
   +\lambda \sigma_{\gamma,\ell}.\]
   Here, 
	 the rerooted loop $\hat{\ell}$ is obtained thanks to the temporal root $\theta_{\gamma,\ell}\in \Theta_{\gamma,\ell}$ determined by $B$.
In words, $T_\ell$ is the time it takes $\Att$ to reach the loop $\ell$, and from the time $T_\ell$, $X$ follows the loop $\ell$ starting from $\theta_\ell$ the temporal root of $\ell$ associated with $\gamma$.

\smallskip

We first construct the sets $\Conf$ and $\Conf_0$, as well as the application $\Att$, and prove that $X=\Att(\gamma,\mathcal{L},\lambda,B)$ is continuous in time as soon as $\lambda>0$ or $(\gamma,\mathcal{L})\in \Conf$.

Then, we introduce a distance on $\mathcal{R}_0$ which is sufficiently strong to prove the continuity of the map $(\gamma,\mathcal{L},\lambda,B)\mapsto \Att(\gamma,\mathcal{L},\lambda,B)$ in the vicinity of $\Conf\times\{0\}$, for {\it any} choice of $B$. The last condition imposed on the smaller set $\Conf$ is key to prove this continuity property. We further prove that convergence in $\mathcal{R}_0$ toward a point in $\Conf$ can be reduced to convergence for a weaker topology together with boundedness of some functionals, which shall be easier conditions to verify later on.

Finally, we consider the case when $(\gamma,\mathcal{L})\sim \mu^\gamma\otimes \mu^D$, i.e. when $\gamma$ is a radial $SLE_2$ from $0$ to $\partial D$ and $\mathcal{L}$ is a Brownian loop soup in $D$ independent from $\gamma$. We shall prove that almost surely $(\gamma,\mathcal{L})\in \Conf$, and that the discrete analogous $(\gamma_n,\mathbb{L}_n)$ on the lattice, properly rescaled, converges in distribution toward $(\gamma,\mathcal{L})$ for the topology we introduced. For a uniform choice of $B_n\in \mathcal{B}_{\gamma_n,\mathbb{L}_n}$,
we shall see that $\Att(\gamma_n,\mathbb{L}_n,1,B_n)$ is distributed as a simple random walk, from which we will deduce that $\Att(\gamma,\mathcal{L})$ is distributed as a Brownian motion.

\subsection{The sets \texorpdfstring{$\Conf,\Conf_0$}{R,R0}}
\label{sub:sets}
We define $\Conf$ the subset of $\mathscr{S}\times \mathscr{L}^{lf}$ which consists of pairs $\mathcal{X}=(\gamma,\mathcal{L})$ such that:
\begin{enumerate}[(i)]
\item \label{it:tmax} The time \[ t_{\mathcal{X},\lambda}\coloneqq
\sum_{\ell \in \mathcal{L}_\gamma} t_\ell +\lambda t_\gamma
\] is finite, and the set $\mathcal{L}_{\geq \epsilon}$ is finite for all $\epsilon>0$.
\item \label{it:unifcont} The set $\mathcal{L}$ is equicontinuous, in the sense that for all $\epsilon>0$, there exists $\delta>0$ such that $\omega_\delta(\mathcal{L})<\epsilon$, i.e.
\[
\forall \ell\in \mathcal{L}, \forall s,t\in [0,t_\ell] \mbox{ such that } |s-t|<\delta, \mbox{ it holds } |\ell_s-\ell_t|< \epsilon.
\]
\item \label{it:inj} The application $\ell\mapsto \sigma_\ell$ defined on $\mathcal{L}_\gamma$ is injective. For any $\ell\in \mathcal{L}_\gamma$, the set of roots $\Theta_\ell$ is reduced to a single element $\theta_\ell$.
\item \label{it:dense} The set $\sigma_{\gamma,  \mathcal{L}}$ is dense in $ [0,t_\gamma]$.
\item \label{it:reg} For all $\ell \in \mathcal{L}_\gamma$, it holds $\ell\in \mathscr{L}_{reg,\gamma}$ the sets defined below.
\end{enumerate}
Set further $\Conf_0\subset \mathscr{S}\times \mathscr{L}^{lf}$ the subset of pairs which satisfy only the conditions~\eqref{it:tmax} and~\eqref{it:unifcont}, and let
$\Conf'=(\Conf\times \{0\})\cup (\Conf_0\times(0,\infty))$.
We will generically let $\mathcal{X},\mathcal{X}'$ be elements of $\Conf_0$. Then, we will use the notation $\gamma$ (resp. $\gamma'$) and $\mathcal{L}$ (resp. $\mathcal{L}'$) for the elements such that $\mathcal{X}=(\gamma,\mathcal{L})$ (resp. $\mathcal{X}'=(\gamma',\mathcal{L}')$).

\subsection{Tie-breaks}
\label{sub:tie}
Let $\mathcal{X}\in\Conf_0$. For $\sigma\in \sigma_{\gamma,\mathcal{L}}$, define $\mathcal{O}_\sigma$ be the set of total orders on the set $\{ \ell \in \mathcal{L}_\gamma : \sigma_\ell=\sigma \}$ of loops rooted at $\sigma$, and define $\mathcal{B}_{\mathcal{X}}$ the set of tie-breaks as
\[
\mathcal{B}_{\mathcal{X}}\coloneqq \Big(\prod_{\sigma \in \sigma_\mathcal{L} }  \mathcal{O}_\sigma \Big) \times \Big( \prod_{\ell \in \mathcal{L}_{\gamma}} \Theta_\ell\Big).
\]
Given $(\mathcal{X},\lambda)\in \Conf'$ and $B\in \mathcal{B}_{\mathcal{X}}$, we will construct $\Att(\mathcal{X},\lambda,B)\in \mathcal{C}([0,t_{\mathcal{X},\lambda} ], \mathbb{R}^2)$.
Elements of
$\mathcal{B}_{\mathcal{X}}$ should be interpreted as a way to decide, whenever there are several choices, which loop to follow first, and where to reroot loops from.
We will say that there is no tie when $\mathcal{B}_{\mathcal{X}}$ is reduced to a single element, which can easily be seen to be equivalent to the condition~\eqref{it:inj}.

An element $B=( (\preceq_\sigma)_{\sigma\in \sigma_{\gamma,\mathcal{L}}}, (\theta_\ell)_{\ell \in \mathcal{L}_{\geq \gamma} }  )$  of $\mathcal{B}_{\mathcal{X}}$, induces naturally a
`lexicographic' total order~$\preceq_{\gamma,B}$ on $\mathcal{L}_\gamma$ which is finer than the initial order $\preceq_\gamma$ and given by
\[
\ell\preceq_{\gamma,B} \ell' \iff ( \sigma_{\gamma,\ell} < \sigma_{\gamma',\ell'} \mbox{ or }  (\sigma_\ell = \sigma_{\ell'} \mbox{ and } \ell \preceq_{\sigma_\ell} \ell' )).
\]
Furthermore, $B$ determines for all $\ell\in \mathcal{L}_\gamma$ a unique $\theta_\ell\in \Theta_\ell$, so that the rerooted loop $\hat{\ell}$ is uniquely defined.

\subsection{The set \texorpdfstring{$\mathcal{T}$}{T} and the function \texorpdfstring{$\sigma_{\mathcal{X},\lambda}$}{sigma} }
Let $\mathcal{X}\in \Conf_0$, $\lambda\geq 0$ and $B\in \mathcal{B}_{\mathcal{X}}$.
Recall $ t_{\mathcal{X},\lambda}= \sum_{\ell\in \mathcal{L}_\gamma} t_\ell+ \lambda t_\gamma $
and
$T_\ell=T_\ell(\mathcal{X},B,\lambda) \coloneqq \sum_{\ell'\prec_{\gamma,B} \ell } t_\ell+\lambda \sigma_{\gamma,\ell} $,
which we shall interpret respectively as the total duration of $\Att(\mathcal{X},\lambda,B)$ and the time it takes for $\Att(\mathcal{X},\lambda,B)$ to reach the loop $\ell$.
We define the set
\[  \mathcal{T} \coloneqq \bigsqcup_{\ell \in \mathcal{L}_\gamma} (T_\ell, T_\ell+t_\ell)\subset [0,t_{\mathcal{X},\lambda} ]. \]
The set $\mathcal{T}$ is the set of time when $\Att(\mathcal{X},\lambda,B)$ follows one of the loops in $\mathcal{L}$, and $  [0,t_{\mathcal{X},\lambda} ]\setminus \mathcal{T} $ is the set of time when $\Att(\mathcal{X},\lambda,B)$ follows $\gamma$.
When $\lambda=0$, one can also think of $[0,t_{\mathcal{X},\lambda} ]\setminus\mathcal{T}$ as a set of exceptional times: it is the topological closure of $\{T_\ell\}_{\ell \in \mathcal{L}_\gamma}$, and $\mathcal{T}$ is dense (Lemma~\ref{le:dense}) on $[0,t_{\mathcal{X},\lambda} ]$ with full Lebesgue measure (which we won't prove nor use in the proof). This is not the case for $\lambda>0$, for then
$\Att(\mathcal{X},\lambda,B)$ spends a positive amount of time on $\gamma$.

We now define a non-decreasing function
$
\sigma_{\mathcal{X},\lambda}: [0,t_{\mathcal{X},\lambda}]\to [0,t_\gamma]$, such that $\sigma_{\mathcal{X},\lambda}(t)$ describes how far $\Att(\mathcal{X},\lambda,B)_{[0,t]}$ went along $\gamma$
(more precisely, it is such that $\Att(\mathcal{X},\lambda,B)$ will not visit $\gamma_{[0, \sigma_{\mathcal{X},\lambda}(t) ]}$ after the time $t$).\footnote{Be careful that $\sigma_{\mathcal{X},\lambda}$ is a function, while $\sigma_{\gamma,\mathcal{L}}$ is the set of first hitting times.}
Once $\sigma_{\mathcal{X},\lambda}$ will be defined, the path $X=\Att(\mathcal{X},\lambda,B)$ is defined by
\begin{equation}
\label{eq:defX}
X_t\coloneqq \left\{
\begin{array}{cc}
\hat{\ell}(t-T_\ell ) & \mbox{if } t\in (T_\ell,T_\ell+t_\ell) \subset \mathcal{T},\\
\gamma( \sigma_{\mathcal{X},\lambda}(t)  ) & \mbox{if } t\in [0,t_{\mathcal{X},\lambda} ]\setminus \mathcal{T}.
\end{array}
\right.
\end{equation}

We now explain how to define $\sigma_{\mathcal{X},\lambda}$ (on {\it all} of $[0, t_{\mathcal{X}, \lambda}]$). For $t\in \mathcal{T}$, there exists a unique loop $\ell\in \mathcal{L}_\gamma$ such that $t\in (T_\ell,T_{\ell}+t_\ell)$, and we then simply define
 \[ \sigma_{\mathcal{X},\lambda}(t)\coloneqq \sigma_{\gamma,\ell}.\]

For $t\notin  \mathcal{T}$, we proceed by affine interpolation. Consider first the simple case when there exist loops $\ell,\ell'\in \mathcal{L}_\gamma$ such that
$T_\ell+t_\ell\leq t \leq T_{\ell'}$, $T_\ell+t_\ell<T_{\ell'}$, and there is no loop in between, i.e. there is no $\ell''\in \mathcal{L}_\gamma $ such that $\ell \prec \ell''\prec \ell'$. In such a case,
the limit from the left $\sigma_{\mathcal{X},\lambda}(T_\ell+t_\ell -)$ is equal to $\sigma_\ell$, and the limit from the right $\sigma_{\mathcal{X},\lambda}(T_{\ell'}+)$ is equal to $\sigma_{\ell'}$. We thus interpolate with
$\sigma_{\mathcal{X},\lambda}(t)\coloneqq  (1-\mu) \sigma_{\ell}+  \mu  \sigma_{\ell'}   $, where
$\mu\in[0,1]$ is the unique value such that $t= (1-\mu) (T_\ell+t_\ell)+  \mu T_{\ell'}$.

Let us now consider the general situation (for $t\in [0,t_{\mathcal{X},\lambda}]\setminus \mathcal{T}$). First, set
\begin{equation}
\label{eq:def:tpm}
t_-\coloneqq 0\vee \sup \{ T_\ell+t_\ell: \ell \in  \mathcal{L}_\gamma, T_\ell+t_\ell\leq t \},
\qquad t_+\coloneqq t_{\mathcal{X},\lambda}  \wedge \inf\{ T_\ell: \ell \in  \mathcal{L}_\gamma, T_\ell\geq t\},
\end{equation}
\[ \sigma_-\coloneqq 0\vee \sup \{ \sigma_\ell: \ell \in  \mathcal{L}_\gamma, T_\ell+t_\ell\leq t \},
\qquad \sigma_+\coloneqq t_\gamma \wedge \inf\{ \sigma_\ell: \ell \in  \mathcal{L}_\gamma, T_\ell\geq t\}
.\]
Then, $t_-\leq t\leq t_+$.
In the case $t_-=t_+$, let $\mu(t)\coloneqq 0$. Otherwise, let $\mu(t)\coloneqq (t-t_-)/(t_+-t_-)\in [0,1]$, i.e. $t= (1-\mu)t_-+\mu t_+$, and set, on $ [0,t_{\mathcal{X},\lambda}]\setminus \mathcal{T}$,
\[\sigma_{\mathcal{X},\lambda}\coloneqq (1-\mu)\sigma_-+\mu \sigma_+.\]
In particular, these functions satisfy $\sigma_-\leq \sigma_{\mathcal{X},\lambda}\leq \sigma_+$.

\begin{remark}
  Clearly the function $\sigma_{\mathcal{X},\lambda}$ does not depend on the choice of the roots in $\prod_{\ell\in \mathcal{L}_\gamma} \Theta_\ell$. Although it is less immediate, it also does not depend on the choice of the orders (thus it does not depend from the choice of $B\in \mathcal{B}_{\mathcal{X}}$ at all). Indeed, the values of $\sigma_-$, $\sigma_+$ do not depend on these orders, and although there are some values of $t$ for which $t_-$ and $t_+$ do depend on the choice of $B$, this is the case only for $t$ such that $\sigma_-=\sigma_+$, in which case $\sigma_{\mathcal{X},\lambda}$ is just equal to this common value anyway.
\end{remark}

For later use, we now prove the following. (Intuitively, this tells us among other things that the resulting path $X$ does not depend on the initial parametrisation of $\gamma$). 

\begin{lemma}
  \label{le:dense}
  For all $\lambda=0$ and $\mathcal{X}\in \Conf_0$, the set $\mathcal{T}$ is dense on $[0,t_{\mathcal{X},0}]$.
\end{lemma}
\begin{proof}
  Let $\mathcal{T'}= \bigcup_{\ell\in \mathcal{L}_\gamma} [T_\ell,T_\ell+t_\ell]$, thus $\mathcal{T}$ is dense in $\mathcal{T}'$.

  The key point is that, for $\lambda=0$, for all $s\in [0,t_{\mathcal{X},0}] \setminus \mathcal{T}'$, it holds $s_-=s_+$ (for $\lambda\neq 0$, they would differ by $\lambda(\sigma_+-\sigma_-)$). Indeed, for such $s$,
  \begin{align*} s_+=\inf\{ \sum_{\ell'\prec \ell } t_{\ell'}: T_\ell\geq s\}
  &=\sum \{ t_{\ell'}: \forall \ell, T_\ell\geq s \implies \ell'\prec \ell\}\\
  &=\sum \{ t_{\ell'}: \forall \ell, T_\ell\geq s \implies T_{\ell'}< T_{\ell}\}\\
  &=\sum \{ t_{\ell'}:  T_{\ell'}< s \},
  \end{align*}
  where the second equality on the first line is by monotone convergence theorem.
  On the other hand,
  \begin{align*}
  s_-=\sup\{ \sum_{\ell'\preceq \ell } t_{\ell'}: T_\ell+t_\ell \geq s\}
  &=\sum \{ t_{\ell'}: \forall \ell, T_\ell+t_\ell \geq s \implies \ell'\preceq \ell\}\\
  &=\sum \{ t_{\ell'}:  T_{\ell'}+t_{\ell'}  \geq s \}\\
  &=\sum \{ t_{\ell'}:  T_{\ell'}< s \},
  \end{align*}
  where the second equality is again by monotone convergence theorem, and the last equality is obtained by using the fact $s\notin \mathcal{T}'$). Thus, $s_-=s_+$ as claimed.

  Let now  $s<t\in [0,t_{\mathcal{X},0}] \setminus \mathcal{T}'$. We will prove there exists $u\in [s,t]\cap \mathcal{T}$.
  Since $s=s_+$, either there exists $\ell$ such that $T_\ell=s$, in which case we can take $u\coloneqq (t+s)/2 \wedge (T_\ell+t_\ell/2)$, or there exists a sequence $(\ell_n)$ such that $T_{\ell_n}$ decreases toward $s$. Then $T_{\ell_n}+t_{\ell_n}\leq T_{\ell_{n-1}}$, thus $(T_{\ell_n}+t_{\ell_n} )$ also decreases toward $s$. Thus, for any $n$ large enough, $T_{\ell_n}+t_{\ell_n} <t$, and we can take $u\coloneqq (T_{\ell_n}+t_{\ell_n}/2)$, which concludes the proof.
\end{proof}

\subsection{Continuity in time of \texorpdfstring{$\Att(\mathcal{X},\lambda,B)$}{Xi(gamma,L,lambda,B)}}
\label{sub:cont}
The end of the section is dedicated to the proof that, for $(\mathcal{X},\lambda)\in \Conf'$ and $B\in\mathcal{B}_\mathcal{X}$, the function $X=\Att(\mathcal{X},\lambda,B)$ is continuous. During the process, we will record two quantitative bounds, in particular the key estimation \eqref{eq:def:omegadelta1} in the next lemma, which will turn useful when we will consider the harder problem of the continuity of the map $(\mathcal{X},\lambda,B)\mapsto \Att(\mathcal{X},\lambda,B)$.
\begin{lemma}
\label{le:sigmaIsCont}
  Let $\mathcal{X}\in \Conf_0$ and $\lambda\geq 0$.
  If either $\lambda>0$ or $\mathcal{X}$  satisfies~Condition \eqref{it:dense}, the application $\sigma_{\mathcal{X},\lambda}$ is continuous on $[0,t_{\mathcal{X},\lambda}]$.

  Furthermore, we have the following quantitative bound, independent of $\lambda$, on the continuity moduli: for all $\delta>0$, let $(\sigma_i)_{i\in \{1,\dots, k\}}$ be an increasing enumeration of the finite set $\{ \sigma_{\ell}: \ell\in \mathcal{L}_{\gamma,\geq \delta} \}$, and let $\sigma_0=0$, $\sigma_{k+1}=t_\gamma$. Then,
  \begin{equation}
  \label{eq:def:omegadelta1}
  \omega_\delta(\sigma_{\mathcal{X},\lambda})\leq \omega^1_{\delta}(\mathcal{X})\coloneqq  \max \{ \sigma_{i+1}-\sigma_i : i\in \{0,\dots, k\} \}.
  \end{equation}
\end{lemma}
\begin{proof}

  We first prove~\eqref{eq:def:omegadelta1}.
  Let $s<t$. Assume $(\sigma_{\mathcal{X},\lambda}(s),\sigma_{\mathcal{X},\lambda}(t))\cap \{ \sigma_i : i\in \{0,\dots, k+1\} \}\neq \emptyset$, and let $\ell \in \mathcal{L}_{\gamma,\geq \delta}  $ such that $\sigma_\ell \in (\sigma_{\mathcal{X},\lambda}(s),\sigma_{\mathcal{X},\lambda}(t))$. For all $u\in (T_\ell,T_\ell+t_\ell)$, $\sigma_{\mathcal{X},\lambda}(s)< \sigma_{\mathcal{X},\lambda}(u)=\sigma_{\ell}< \sigma_{\mathcal{X},\lambda}(t)$. Since $\sigma_{\mathcal{X},\lambda}$ is non-decreasing, we deduce $s\leq u\leq t$.
  It follows $s\leq T_\ell$ and $t\leq T_\ell+t_\ell$, thus $t-s\geq t_\ell\geq \delta$. By contraposition,
  \[ t-s<\delta\implies
  (\sigma_{\mathcal{X},\lambda}(s),\sigma_{\mathcal{X},\lambda}(t))\cap \{ \sigma_i : i\in \{0,\dots, k+1\}\}
   = \emptyset \implies \sigma_{\mathcal{X},\lambda}(t)-\sigma_{\mathcal{X},\lambda}(s) \leq \omega^1_{\delta}(\mathcal{X}),\]
  which proves~\eqref{eq:def:omegadelta1}.

  For the continuity of $\sigma_{\mathcal{X},\lambda}$, we consider separately the case when $\mathcal{X}$ satisfies the condition~\eqref{it:dense} and the case $\lambda>0$.

  The former case follows directly from~\eqref{eq:def:omegadelta1}. Indeed, as $\delta\to 0$, $\omega^1_{\delta}(\mathcal{X})$ converges toward \[ \sup \{ \sigma'-\sigma :\sigma'\geq \sigma, (\sigma,\sigma')\cap \{ \sigma_{\gamma,\ell}: \ell \in \mathcal{L} \}  =\emptyset \},\]
  which is $0$ by~\eqref{it:dense}. Thus $\omega_\delta(\sigma)\leq \omega^1_{\delta}(\mathcal{X}) \underset{\delta\to 0}\longrightarrow 0$, hence $\sigma$ is continuous.

  Consider now the latter case $\lambda >0$. If the set $\sigma_\mathcal{L}$ is finite, then we can easily show that the function $\sigma_{\mathcal{X},\lambda}$ is Lipschitz-continuous. On the other end, if $\sigma_\mathcal{L}$ is dense, the argument above also allows to conclude that $\sigma_{\mathcal{X},\lambda}$ is continuous. However, when $\sigma_\mathcal{L}$ is neither discret nor dense, neither argument allows to conclude: we must `localise' the previous argument. For some $t$,  we need to use `the Lipschitz argument' on the left and `the density argument' on the right, or the other way around, so we also distinguish between left and right continuity. 

  \begin{itemize}
  \item The continuity of $\sigma_{\mathcal{X},\lambda}$ at any $t\in \mathcal{T}$ is obvious, since there is an open set $(T_\ell,T_\ell+t_\ell)\ni t$ on which $\sigma_{\mathcal{X},\lambda}$ is locally constant. Let now $t\in [0,t_{\mathcal{X},\lambda}]\setminus \mathcal{T}$.
  \item For $t\in [0,t_{\mathcal{X},\lambda} ]\setminus \mathcal{T}$ such that $t_-<t<t_+$, the restriction of $\sigma_{\mathcal{X},\lambda}$ to $(t_-,t_+)$ is affine, thus continuous. For the same reason, if $t<t_+$ then $\sigma_{\mathcal{X},\lambda}$ is right-continuous at $t$, and if $t>t_-$ then
  $\sigma_{\mathcal{X},\lambda}$ is left-continuous at $t$.
  \item It only remains to show that, for $t=t_+$, $\sigma_{\mathcal{X},\lambda}$ is right-continuous at $t$, and that  for $t=t_-$, $\sigma_{\mathcal{X},\lambda}$ is left-continuous at $t$.

  In the case $t=t_-$, $\mu(t)=0$ so that $\sigma_{\mathcal{X},\lambda}(t)=\sigma_-(t)$.

  In the case $t=t_+$, either $\mu(t)=1$  or $t_-=t_+$. In the former case $\mu(t)=1$, it holds $\sigma_{\mathcal{X},\lambda}(t)=\sigma_+(t)$.
  In the latter case $t_-=t_+$, there exists a non-decreasing sequence $(\ell^-_n)$ and a  non-increasing sequence $(\ell^+_n)$ such that $T_{\ell^+_n}$ and
  $T_{\ell^-_n}+t_{\ell_n^-}$  both converges toward $t=t_-=t_+$ as $n\to \infty$. However, the difference $T_{\ell^+_n}-(T_{\ell^-_n}+t_{\ell_n^-}) $ is bounded below by $\lambda\cdot (\sigma_{\ell^+_n}-\sigma_{\ell^-_n})\geq \lambda\cdot (\sigma_{+}(t)-\sigma_{-}(t))$. Since $\lambda>0$, we deduce $\sigma_-(t)=\sigma_+(t)$, hence $\sigma_{\mathcal{X},\lambda}(t)=\sigma_+(t)$.
  In both the case $\mu(t)=1$  and the case $t_-=t_+$, we have proved $\sigma_{\mathcal{X},\lambda}(t)=\sigma_+(t)$.

  In the rest of the proof, the cases $t=t_-$ and $t=t_+$ are now symmetrical, so we only consider the right-continuity for $t=t_+$.

  We make two subcases depending on whether the infimum defining $t_+$ is achieved of not. If it is achieved, i.e. there exists $\ell\in \mathcal{L}_\gamma$ such that $t=T_\ell$, then $\sigma_+(t)=\sigma_{\gamma,\ell}$ by definition. Thus
  $\sigma_{\mathcal{X},\lambda}(t)=\sigma_{\gamma,\ell}$. For all $s\in (t,t+t_\ell)$, by definition of $\sigma_{\mathcal{X},\lambda}$, $\mathcal{X},\sigma_{\mathcal{X},\lambda}(s)=\sigma_{\gamma,\ell}$. Thus, $\sigma_{\mathcal{X},\lambda}$ is locally constant on $[t,t+t_\ell)$, hence right-continuous at $t$.

  In the case when the infimum defining $t_+$ is not achieved, there exists a sequence $(\ell_n)_n$ in $\mathcal{L}_\gamma$ such that $T_{\ell_n}$ decreases toward $t_+=t$.
Since $\sigma_{\mathcal{X},\lambda}$ is non-decreasing and  $T_{\ell_n}$ decreases toward $t$, $\lim_{\substack{s\to t,\\ s>t}}  \sigma_{\mathcal{X},\lambda}(s)=\lim_{n\to \infty} \sigma_{\mathcal{X},\lambda}(T_{\ell_n})$.
By definition of $\sigma_+$ and monotonicity of $\ell\mapsto \sigma_{\gamma,\ell}$,  $\sigma_{\mathcal{X},\lambda}(T_{\ell_n})$ converges toward $\sigma_+(t)$. Since  $\sigma_+(t)=\sigma_{\mathcal{X},\lambda}(t)$,
$\lim_{\substack{s\to t,\\ s>t}}  \sigma_{\mathcal{X},\lambda}(s)=\sigma_{\mathcal{X},\lambda}(t)$, which concludes the proof.\qedhere
  \end{itemize}
\end{proof}

\begin{proposition}
  \label{prop:extension}
  Let $\mathcal{X}\in \Conf_0$ and $\lambda\geq 0$.
  If either $\lambda>0$ or $\mathcal{X}$  satisfies~\eqref{it:dense}, for whichever choice of tie break $B\in \mathcal{B}_{\mathcal{X}}$, for all $\epsilon>0$,
  \[ \omega_\epsilon(\Att(\mathcal{X},\lambda,B))\leq 2\, \omega^l_\epsilon(\mathcal{L})+
  \omega_\epsilon(\gamma\circ \sigma_{\cX,\lambda}).\]

  In particular, the function $X=\Att(\mathcal{X},\lambda,B)$ is continuous on $[0,t_{\mathcal{X},\lambda}]$,
  and the family of functions $\{\Att(\gamma,\mathcal{L},\lambda,B)\}_{B\in \mathcal{B}_{\gamma,\mathcal{L}}}$ is equicontinuous.
\end{proposition}
\begin{proof}
  Define $\pi_-,\pi_+:[0,t_{\mathcal{X},\lambda}]\to [0,t_{\mathcal{X},\lambda}]\setminus \mathcal{T}$ by $\pi_-t(t)=\pi_+(t)=t$ for $t\in [0,t_{\mathcal{X},\lambda} ]\setminus \mathcal{T}$, and for $t\in(T_\ell,T_\ell+t_\ell)$, $\pi_-(t)\coloneqq T_\ell$, $\pi_+(t)\coloneqq T_\ell+t_\ell$.
  In particular, for all $t$, $\pi_-(t)\leq t \leq \pi_+(t)$. Furthermore, for any $s,t\in [0,t_{\mathcal{X},\lambda}]$ with $s<t$, either there exists $\ell$ such that
  $T_\ell<s<t< T_\ell+t_\ell$, or $\pi_+(s)\leq \pi_-(t)$.

  Let $s,t \in [0,t_{\mathcal{X},\lambda}]$ such that $s<t<s+\epsilon$. In the case there exists $\ell$ such that $T_\ell<s<t< T_\ell+t_\ell$,
  \[ |X(t)-X(s)|=|\hat{\ell}(t-T_\ell)-\hat{\ell}(s-T_\ell)|\leq \omega_\epsilon(\hat{\ell})\leq \omega^l_\epsilon(\mathcal{L}).\]
  Otherwise, $s<\pi_+(s)\leq \pi_-(t)<t<s+\epsilon$. Since $\pi_-(t),\pi_+(s)\in  [0,t_{\mathcal{X},\lambda}]\setminus \mathcal{T}   $,
  $X( \pi_-(t))=\gamma\circ \sigma (\pi_-(t))$ and  $X( \pi_+(s))=\gamma\circ \sigma (\pi_+(s))$. Thus,
  \begin{align*}
  |X(t)-X(s)|&\leq |X(t)-X( \pi_-(t))|+|\gamma\circ \sigma( \pi_-(t))-\gamma\circ \sigma(\pi_+(s))|+|X(\pi_+(s))-X(s)|\\
  &\leq \omega^l_\epsilon(\mathcal{L})+
  \omega_\epsilon(\gamma\circ \sigma)+  \omega^l_\epsilon(\mathcal{L}).
  \end{align*}
  The equicontinuity condition and the continuity of $\gamma$ and $\sigma$ ensure $2 \omega^l_\epsilon(\mathcal{L})+
  \omega_\epsilon(\gamma\circ \sigma)\underset{\epsilon\to 0}\longrightarrow 0$,
  which concludes the proof.
\end{proof}

\section{Attaching a single loop to a simple path}
\label{sec:attaching}
In this section, we consider continuity of the applications $(\gamma,\ell)\mapsto x_{\gamma,\ell},\ \sigma_{\gamma,\ell},\ \theta^{-}_{\gamma,\ell},\ \theta^+_{\gamma,\ell}$. One can easily convince oneself that discontinuity points $(\gamma,\ell)$ are not generic, and we thus search for generic conditions on $(\gamma,\ell)$ ensuring the continuity at $(\gamma,\ell)$.
%

\subsection{The set of regular pairs}
\label{sec:regPairs}

For $\gamma\in \mathscr{S}$ and $x\in \gamma((0,t_\gamma))$, we first construct two connected open sets $L(x)$ and $R(x)$ such that the disjoint union $\Range(\gamma)\sqcup L(x)\sqcup R(x)$ contains a neighbourhood of $x$, and which we can think as the `left' and `right' sides of $\gamma$, around $x$.

Let $\gamma^{-1}:\Range(\gamma)\to [0,t_\gamma]$ be the inverse of $\gamma$, $o\coloneqq \gamma(0)$, $f\coloneqq \gamma(t_\gamma)$, and $\rhos(x)\coloneqq d(x,\{ o,f\} )$. Then, for any $x\in \Range(\gamma)\setminus\{o,f\} $, it holds $\rhos(x)>0$. Let
$\sigma_\gamma^-(x)$ and $\sigma_\gamma^+(x)$ be respectively the last time when $\gamma$ enters $B_{\rhos(x)}(x)$ before reaching $x$, and the first time when $\gamma$ exits $B_{\rhos(x)}(x)$ after reaching $x$:
\[\sigma_\gamma^-(x)\coloneqq  \sup\{ \sigma< \gamma^{-1}(x) : \gamma_\sigma \in \partial B_{\rhos(x)}(x) \}\quad\text{and}\quad\sigma_\gamma^+(x)\coloneqq  \inf\{ \sigma> \gamma^{-1}(x): \gamma_\sigma \in \partial B_{\rhos(x)}(x) \}.\]
In particular, $0\leq \sigma_\gamma^-(x)\leq \gamma^{-1}(x)\leq \sigma_\gamma^+(x)\leq t_\gamma$.

The restriction of $\gamma$ to the interval $(\sigma_\gamma^-(x), \sigma_\gamma^+(x))$ takes values on $B_{\rhos(x)}(x)$, whilst $\gamma(\sigma_\gamma^\pm(x))\in \partial B_{\rhos(x)}(x)$. Let $ \overleftarrow{\gamma}$ (which depends on $\gamma$ but also on $x$) be the arclength parametrisation of the arc going anticlockwise from
$\gamma(\sigma_\gamma^+(x))$ to $\gamma(\sigma_\gamma^-(x))$ along $\partial B_{\rhos(x)}(x)$. Then, the concatenation $\eta\coloneqq \gamma_{|[\sigma_\gamma^-(x), \sigma_\gamma^+(x)]}\cdot \overleftarrow{\gamma}$ is a simple loop.  By the Jordan curve theorem, $\mathbb{C}\setminus \Range(\eta)$ consists of exactly two connected components. We let $L(x)$ be the bounded one, and $R(x)$ be the unbounded one.
\begin{definition}
\label{eq:def:reg}
    The loop $\ell$ is said to be {\bf regular for $\gamma$} if it intersects $\gamma$ and the following holds.
    \begin{enumerate}[i)]
    \item \label{item:prop1} The set of temporal roots $\Theta_\ell$ is reduced to a single element $\theta_\ell$. Furthermore, $\sigma_\ell\notin \{0,t_\gamma\}$.
     Remark this prevents $\theta^-_\ell=0$, since then $\theta^+_\ell=t_\ell$.
    \item \label{item:prop2} The intersection at $\theta_\ell$ is bilateral: for any $\epsilon>0$, there exists $\theta^L,\theta^R\in [\theta_\ell-\epsilon,\theta_\ell+\epsilon]$ such that
    $\ell(\theta^L)\in L(x_\ell)$ and $\ell(\theta^R)\in R(x_\ell)$.
    \end{enumerate}
    The subset of $\mathscr{L}_\gamma$ of loops which are regular for $\gamma$ is denoted $\mathscr{L}_{reg,\gamma}$. We define
		\[ \mathscr{L}_{reg}\coloneqq \bigcup_{\gamma\in \mathscr{S}} (\{\gamma\} \times \mathscr{L}_{reg,\gamma}) \qquad \mbox{and} \qquad
		\mathscr{L}_0\coloneqq\bigcup_{\gamma\in \mathscr{S}} (\{\gamma\} \times \mathscr{L}_{\gamma}).\]
\end{definition}
The main goal of this subsection is to prove the following properties, 	where	$\mathscr{S}\times \mathscr{L}\supset \mathscr{L}_0$ is endowed with the topology induced by $\rho\times d_\infty$.
\begin{theorem}\
    \label{th:rerootingcontinuity}
    \begin{enumerate}
    \item   \label{th:rerootingcontinuity1} The set $\mathscr{L}_0$ is a neighbourhood of $\mathscr{L}_{reg}$ in $\mathscr{S}\times \mathscr{L}$ (i.e., there exists $U$ an open set such that $\mathscr{L}_{reg} \subset U \subset \mathscr{L}_0$).
    \item \label{th:rerootingcontinuity2}   The applications which map $(\ell,\gamma)$ to $\theta^-_\ell$, $\theta^+_\ell$, $x_\ell$, and $\sigma_\ell$, all defined on
    $\mathscr{L}_0$,
    are continuous on $\mathscr{L}_{reg}$.
    \end{enumerate}
\end{theorem}

We remind the reader once more that continuity on $\mathscr{L}_{reg}$ is stronger than continuity of the applications restricted to $\mathscr{L}_{reg}$: what this says is that for any fixed $(\gamma,\ell) \in \mathscr{L}_{reg}$, and any sequence $(\gamma_n,\ell_n) $ in $\mathscr{L}_{0}$ such that $(\gamma_n , \ell_n) \to (\gamma, \ell)$ we have that $\theta^-_{\ell_n} \to \theta^-_{\ell}$, and analogously for the other applications.

The proof of this theorem is given in Sections \ref{SS:32} and \ref{SS:33}. We first make a few elementary observations first. 

\begin{remarks}
    By Lemma~\ref{le:sametopo} below, the topologies induced by $d_\infty$ and $\rho$ on $\mathcal{C}([0,\cdot],\mathbb{R}^d)$, hence on $\mathscr{S}$ and $\mathscr{L}$, are equal to each other. Thus, we will alternatively use the distances $\rho\times \rho$, $\rho\times d_\infty$,  and $d_\infty\times d_\infty$.

    Furthermore, the considered topologies are metric, hence first countable. We will thus rely on sequential characterisation of continuity
     without further explanation.

    Neither $\mathscr{L}_{0}$ nor $\mathscr{L}_{reg}$ are open in $\mathscr{S}\times \mathscr{L}$ (in fact, we will prove $\mathscr{L}_{0}$ is closed).

\end{remarks}

The following is classical, and says that the two distances on paths from the introduction, $\rho$ and $d_\infty$, are equivalent. We state it and prove it for completeness.

\begin{lemma}
\label{le:sametopo}
  For any $\gamma_0,\gamma_1\in  \bigcup_{T\geq 0} \mathcal{F}( [0,T], \mathbb{R}^d)$, it holds
  \[ \rho (\gamma_0, \gamma_1)\leq \omega_{d_\infty(\gamma_0,\gamma_1) }(\gamma_0)+3 d_\infty(\gamma_0,\gamma_1) \]and \[
  d_\infty(\gamma_0,\gamma_1)\leq 2\rho (\gamma_0, \gamma_1) +\min(\omega_{2 \rho (\gamma_0, \gamma_1)}(\gamma_0),\omega_{2 \rho (\gamma_0, \gamma_1)}(\gamma_1)).\]
  In particular, the distances $\rho$ and $d_\infty$ induce the same topology on $ \mathcal{C}( [0,\cdot], \mathbb{R}^d)$.
\end{lemma}
\begin{proof}
   Let $\gamma_0,\gamma_1\in\bigcup_{T\geq 0} \mathcal{F}( [0,T], \mathbb{R}^d)$, and set $\varphi:t\mapsto t\cdot t_{\gamma_1}/t_{\gamma_0}$, which maps $[0,t_{\gamma_0}]$ to $[0,t_{\gamma_1}]$, and
   $\delta t\coloneqq |t_{\gamma_0}- t_{\gamma_1}|$.
  Then, \[ \|\varphi-\operatorname{Id}\|_{\infty,[0,t_{\gamma_0}] }=\|\varphi^{-1}-\operatorname{Id}\|_{\infty,[0,t_{\gamma_1}] }= \delta t.\]
  Assume $t_{\gamma_0}\geq t_{\gamma_1}$, so that $\gamma_0\circ \varphi$ is well-defined on $[0,t_{\gamma_0}]$. Then,
  \begin{align*} \|{\gamma}_0-{\gamma}_1\circ \varphi \|_{\infty,[0,t_{\gamma_0}] }
  &\leq
  \|{\gamma}_0-{\gamma}_0\circ \varphi \|_{\infty,[0,t_{\gamma_0}] }+
  \|{\gamma}_0\circ \varphi-{\gamma}_1\circ \varphi\|_{\infty,[0,t_{\gamma_0}] }\\
  &\leq \omega_{\|\varphi-\operatorname{Id}\|_{\infty,[0,t_{\gamma_0}] }  }(\gamma_0)+\|{\gamma}_0-{\gamma}_1\|_{\infty,[0,t_{\gamma_1}] }
  =\omega_{\delta t }(\gamma_0)+\|{\gamma}_0-{\gamma}_1\|_{\infty,[0,t_{\gamma_1}] }.
  \end{align*}
  Thus,
  \begin{equation}
  \label{eq:temp:gamma12}
  \rho (\gamma_0, \gamma_1)\leq \omega_{\delta t }(\gamma_0)+\|{\gamma}_0-{\gamma}_1\|_{\infty,[0,t_{\gamma_1}] }+\delta t \leq \omega_{d_\infty(\gamma_0,\gamma_1) }(\gamma_0)+d_\infty(\gamma_0,\gamma_1).
  \end{equation}
  In the case $t_{\gamma_0}\leq t_{\gamma_1}$, it holds $\omega_{d_\infty(\gamma_0,\gamma_1) }(\gamma_1)\leq \omega_{d_\infty(\gamma_0,\gamma_1) }(\gamma_0)+2 d_\infty(\gamma_0,\gamma_1) $. Applying~\eqref{eq:temp:gamma12} with $\gamma_0$ and $\gamma_1$ swapped, we deduce
  \[ \rho (\gamma_0, \gamma_1)\leq
  \omega_{d_\infty(\gamma_0,\gamma_1) }(\gamma_1)+d_\infty(\gamma_0,\gamma_1)\leq \omega_{d_\infty(\gamma_0,\gamma_1) }(\gamma_0)+3 d_\infty(\gamma_0,\gamma_1),
  \]
  which is valid for both $t_{\gamma_0}\geq t_{\gamma_1}$ and $t_{\gamma_0}\leq t_{\gamma_1}$.

  On the other hand, for any $\gamma_0,\gamma_1$, let  $\varphi$ be a continuous increasing bijection from $[0,t_{\gamma_0}]$  to $[0,t_{\gamma_1}]$   such that
  \[
   \|\varphi-\operatorname{Id}\|_{\infty,[0,t_{\gamma_0}] }+\|\varphi^{-1}-\operatorname{Id}\|_{\infty,[0,t_{\gamma_1}] }+\|{\gamma}_0-{\gamma}_1\circ \varphi\|_{\infty,[0,\infty)} \leq 2
   \rho (\gamma_0, \gamma_1).
  \]
  One has
  \begin{align*} d_\infty(\gamma_0,\gamma_1)
  &\leq   \|\gamma_0-\gamma_1\circ \varphi\|_{\infty,[0,t_{\gamma_0}]} +d_\infty(\gamma_1\circ\varphi,\gamma_1 ) \\
  &\leq  \|\gamma_0-\gamma_1\circ \varphi\|_{\infty,[0,t_{\gamma_0}]} + |t_{\gamma_0}-t_{\gamma_1}| + \omega_{\|\varphi-\operatorname{Id}\|_{\infty,[0,t_{\gamma_0}]} }(\gamma_1)  \\
  &\leq  \|\gamma_0-\gamma_1\circ \varphi\|_{\infty,[0,t_{\gamma_0}]} + \|\varphi -\operatorname{Id}\|_{\infty,[0,t_{\gamma_0}]} + \omega_{\|\varphi-\operatorname{Id}\|_{\infty,[0,t_{\gamma_0}]} }(\gamma_1) \\
  &\leq 2\rho (\gamma_0, \gamma_1) +\omega_{2 \rho (\gamma_0, \gamma_1)}(\gamma_1),
  \end{align*}
which by symmetry between $\gamma_0$ and $\gamma_1$ we can improve into \[d_\infty(\gamma_0,\gamma_1)\leq 2\rho (\gamma_0, \gamma_1) +\min(\omega_{2 \rho (\gamma_0, \gamma_1)}(\gamma_0),\omega_{2 \rho (\gamma_0, \gamma_1)}(\gamma_1)).\]

  When $\gamma_0,\gamma_1$ are restricted to $\mathcal{C}([0,\cdot],\mathbb{R}^d)$, $\omega_\epsilon(\gamma_0)\to 0$ as $\epsilon\to 0$, hence $\rho$ and $d_\infty$ induce the same topology.
\end{proof}

Later on we will also need the following
\begin{lemma}
  \label{le:closed}
  The set $\mathscr{L}_0$ is closed in $\mathscr{S}\times \mathscr{L}$.
\end{lemma}
\begin{proof}
By Lemma~\ref{le:sametopo}, we can work with the distance $\rho\times \rho$.
  For two subsets $X,Y$ from $\mathbb{R}^2$, let $d_{\inf}(X,Y)\coloneqq \inf \{ |x-y|: x\in X, y\in Y\} $. For $X,Y$ compact, this infimum is achieved, i.e. there exist $x\in X, y\in Y$ such that $d_{\inf}(X,Y)=|x-y|$. In particular, for $X,Y$ compact, $d_{\inf}(X,Y)=0$ if and only if $X\cap Y\neq \emptyset$.

  Let $(\gamma,\ell) \in (\mathscr{S}\times \mathscr{L}) \setminus \mathscr{L}_0$. By compactness, and since $\Range(\gamma)\cap\Range(\ell)=\emptyset$, \[ d_{\inf}(\Range(\gamma),\Range(\ell))>0.\]

  Let $(\gamma',\ell')$ at $\rho\times \rho$ distance strictly less than $\delta=d_{\inf}(\Range(\gamma),\Range(\ell))$ from $(\gamma,\ell)$.
  By compactness, there exists $\sigma\in [0,t_\gamma]$ and $t\in [0,t_\ell]$ such that $d_{\inf}(\Range(\gamma'),\Range(\ell'))=|\gamma'(\sigma)-\ell'(t)| $.
  By considering $\varphi_\gamma,\varphi_\ell$ continuous bijection such that
  $\| \gamma-\gamma'\circ \varphi_\gamma\|_\infty
  +\| \ell-\ell'\circ \varphi_\ell\|_\infty
  <\delta $, we deduce $|\gamma'(\sigma)-\ell'(t)| > d_{\inf}(\Range(\gamma),\Range(\ell))- \delta =0 $. Thus
  $d_{\inf}(\Range(\gamma'),\Range(\ell'))>0$, and by compactness we deduce $\Range(\gamma')\cap \Range(\ell')=\emptyset$, hence $(\gamma',\ell')\notin \mathscr{L}_0$.
  We have proved  $ \mathscr{S}\times \mathscr{L} \setminus \mathscr{L}_0$ is open, hence $\mathscr{L}_0$ is closed.
\end{proof}

\subsection{\texorpdfstring{$\mathscr{L}_0$}{L0} as a neighborhood of \texorpdfstring{$\mathscr{L}_{reg}$}{Lreg}}

\label{SS:32}

Note that the first item~\ref{th:rerootingcontinuity1} in Theorem~\ref{th:rerootingcontinuity} can be reformulated as follows. 
    Let $(\gamma,\ell)\in \mathscr{L}_{reg}$. Then, there exists $\delta>0$ such that for all $\gamma'\in \mathscr{S}$ and $\ell'\in \mathscr{L}$ such that $d_\infty(\gamma,\gamma')<\delta$ and $d_\infty(\ell,\ell')<\delta$, it holds $\ell'\in \mathscr{L}_{\gamma'}$.
    
    We shall directly prove the following much stronger version, which also state that the spatial root  $x_{\gamma,\ell}$ and temporal root $\theta_{\gamma,\ell}$ can be approached respectively by {some} intersection point $x$ between $\gamma'$ and $\ell'$ and by {some} time $\theta: \ell'(\theta)=x$. Remark we do not claim yet that we can chose $x=x_{\gamma',\ell'}$ and $\theta\in \Theta_{\gamma',\ell'}$
(even less that we can chose $\theta=\theta^-_{\gamma',\ell'}$ or $\theta=\theta^+_{\gamma',\ell'}$  ), so this is still far from proving the continuity of $\gamma,\ell\mapsto x_{\gamma,\ell}, \ \theta_{\gamma,\ell}$.
\begin{lemma}
  \label{le:notstupid}
    Let $(\gamma,\ell)\in \mathscr{L}_{reg}$ and $\epsilon>0$. Then, there exists $\delta>0$ such that for all $\gamma'\in \mathscr{S}$ and $\ell'\in \mathscr{L}$ such that $d_\infty(\gamma,\gamma')<\delta$ and $d_\infty(\ell,\ell')<\delta$, it holds $\ell'\in \mathscr{L}_{\gamma'}$
    and there exists $\theta\in [\theta_\ell-\epsilon,\theta_\ell+\epsilon]$ such that $\ell'(\theta)\in \Range(\gamma')$ and $|\ell'(\theta)- x_\ell |\leq \epsilon$.
\end{lemma}
\begin{proof}

  By Lemma~\ref{le:sametopo}, it suffices to prove the same result but with $\rho$ instead of $d_\infty$.
	First, we reduce to the problem to the case when $t_\gamma=t_{\gamma'}$, $t_\ell=t_{\ell'}$, $\| \gamma-\gamma'\|_\infty<\delta$, and $\| \ell-\ell'\|_\infty<\delta$.

  Let $\delta_1\in (0,\epsilon)$ be sufficiently small that the continuity modulus $\omega_{\delta_1}(\ell)$ is smaller than both $\epsilon/4$ and $\rhos(x_\ell)/2$. Let $\theta^L,\theta^R\in [\theta_\ell-\delta_1, \theta_\ell+\delta_1]$ such that $\ell(\theta^L)\in L(x_\ell)$, $\ell(\theta^R)\in R(x_\ell)$, which exist since $\ell\in \mathcal{L}_{reg,\gamma}$.
  Then, let
  \begin{equation}
  \label{eq:deltadef}
  \delta \coloneqq \frac{1}{4} \min \{\epsilon,\delta_1, \rhos(x_\ell), d(\ell(\theta^L), R(x_\ell) ), d(\ell(\theta^L), L(x_\ell) ) \}>0.
  \end{equation}

  Let now $\gamma',\ell'\in \mathscr{S}\times\mathscr{L}$ such that $d_\infty(\gamma,\gamma')<\delta$ and $d_\infty(\ell,\ell')<\delta$. Let further $\varphi_{\gamma}:[0,t_\gamma]\to [0,t_{\gamma'}]$ and $\varphi_{\ell}:[0,t_\ell]\to [0,t_{\ell'}] $
	be continuous bijections such that
  \[
  \|\gamma-\gamma'\circ \varphi_{\gamma}\|_{\infty,[0,t_{\gamma}] }+\frac{\|\varphi_{\gamma}-\operatorname{Id}\|_{\infty,[0,t_{\gamma}] }+\|\varphi_{\gamma}^{-1}-\operatorname{Id}\|_{\infty,[0,t_{\gamma'}] }}{2}<\delta
  \]
  \[\mbox{and} \qquad
  \|\ell-\ell'\circ \varphi_{\ell}\|_{\infty,[0,t_{\ell}] }+\frac{\|\varphi_{\ell}-\operatorname{Id}\|_{\infty,[0,t_{\ell}] }+\|\varphi_{\ell}^{-1}-\operatorname{Id}\|_{\infty,[0,t_{\ell'}] }}{2} <\delta.
  \]
  Let $\tilde{\theta}^L=\varphi_\ell^{-1}(\theta^L)$, $\tilde{\theta}^R=\varphi_\ell^{-1}(\theta^R)$. For $s,t\in [0,\infty)$, write $[s,t]$ for the interval $[s\wedge t,s\vee t ]$.
  We will find $\tilde{\theta}\in [\tilde{\theta}^L,\tilde{\theta}^R  ]$ such that $\ell'\circ \varphi_\ell(\theta)\in \Range(\gamma'\circ \varphi_\gamma)$.
  This is sufficient to conclude: letting then $ \theta\coloneqq \varphi_\ell(\tilde{\theta})$, we have
  \begin{itemize}
  \item By  monotonicity of $\varphi_\ell$,  $ \theta= \varphi_\ell(\tilde{\theta})\in \varphi_\ell([ \tilde{\theta}^L,\tilde{\theta}^R   ])=
  [\theta^L,\theta^R ]\subseteq [\theta_\ell-\epsilon,\theta_\ell+\epsilon]$.
  \item $\ell(\theta)=\tilde{\ell}(\tilde{\theta}) \in \Range(\gamma'\circ \varphi_\gamma)= \Range(\gamma')$.
  \item $|{\ell}'({\theta})-x_\ell| = |\tilde{\ell}'(\tilde{\theta})-\ell(\theta_\ell)|
  \leq |   \tilde{\ell}'(\tilde{\theta})-{\ell}(\tilde{\theta})|+|{\ell}(\tilde{\theta})-\tilde{\ell}(\tilde{\theta})|+|\tilde{\ell}(\tilde{\theta})-\ell(\theta_\ell)|
  \leq \| \tilde{\ell}'-\ell\|_\infty+\omega_{\|\phi_\ell-\operatorname{Id}\|_\infty}(\ell)+\epsilon/2< \delta_1+\omega_{\delta_1}(\ell)+\epsilon/2=\epsilon$.
  \end{itemize}

We have finish the reduction step (remark also that we no longer need to prove $|{\ell}'({\theta})-x_\ell|$), and we now stablish the existence of $\tilde{\theta}$. We only need to consider $\tilde{\ell}'\coloneqq \ell'\circ \varphi_{\ell}$ and $\tilde{\gamma}'\coloneqq \gamma'\circ \varphi_{\gamma}$ rather than $\ell'$ and $\gamma'$, so we write $\ell'$ for $\tilde{\ell}'$, $\gamma'$ for $\tilde{\gamma}'$, $\theta^L$ for $\tilde{\theta}^L$ and $\theta^R$ for $\tilde{\theta}^R$ .
  Thus $t_{\gamma'}=t_\gamma$, $t_{\ell'}=t_\ell$, $\|\gamma'-\gamma\|_\infty<\delta$, and $\|\ell'-\ell\|_\infty<\delta$.

  Remember the loop
  $\eta= \gamma_{|[\sigma_\gamma^-(x), \sigma_\gamma^+(x)]}\cdot \overleftarrow{\gamma}$ which parametrises $\partial L(x_\ell)$ anticlockwise.
  For $a,b\in \mathbb{R}^2$ (with possibly $a=b$), let $[a,b]$ be the function from $[0,1]$ that goes in an affine way from $a$ at time $0$ to $b$ at time $1$.
  Let $a_\pm=\gamma(\sigma_\gamma^\pm(x))$, $a'_\pm=\gamma'(\sigma_\gamma^\pm(x)))$, and define
  \[ \eta_0\coloneqq [a_-,a_-]\cdot \gamma_{|[\sigma_\gamma^-(x), \sigma_\gamma^+(x)]}\cdot [a_+,a_+]\cdot\overleftarrow{\gamma},\]
  \[ \eta_1\coloneqq [a_-,a'_-]\cdot \gamma'_{|[\sigma_\gamma^-(x), \sigma_\gamma^+(x)]}\cdot [a'_+,a_+]\cdot \overleftarrow{\gamma}.\]
  Remark $\eta_0,\eta_1$ are defined on the same time interval and $\|\eta_1-\eta_0\|_\infty\leq\|\gamma-\gamma'\|_\infty<\delta$.

  For $\lambda\in [0,1]$, define the loop
  $\eta_\lambda=(1-\lambda) \eta_0+\lambda \eta_1$. For all $\lambda\in [0,1]$, $\|\eta_\lambda- \eta_0\|_\infty=(1-\lambda)   \|\eta_1-\eta_0\|_\infty<\delta$.
  Thus, $d(\ell(\theta^R),\eta_\lambda)> d(\ell(\theta^R), \eta_0) -\delta\geq 0$  and similarly $d(\ell(\theta^L),\eta_\lambda)>0$. Thus $(\eta_\lambda)_\lambda$ is a homotopy from $\eta_0$ to $\eta_1$ which avoids $\ell(\theta^R)$ and $\ell(\theta^L)$. Since
  the property that a point $x$ lies in the unbounded path-connected component delimited by a loop is invariant by homotopies avoiding $x$, we deduce that  $\ell(\theta^R)$ lies in the unbounded path-connected component delimited by $\eta_1$, which we will call $R'$, and that
  $\ell(\theta^L)$ does not.
  Furthermore, $d(\ell(\theta^R),\eta_1)\geq d(\ell(\theta^R),\eta_0)-\|\eta_1,\eta_0\|_\infty\geq 2\delta-\delta=\delta$. Thus, $\ell'(\theta^R ) \in B_\delta( \ell(\theta^R)) \subset R'$.
  From a similar argument, we deduce $
  \ell'(\theta^L)  \notin R'$.
  By intermediate value theorem, there exists $\theta\in [\theta^L,\theta^R]$ such that $\ell'(t)\in \Range(\eta_1)$. We will now show that $\ell'([\theta^L,\theta^R] )$ does not intersect $\overleftarrow{\gamma}$ nor the segments $[a_-,a'_-]$ and $[a'_+,a_+]$, from which it follows that
	$\ell'(t)\in \Range(\gamma'_{|[\sigma_\gamma^-(x), \sigma_\gamma^+(x)]})\subseteq \Range(\gamma')$, which concludes the proof.
	%
\end{proof}

\subsection{Continuity on \texorpdfstring{$\mathscr{L}_{reg}$}{Lreg} }

\label{SS:33}

We now start the proof of the continuity of $(\gamma,\ell)\mapsto \sigma_{\gamma,\ell}$ on $\mathcal{L}_{reg}$,  which will be concluded in Proposition~\ref{prop:sigmaxcont}.
\begin{lemma}
    \label{le:sigmaCont}
    Let $\gamma_n\in \mathscr{S} $ be a sequence which converges toward $\gamma\in \mathscr{S}$. Let $x_n\in \Range(\gamma_n)$ which converges toward $x\in \Range(\gamma )$.
    Then, $\gamma_n^{-1}(x_n)$ converges toward $\gamma^{-1}(x)$.
\end{lemma}
\begin{proof}
    First remark that for any sequence $(\sigma_n)_n$ such that $\sigma_n\in[0,t_{\gamma_n}]$ for all $n$, we have the implication
    \begin{equation}
    \label{eq:impliqueCV}
    \sigma_n\underset{n\to \infty}\longrightarrow \sigma \implies \gamma_n(\sigma_n)\underset{n\to \infty}\longrightarrow \gamma(\sigma).
    \end{equation} Indeed, the left-hand side implies $\sigma\leq t_\gamma$ and
        \[
    | \gamma_n(\sigma_n)- \gamma(\sigma)|\leq | \gamma_n(\sigma_n)- \gamma(\sigma_n\wedge t_\gamma)|+| \gamma(\sigma_n\wedge t_\gamma )- \gamma(\sigma)|\leq d_\infty(\gamma_n,\gamma)+| \gamma(\sigma_n\wedge t_\gamma)- \gamma(\sigma)|\underset{n\to \infty}\longrightarrow 0.
    \]
    Let $x_n,x$ as in the lemma, and let $\sigma$ be any accumulation point for $\{\gamma_n^{-1}(x_n)\}$. Since, for all $n$ large enough, $\gamma_n^{-1}(x_n)\in [0,t_\gamma+1]$  which is compact, it suffices to prove $\sigma= \gamma^{-1}(x)$. We let $n$ vary along an infinite subset of $\mathbb{N}$ such that $\gamma_n^{-1}(x_n)\to \sigma$. By~\eqref{eq:impliqueCV}, it follows
    $x_n=\gamma_n(\gamma_n^{-1}(x_n) )\to \gamma(\sigma)$.
    By assumption, $x_n\to x$, thus $x=\gamma(\sigma)$. Since $\gamma$ is injective, we deduce $\sigma=\gamma^{-1}(x)$, which concludes the proof.
\end{proof}
\begin{corollary}
    \label{coro:notstupidbound}
    Let $(\gamma_n,\ell_n)\in \mathscr{L}_0$ be a sequence which converges toward $(\gamma,\ell)\in \mathscr{L}_{reg}$. Then, \[\limsup \sigma_{\gamma_n,\ell_n}\leq \sigma_{\gamma,\ell}.\]
\end{corollary}
\begin{proof}
    By Lemma~\ref{le:notstupid}, there exists a sequence $\theta_n$ such that $\theta_n$ converges toward $\theta_{\gamma,\ell}$, and $x_n\coloneqq \ell_n(\theta_n)$ converges toward $x_{\gamma,\ell}$, and for all $n$ large enough $\ell_n(\theta_n)\in \Range(\gamma_n)$.
    Applying Lemma~\ref{le:sigmaCont}, which we can do since $x_{\gamma,\ell}\in \Range(\gamma )$, we deduce that
    $\gamma_n^{-1}( x_n)$ converges toward $\gamma^{-1}(x_{\gamma,\ell})=\sigma_{\gamma,\ell}$   as $n\to \infty$. Since $x_n\in \Range(\ell_n)\cap \Range(\gamma_n)$, $\sigma_{ \gamma_n,\ell_n}\leq \gamma_n^{-1}( x_n)$, thus $\limsup \sigma_{\ell_n, \gamma_n}\leq \sigma_{\gamma,\ell}$.
\end{proof}

\begin{lemma}
    \label{le:stupidbound}
    Let  $(\gamma_n,\ell_n)\in \mathscr{L}_0 $ be a sequence which converges toward $(\gamma,\ell)\in \mathscr{L}_{reg}$.
    Let $(t_n)_n$ be a sequence of times such that $\ell_n(t_n)\in \Range(\gamma_n)$ and assume $t_n\to t$ as $n\to \infty$. Then, $\ell(t)\in\Range(\gamma )$.
\end{lemma}
\begin{proof}
    By triangle inequality,
    \[
    d(\ell(t), \Range(\gamma) )
    \leq |\ell(t)- \ell(t_n\wedge t_\ell)|+ |\ell(t_n\wedge t_\ell)- \ell_n(t_n)|+d(\ell_n(t_n), \Range(\gamma)).
    \]
  Since $\ell_n\to \ell $, $t_n\leq t_{\ell_n}\underset{n\to \infty}{\longrightarrow} t_\ell$. Since furthermore $t_n\to t$, it holds $t_n\wedge t_\ell\to t$. By continuity of $\ell$, the first term on the right-hand side goes to $0$ as $n\to \infty$. The second term is less than $d_\infty(\ell,\ell_n)$ which goes to $0$ as well. Since $\ell_n(t_n)\in \Range(\gamma_n)$, the third term is less than $\rho(\gamma_n,\gamma)$, which goes to $0$ as well. Since the left-hand side does not depend on $n$, it is equal to $0$, i.e. $\ell(t)\in \overline{\Range(\gamma)}=\Range(\gamma)$.
\end{proof}

\begin{corollary}
    \label{coro:stupidbound}
     Let  $(\gamma_n,\ell_n)\in \mathscr{L}_0 $ be a sequence which converges toward $(\gamma,\ell)\in \mathscr{L}_{reg}$.
    Then, \[\liminf \sigma_{\gamma_n,\ell_n}\geq \sigma_{\gamma,\ell}.\]
\end{corollary}
\begin{proof}
    Let $\sigma$ be an accumulation point for $\sigma_{\gamma_n,\ell_n}$, and fix a subsequence along which $\sigma_{\gamma_n,\ell_n}\to \sigma$. Take a further subsequence such that $\theta_{\gamma_n,\ell_n}$ converges along this subsequence (this exists  by compactness of $[0,t_\gamma+1]$), and let $\theta$ be the limit. Along this subsequence, $x_{\gamma_n,\ell_n}= \ell_n(\theta_{\gamma_n,\ell_n})$ converges toward $x\coloneqq \ell(\theta)$ (this is by continuity of $\ell$, convergence of $\theta_n$ toward $\theta$, and uniform convergence of $\ell_n$ toward $\ell$).
By Lemma~\ref{le:stupidbound}, $x\in \Range(\gamma)$. Thus, we can apply Lemma~\ref{le:sigmaCont} along the subsequence, to deduce that $\sigma_{\gamma_n,\ell_n}=\gamma_n^{-1}(x_{\gamma_n,\ell_n})$ converges toward $\gamma^{-1}(x)$. By unicity of the limit, $\sigma= \gamma^{-1}(x)$. Since
 $x\in \Range(\gamma)\cap \Range(\ell)$,  $\sigma_{\gamma, \ell}\leq \gamma^{-1}(x)=\sigma  $, which concludes the proof.
\end{proof}

\begin{proposition}
    \label{prop:sigmaxcont}
    The maps $(\gamma,\ell)\mapsto \sigma_{\gamma,\ell}$ and $(\gamma,\ell)\mapsto x_{\gamma,\ell}$ defined on $\mathscr{L}_0$ are continuous on $\mathscr{L}_{reg}$.
\end{proposition}
\begin{proof}
    The continuity of $\sigma$ is equivalent to Corollaries~\ref{coro:notstupidbound} and~\ref{coro:stupidbound}. Since $x_{\gamma,\ell}=\gamma(\sigma_{\gamma,\ell} )$, the continuity of $x$ follows:
    \[
    |x_{\gamma,\ell}-x_{\gamma',\ell'}|=|\gamma(\sigma_{\gamma,\ell} )-\gamma'(\sigma_{\gamma',\ell'} )|
		\leq d_\infty(\gamma',\gamma)+ |\gamma(\sigma_{\gamma',\ell'})-\gamma(\sigma_{\gamma,\ell})|\underset{(\gamma',\ell')\to (\gamma,\ell)}\longrightarrow 0,
    \]
    by continuity of $\gamma$.
\end{proof}
It remains to prove the continuity of $\theta^-$ and $\theta^+$.

\begin{corollary}
  \label{coro:thetacont}
  The functions $\theta^-$ and $\theta^+$ defined on $\mathscr{L}_0$ are continuous on $\mathscr{L}_{reg}$.
\end{corollary}
\begin{proof}
    We treat the case of $\theta^-$, the case of $\theta^+$ is identical.    Let $(\gamma,\ell)\in \mathscr{L}_{reg}$, and $(\gamma_n, \ell_n)\in \mathscr{L}_0$ which converges toward $(\gamma, \ell)$.
    By compactness of $[0,t_\ell+1]$, it suffices to prove that for any accumulation point $\theta$ of $(\theta^-_{\ell_n, \gamma_n})_n$, it holds
    $\theta=\theta^-_{\gamma,\ell}$.
    Fix a subsequence along which $\theta^-_{\gamma_n,\ell_n}\to \theta$. Since $(\gamma,\ell)\in \mathcal{L}_{reg}$,
     $\theta^-_{\gamma,\ell}=\theta^+_{\gamma,\ell}=\theta_{\gamma,\ell}$. Along the subsequence, $x_{\gamma_n,\ell_n}=\ell_n(\theta^-_{\gamma_n,\ell_n})\to \ell(\theta)$ (By uniform convergence of $\ell_n$ toward $\ell$, continuity of $\ell$, and convergence of $\theta^-_{\gamma_n,\ell_n}$ toward $\theta$).
      On the other hand, by continuity of $(\gamma,\ell)\mapsto x_{\gamma,\ell}$ (Proposition~\ref{prop:sigmaxcont}), $x_{\gamma_n,\ell_n}\to x_{\gamma,\ell}= \ell(\theta_{\gamma,\ell})$. Thus
    $ \ell(\theta)=  \ell(\theta_{\gamma,\ell})$, i.e. $\theta\in \Theta_{\gamma,\ell}$. Since $\ell$ is regular for $\gamma$, $\Theta_{\gamma,\ell}=\{\theta_{\gamma,\ell}\}$, hence $\theta=\theta_{\gamma,\ell}$, which concludes the proof.
\end{proof}

\section{Continuity of the attachement map}
\label{sec:continuity}
\subsection{\texorpdfstring{$\delta$}{delta}-isomorphisms, suited and strongly suited}
In order to construct a distance on $\mathscr{L}^{lf}$, we define a {\it $\delta$-isomorphism from $\mathcal{L}$ to $\mathcal{L}'$} in $\mathscr{L}^{lf}$ as a bijection
$\phi:\mathcal{L}_0\to \mathcal{L}'_0$, where $\mathcal{L}_0,\mathcal{L}'_0$ are such that
$\mathcal{L}'_{\geq \delta}\subseteq \mathcal{L}'_0\subseteq \mathcal{L}'$ and $\mathcal{L}_0= \mathcal{L}_{\geq \delta} \cup \phi^{-1}(\mathcal{L}'_{\geq \delta})$
(from which it follows that $\mathcal{L}'_0= \mathcal{L}'_{\geq \delta} \cup \phi(\mathcal{L}_{\geq \delta})$), and such that for all $\ell\in \mathcal{L}_0$, \[d_\infty(\ell,\phi(\ell))\leq\delta/2.\]

We will say that an application $\phi$ {\it induces}  a $\delta$-isomorphism if it is defined on $\mathcal{L}_{\geq \delta}$, and the application $\tilde{\phi}$ defined from $\mathcal{L}_{\geq \delta} \cup \phi^{-1}(\mathcal{L}'_{\geq \delta})$ to
$\phi(\mathcal{L}_{\geq \delta}) \cup \mathcal{L}'_{\geq \delta}$ by $\tilde{\phi}(\ell)\coloneqq\phi(\ell)$ is a $\delta$-isomorphism.

For a $\delta$-isomorphism, it holds in particular
\[ \mathcal{L}_0\subseteq \mathcal{L}_{\geq \delta/2}\qquad \mbox{and} \qquad \mathcal{L}'_0\subseteq \mathcal{L}'_{\geq \delta/2}
.\]
Indeed, $\mathcal{L}_{\geq \delta}\subseteq \mathcal{L}_{\geq \delta/2}$, and for $ \ell\in \phi^{-1}(\mathcal{L}_{\geq \delta})$, $t_\ell\geq t_{\phi(\ell)}-|t_\ell-t_{\phi(\ell)}|\geq  \delta -\delta/2 $,
thus $\phi^{-1}(\mathcal{L}_{\geq \delta})\subseteq  \mathcal{L}_{\geq \delta/2}$.
The inclusion $\mathcal{L}'_0\subseteq \mathcal{L}'_{\geq \delta/2}$ can be checked with similar computation.

\begin{lemma}
\label{le:tri}
  Let $\delta>0$ and let $\phi$ be a $\delta$-isomorphism from $\mathcal{L}$ to $\mathcal{L}'$. Then,
  \begin{itemize}
  \item $\phi^{-1}$ is a $\delta$-isomorphism from $\mathcal{L}'$ to $\mathcal{L}$.

  \item For  $\delta'>\delta$, $\phi$ induces a $\delta'$-isomorphism.

  \item If $\psi$ is a $\delta'$-isomorphism from $\mathcal{L}'$ to $\mathcal{L}''$, $\psi\circ \phi$ induces a $(\delta+\delta')$-isomorphism from $\mathcal{L}$ to $\mathcal{L}''$.
  \end{itemize}
\end{lemma}
\begin{proof}
  The first item is trivial to check.

  For the second one, let $\tilde{\mathcal{L}}_0\coloneqq \mathcal{L}_{\geq \delta'}\cup \phi^{-1}( \mathcal{L}'_{\geq \delta'}  )$ and $\tilde{\mathcal{L}}'_0\coloneqq \phi(\mathcal{L}_{\geq \delta'})\cup   \mathcal{L}'_{\geq \delta'}  $.
  Since $\phi$ is injective and $\tilde{\mathcal{L}}'_0=\phi(\tilde{\mathcal{L}}_0)$,
  $\phi:\tilde{\mathcal{L}}_0\to\phi(\tilde{\mathcal{L}}_0)$ is bijective. For all $\ell\in \tilde{\mathcal{L}}_0$, $d_\infty(\ell,\phi(\ell))<\delta/2<\delta'/2$.

  For the last item, 
  remark first that $\psi^{-1}(\mathcal{L}''_{\geq \delta+\delta'})\subseteq \mathcal{L}'_{\geq \delta}$ and that $\phi(\mathcal{L}_{\geq \delta+\delta'})\subseteq \mathcal{L}'_{\geq \delta'}$, so that the sets
  \[
  \mathcal{L}_0\coloneqq \mathcal{L}_{\geq \delta+\delta'}\cup \phi^{-1}(\psi^{-1}(\mathcal{L}''_{\geq \delta+\delta'})), \qquad
  \mathcal{L}''_0\coloneqq\psi ( \phi( \mathcal{L}_{\geq \delta+\delta'} ))\cup (\mathcal{L}''_{\geq \delta+\delta'})
  \]
  are well-defined.
  Furthermore, $\psi\circ \phi$ is well-defined on $\mathcal{L}_0$, and $\psi\circ \phi(\mathcal{L}_0)=\mathcal{L}''_0$. Since $\psi\circ \phi$ is injective, we deduce
  $\Phi:\mathcal{L}_0\to  \mathcal{L}''_0$, $\ell\mapsto \psi\circ \phi(\ell)$ is a bijection, and for all $\ell\in \mathcal{L}_0$,
  \[d_\infty(\ell,\Phi(\ell))\leq d_\infty(\ell,\phi(\ell)) +d_\infty(\phi(\ell),\psi(\phi(\ell)))\leq (\delta+\delta')/2 .\qedhere\]
\end{proof}
Given $\gamma,\gamma'\in \mathscr{S}$, we will say that $\phi$ is a {\it $\delta$-suited isomorphism}, with respect to $(\gamma,\gamma')$, if it is a $\delta$-isomorphism and furthermore $\phi (\mathcal{L}_{\gamma, \geq \delta'})\subseteq \mathcal{L}'_{\gamma'}$ and $\phi^{-1} (\mathcal{L}'_{\gamma', \geq \delta'})\subseteq \mathcal{L}_{\gamma}$.
We will say that
$\phi$ is a {\it $\delta$-strongly suited isomorphism}, with respect to $(\gamma,\gamma')$, if it is a $\delta$-suited isomorphism
and furthermore,
\[
\forall \ell,\ell'\in \mathcal{L}_{\gamma, \geq \delta'} \cup \phi^{-1}(\mathcal{L}'_{\gamma', \geq \delta'}  ), \qquad \ell \prec_\gamma \ell' \iff \phi(\ell)\prec_{\gamma'} \phi(\ell').
\]
For $\mathcal{X}=(\gamma,\mathcal{L})$ and $\mathcal{X}'=(\gamma',\mathcal{L}')$, we will say that $\phi$ is a $\delta$-suited (resp. strongly suited) isomorphism from $\mathcal{X}$ to $\mathcal{X}'$ when it is a $\delta$-suited (resp. strongly suited) isomorphism from $\mathcal{L}$ to $\mathcal{L}'$ with respect to $(\gamma,\gamma')$.

\subsection{Distances on \texorpdfstring{$\Conf_0$}{R0}}
\label{sub:d}
It follows from Lemma \ref{le:tri} that $d$ defined by
\begin{equation}
 \label{eq:def:dist}
d( \mathcal{L},\mathcal{L}'  )\coloneqq
\inf \{\delta>0: \mbox{there exists a $\delta$-isomorphism from $\mathcal{L}$ to $\mathcal{L}'$}\}
\end{equation}
is a distance on $\mathscr{L}^{lf}$.
%
We use it to define a distance $d_{\Conf_0}$ on $\Conf_0$: for $\mathcal{X}=(\gamma, \mathcal{L})$ and $\mathcal{X}'=(\gamma', \mathcal{L}')$ in $\Conf_0$,
\[
d_{\Conf_0}(\mathcal{X},\mathcal{X}' )\coloneqq \rho(\gamma,\gamma')+
d( \mathcal{L},\mathcal{L}'  )+|t_{\mathcal{X}}-t_{\mathcal{X}'} |+\| \delta\mapsto (\omega_\delta(\mathcal{L})-\omega_ \delta(\mathcal{L}'))\|_{\infty,[0,1]}
\]
The ball of $(\Conf_0,d_{\Conf_0})$ with radius $r$ centred at $\mathcal{X}$ is denoted $B_r(\mathcal{X})$.
The last two terms in the definition of $d_{\Conf_0}$ may look a bit artificial, but they are very crucial. The term $|t_{\mathcal{X}}-t_{\mathcal{X}'} |$ prevents the small loops to accumulate total time of order $1$ (this prevents elements in $\Att(\mathcal{X}')$ from `getting stuck' as $\mathcal{X}'\to \mathcal{X}$). If we were to remove it from the definition of $d_{\Conf_0}$, the map $\Att$ would be nowhere continuous (although it may remain continuous when viewed as a map with values in the space of paths up non-decreasing reparametrisations). As for the last term involving the continuity moduli, it prevents for example to add a single loop that would be very small in term of time duration but macroscopic spacewise (the limiting $\Att(\mathcal{X}')$ would `skip' such a loop).

It follows from Lemma~\ref{le:Tbound} and~\ref{le:omegabound} below that, in order to prove convergence in $d_{\Conf_0}$, it suffices to prove convergence in the weaker distance $d'_{\Conf_0}$ given by
\begin{equation}
\label{eq:defdprime}
d'_{\Conf_0}(\mathcal{X},\mathcal{X}' )\coloneqq  \rho(\gamma,\gamma')+
d( \mathcal{L},\mathcal{L}'  ),
\end{equation}
together with uniform boundedness of $\epsilon\mapsto \omega_\epsilon$ and $\epsilon\mapsto T_{\leq\epsilon}$ (defined above these lemmas in \eqref{eq:def:Tleqepsilon}) by functions vanishing at $0$. Balls for the distance $d'_{\Conf_0}$ are written $B'_r(\mathcal{X})$.

On the other hand, we will now prove that for any given $\mathcal{X}=(\gamma,\mathcal{L})\in \mathcal{R}$, 
for all $\epsilon>0$, there exists $\delta>0$ such that any $\mathcal{X}'\in \mathcal{R}_0$ at distance less than $\delta$ from $\mathcal{X}$, there exists a strongly suited isomorphism from $\mathcal{X}$ to $\mathcal{X}'$. In words, the topology induced by $d_{\Conf_0}$ is equal to the a priori stronger topology obtained by replacing `isomorphism' with `strongly suited isomorphism' in \eqref{eq:def:dist} (but strongly suited isomorphisms do not compose well, and triangle inequality would fail with such a modification).
\begin{remark}
  The set $\Conf_0$ is a measurable subset of $\mathscr{S}\times \mathscr{L}^{lf}$ for the $\sigma$-algebra generated by $(\gamma,\mathcal{L})$ (the pair of an $SLE_2$ path and a Brownian loop soup), and the identity from $(\Conf_0,\sigma(\mathscr{S}\times \mathscr{L}^{lf}) )$ to $(\Conf_0, d_{\Conf_0})$ is measurable, which we leave as an exercise about point processes to prove. We will see that $(\gamma, \mathcal{L})\in \Conf_0$
  with probability $1$, and it thus makes sense to consider it as a random variable in $(\Conf_0, d_{\Conf_0})$. 
\end{remark}

\subsection{Equality with stronger topology}
\label{sub:strong}

As announced above, we now show that convergence (for the distance induced by the notion of $\delta$-isomorphism) implies a stronger form of convergence with respect to the notion of strongly suited isomorphism. 

\begin{lemma}
  \label{le:suited}
  Let $\mathcal{X}=(\gamma,\mathcal{L})\in\Conf$. For all $\epsilon>0$, there exists $\delta=\delta(\mathcal{X})\in(0,\epsilon]$ such that for all $\mathcal{X}'=(\gamma',\mathcal{L}')\in B'_{\delta}(\mathcal{X})$, all $\delta$-isomorphism $\phi$ from $\mathcal{X}$ to $\mathcal{X}'$ induces an $\epsilon$-suited isomorphism (with respect to $(\gamma,\gamma')$).
\end{lemma}
\begin{proof}
  Let $\epsilon>0$ and $\mathcal{X}=(\gamma,\mathcal{L})\in\Conf$. For $\ell\in \mathcal{L}_\gamma$, let $\delta(\gamma,\ell)>0$ such that for all $(\gamma',\ell')\in \mathscr{S}\times \mathscr{L}$ at $(\rho\times d_\infty)$-distance less than or equal to $\delta(\gamma,\ell)$ from $(\gamma,\ell)$, it holds $\ell'\in \mathscr{L}_{\gamma'}$.
  The existence of such a $\delta(\gamma,\ell)$ follows from Theorem~\ref{th:rerootingcontinuity}, item~\ref{th:rerootingcontinuity1} and the assumption~\eqref{it:reg} in the definition of $\Conf$.

  For $\ell\in \mathcal{L}\setminus \mathcal{L}_\gamma $, $(\gamma,\ell)\notin \mathscr{L}_0$.
  Since $\mathscr{L}_0$ is closed in $\mathscr{S}\times \mathscr{L}$ (Lemma~\ref{le:closed}), there exists
  $\delta'(\gamma,\ell)>0$ such that for all $(\gamma',\ell')\in \mathscr{S}\times \mathscr{L}$ at $(\rho\times d_\infty)$-distance less than or equal to $\delta'(\gamma,\ell)$ from $(\gamma,\ell)$, it holds $\ell'\notin \mathscr{L}_{\gamma'}$.

  Let \[\delta\coloneqq\epsilon \wedge \min \{\delta(\gamma,\ell)/2 : \ell \in \mathcal{L}_{\gamma, \geq \epsilon} \}
  \wedge \min \{ \delta'(\gamma,\ell) : \ell \in \mathcal{L}_{ \geq \epsilon/2}\setminus \mathcal{L}_\gamma \}
  >0.\]
  Let $\phi$ be any $\delta$-isomorphism from $\mathcal{X}$ to $\mathcal{X}'=(\gamma',\mathcal{L}')\in B_{\delta}(\mathcal{X})$. Since $\phi$ induces an $\epsilon$-isomorphism, we only need to check the extra conditions that $\phi(\mathcal{L}_{\gamma,\geq \epsilon})\subseteq \mathcal{L}'_{\gamma'}$ and $\phi^{-1}(\mathcal{L}'_{\gamma',\geq \epsilon})\subseteq \mathcal{L}_{\gamma}$. For all $\ell \in \mathcal{L}_{\gamma,\geq \epsilon}$,
  $\rho(\gamma,\gamma')+ d_\infty(\ell,\phi(\ell))\leq \delta+\delta/2\leq \delta(\gamma,\ell)$, thus $\phi(\ell)\in \mathscr{L}_{\gamma'}$. This proves the first inclusion.
  For the second inclusion, let $\ell'\in \mathcal{L}'_{\gamma',\geq \epsilon}$. Since $\phi$ induces an $\epsilon$-isomorphism,  $\ell\coloneqq \phi^{-1}(\ell')$ is well-defined and $d_\infty(\ell,\ell')\leq \epsilon/2$. In particular $t_{\ell}\geq t_{\ell'}-d_\infty(\ell,\ell')\geq \epsilon$, thus $\ell\in \mathcal{L}_{\geq \epsilon/2}$.
  Let us assume by contradiction that $\ell \notin \mathcal{L}_\gamma$. Then \[(\rho\times d_\infty)((\gamma,\ell),(\gamma',\ell'))\leq d'_{\mathcal{R}_0}(\mathcal{X},\mathcal{X}')\leq  \delta\leq \delta'(\gamma,\ell),\]
  and by construction of $\delta'(\gamma,\ell)$ it follows $\ell'\notin \mathscr{L}_{\gamma'}$, which is absurd.
\end{proof}
\begin{lemma}
  \label{le:strongsuited}
  Let $\mathcal{X}=(\gamma,\mathcal{L})\in\Conf$. For all $\epsilon>0$, there exists $\eta=\eta(\mathcal{X})\in(0,\epsilon]$ such that for all $\mathcal{X}'=(\gamma',\mathcal{L}')\in B'_{\eta}(\mathcal{X})$, all $\eta$-isomorphism $\phi$ from $\mathcal{X}$ to $\mathcal{X}'$ induces an $\epsilon$-strongly suited isomorphism (with respect to $(\gamma,\gamma')$).
\end{lemma}
\begin{proof}
  Fix $\mathcal{X}\in \Conf$ and $\epsilon>0$.
  Let $(\ell_i)_{i\in \{1,\dots , k\}}$ be the increasing enumeration of the finite totally ordered set $\mathcal{L}_{\gamma, \geq \epsilon}$. Define
  \[ \eta_1\coloneqq \min \{ \sigma_{\ell_{i+1}}-\sigma_{\ell_{i}} :i\in \{0,\dots, k\} \}>0,\]
  where formally $\sigma_{\ell_0}\coloneqq 0$ and $\sigma_{\ell_{k+1}}\coloneqq t_\gamma$. Recall $(\gamma,\ell)\mapsto \sigma_{\gamma,\ell}$ is continuous  on $\mathscr{L}_{reg}$ and $(\gamma,\ell)\in \mathscr{L}_{reg}$ for all  $\ell \in \mathcal{L}_\gamma$.
  Thus, for all $\ell\in \mathcal{L}_\gamma$, there exists $\eta_{\ell}>0$ such that for all $(\gamma',\ell')$ with $(\rho\times d_\infty)((\gamma',\ell' ),(\gamma,\ell) )<\eta_\ell$, it holds
  $|\sigma_{\gamma,\ell}-\sigma_{\gamma',\ell'}|< \eta_1/2$.
  Let $\eta_2\coloneqq \min \{ \eta_{\ell}: \ell \in \mathcal{L}_{\gamma, \geq \epsilon}\}$, and $\eta=\eta_2 \wedge \delta$, where $\delta>0$ is obtained by applying Lemma~\ref{le:suited} to $\mathcal{X}$ and~$\epsilon$.

  Let $\mathcal{X}'=(\gamma',\mathcal{L}')\in B'_{\eta}(\mathcal{X})$, and $\phi$ be an $\eta$-isomorphism from $\mathcal{X}$ to $\mathcal{X}'$. Since $\eta\leq \delta$, $\phi$ induces a $\delta$-isomorphism which is
  $\delta$-suited. Since $\delta<\epsilon$, the $\epsilon$-isomorphism induced by $\phi$ is also $\epsilon$-suited.
  Furthermore, for all $\ell\in \mathcal{L}_{\gamma, \geq \epsilon}$, it holds
  \[
  (\rho\times d_\infty)((\gamma', \phi(\ell)),(\gamma,\ell) )\leq d'_{\Conf_0}(\mathcal{X}, \mathcal{X}')<\eta_2\leq \eta_\ell, \qquad \mbox{thus}\qquad
  |\sigma_{\gamma,\ell}-\sigma_{\gamma',\phi(\ell)}|\leq \eta_1/2.\]
  Thus, for  $\ell_1,\ell_2\in \mathcal{L}_{\gamma,\geq \epsilon}$, if  $\ell_1\succ_{\gamma} \ell_2$,
  \begin{align*}
  \sigma_{\gamma',\phi(\ell_1)}-\sigma_{\gamma',\phi(\ell_2)}&\geq
  \sigma_{\gamma,\ell_1}-\sigma_{\gamma,\ell_2} -
  |\sigma_{\gamma',\phi(\ell_1)}-\sigma_{\gamma,\ell_1}|
  -|\sigma_{\gamma',\phi(\ell_2)}-\sigma_{\gamma,\ell_2}|\\
  &>\eta- \eta/2-\eta/2=0.
  \end{align*}
  Thus, $\sigma_{\gamma',\phi(\ell_1)}>\sigma_{\gamma',\phi(\ell_2)}$, which implies $\phi(\ell_1)\succ_{\gamma'} \phi(\ell_2)$.
  We have proved the implication
  \begin{equation}
  \label{eq:temp:imp1}
  \ell_1\succ_{\gamma} \ell_2\implies \phi(\ell_1)\succ_{\gamma'} \phi(\ell_2).\end{equation}
  By contraposition, the reciprocal implication is
  \[
  \ell_1\succeq_{\gamma} \ell_2\implies \phi(\ell_1)\succeq_{\gamma'} \phi(\ell_2).
  \]
  This is true because
  $\prec_{\gamma}$ is a total strict order (recall $\mathcal{X}\in \Conf$). Indeed,
  \[
  \ell_1\succeq_{\gamma} \ell_2
  \iff (\ell_1\succ_{\gamma} \ell_2\mbox{ or } \ell_1= \ell_2 )
  \implies (   \phi(\ell_1)\succ_{\gamma'}, \phi(\ell_2) \mbox{ or } \phi(\ell_1)= \phi(\ell_2) )
  \implies \phi(\ell_1)\succeq_{\gamma'}, \phi(\ell_2).
  \]
  Thus, the $\epsilon$-suited isomorphism induced by $\phi$ is $\epsilon$-strongly suited.
\end{proof}
Considering Lemma~\ref{le:strongsuited}, for any $\mathcal{X} \in \Conf$, we will let $\eta=\eta(\mathcal{X}):(0,\infty)\to(0,\infty)$, be a function such that $\eta(\delta) < \delta$ for all $\delta>0$, and any $\eta(\delta)$-isomorphism from $\mathcal{X}$ is $\delta$-strongly suited.

\subsection{Equivalence with weaker distances under boundedness assumption}
\label{sub:weak}
In the following, for $\mathcal{X}=(\gamma,\mathcal{L})\in \Conf_0$ and $\epsilon >0$, we let 
\begin{equation}
\label{eq:def:Tleqepsilon}
T_{\leq \epsilon}(\mathcal{X})\coloneqq \sum_{\ell \in \mathcal{L}_{\gamma,\leq \epsilon}} t_\ell.
\end{equation}
\begin{lemma}
  \label{le:Tbound}
  For $f:[0,\infty)\to [0,\infty)$, let $K_f\coloneqq\{ \mathcal{X}\in \Conf_0: \forall \epsilon>0, \ T_{\leq \epsilon}(\mathcal{X})\leq f(\epsilon)\}$.
  For $r>0$ and $\mathcal{X}\in\Conf$, let $\tilde{B}_{r}(\mathcal{X})$ be the ball in $\Conf_0$ with radius $r$ centred at $\mathcal{X}$ for the distance $\rho(\gamma,\gamma')+  d( \mathcal{L},\mathcal{L}'  )+d_{\infty,[0,1]}(\omega(\mathcal{L}),\omega(\mathcal{L}') )$.

  Then, for all $\mathcal{X}\in\Conf$,
  \begin{itemize}
  \item For all continuous function $f$ with $f(0)=0$, for all $\epsilon>0$, there exists $\delta>0$ (which depends on $\epsilon$, $f$, and $\mathcal{X}$)
  such that $\tilde{B}_{\delta}(\mathcal{X})\cap K_f\subseteq B_\epsilon(\mathcal{X})$.
  \item There exist continuous functions $(f_\epsilon)_{\epsilon>0}$ (which depend on $\mathcal{X}$) such that $f_\epsilon(\epsilon)\underset{\epsilon\to 0}\longrightarrow 0$ and such that
  for all $\mathcal{X}'\in \Conf_0$, if there exists an $\epsilon$-suited isomorphism $\phi$ from $\mathcal{X}$ to $\mathcal{X}'$, then $\mathcal{X}'\in  K_{f_\epsilon} $.

  In particular, $B_{\eta(\epsilon)} (\mathcal{X})\subseteq K_{f_\epsilon} $.
  \end{itemize}
\end{lemma}
\begin{remark}
  On the other hand, there does not even exist any $\mathcal{X}\in \Conf_0$, $\epsilon>0$ and $f<\infty$ such that $\tilde{B}_{\eta(\epsilon)}(\mathcal{X})\subset K_f$.  Indeed, we can always and easily find $\mathcal{X}'$ at distance less than $\epsilon$ from $\mathcal{X}$ and such that the very small loops accumulate a total time larger than $\|f\|_{\infty,[0,1]}$.
\end{remark}
\begin{proof}[Proof of Lemma~\ref{le:Tbound}]
  For the first point, fix $\epsilon>0$. Let $\delta_0=\epsilon/4$, let  $\delta_1>0$ such that $T_{\leq\delta_1}(\mathcal{X})\leq \epsilon/4$ (which exists since $t_{\mathcal{X}}<\infty$), let $\delta_2>0$ such that $f(\delta_2)\leq \epsilon/4$, let $\epsilon'=\min(\delta_1,\delta_2)$,
  let $\delta_3 = \epsilon/ (4 \#\mathcal{L}_{\gamma,\geq \epsilon'}  )$, and $\delta=\min(\delta_0,\delta_1,\delta_2,\delta_3)$. Let $\eta>0$ such that any $\eta$-isomorphism is $\delta$-suited.
  Let $\mathcal{X}'\in \tilde{B}_{\eta}(\mathcal{X})\cap K_f$, and let $ \phi$ be an $\eta$-isomorphism. Then,
  \begin{align*}
  |t_{\mathcal{X}'}-t_{\mathcal{X}}|&\leq
  \Big|\sum_{\ell \in \mathcal{L}_{\gamma, \geq \epsilon'}} (t_\ell -t_{\phi(\ell)})\Big|+ \Big|\sum_{\ell \in \mathcal{L}_{\gamma, \leq \epsilon'}} t_\ell \Big|
  + \Big|\sum_{\ell' \in \mathcal{L}'_{\gamma', \leq \epsilon'}} t_\ell \Big|\\
  &\leq \delta_3 \cdot \#\mathcal{L}_{\gamma,\geq \epsilon'} + T_{\leq \delta_1}(\mathcal{X})+ T_{\leq \delta_2}(\mathcal{X}')\leq 3\epsilon/4.
  \end{align*}
  Since $\delta\leq \delta_0=\epsilon/4$,
  $\rho(\gamma,\gamma')+
  d( \mathcal{L},\mathcal{L}'  )+d_{\infty,[0,1]}(\omega(\mathcal{L}),\omega(\mathcal{L}') )\leq \epsilon/4$, thus $d_{\Conf_0}(\mathcal{X},\mathcal{X}' )\leq \epsilon$, which proves the first point.

  For the second point, let $\delta:(0,1]\to (0,1]$ be a continuous non-decreasing function of $\epsilon$ which goes to $0$ with $\epsilon$, but sufficiently slowly that $\# \mathcal{L}_{\gamma, \geq \delta(\epsilon) }\leq \epsilon^{-\frac{1}{2}}$. Up to replacing $\delta$ with $\epsilon\mapsto \max(\epsilon,\delta(\epsilon))$, we can further assume $\delta(\epsilon)\geq \epsilon$. Fix $T:[0,\infty]\to [0,\infty)$ such that $T(0)=0$, $T$ is continuous, and $T(u)\geq T_{\leq u}(\mathcal{X})$ for all $u\geq 0$ (e.g. $T(u)=\frac{1}{u}\int_u^{2u} T_{\leq s}(\mathcal{X})\d s $).
  The function $f_\epsilon$ will be given by
  \[
  f_\epsilon(u)=\epsilon+\sqrt{\epsilon}+T(u)+u/\sqrt{\epsilon}.
  \]

  Let $\epsilon>0$ and $\delta\in(0,\epsilon)$ such that all $\delta$-isomorphism are $\epsilon$-suited. Let $\mathcal{X}'$ be such that an $\epsilon$-suited isomorphism $\phi$ exists. Since $\delta\geq \epsilon$, $\phi$ is $\delta$-suited as well.
  For $u>0$,
  \begin{align*}
&  T_{\leq u}(\mathcal{X}')=
  t_{\mathcal{X}'}-\sum_{\ell'\in \mathcal{L}'_{\gamma', \geq u}   } t_{\ell'}
  \\
  &= t_{\mathcal{X}'}- t_{\mathcal{X}}
  +\sum_{\ell \in \mathcal{L}_{\gamma,\leq \delta} } t_\ell+
  \sum_{\ell \in \mathcal{L}_{\gamma,\geq \delta} } (t_\ell-t_{\phi(\ell)})+
  \sum_{\ell \in \mathcal{L}_{\gamma,\geq \delta} } t_{\phi(\ell)}-\sum_{\ell'\in \mathcal{L}'_{\gamma', \geq u}   } t_{\ell'}\\
  &\leq \epsilon + T_{\leq\delta}(\mathcal{X})
  + \epsilon \cdot  \#  \mathcal{L}_{\gamma,\geq \delta} +
  \sum_{\ell' \in \phi(\mathcal{L}_{\gamma,\geq \delta})\setminus    \mathcal{L}'_{\gamma', \geq u}   } t_{\ell'}\\
  &\leq \epsilon+ T_{\leq\delta}(\mathcal{X})
  + \epsilon  \#  \mathcal{L}_{\gamma,\geq \delta} +
  u\cdot \#  \phi(\mathcal{L}_{\gamma,\geq \delta}) \quad \mbox{since $\phi(\mathcal{L}_{\gamma,\geq \delta})
  \setminus    \mathcal{L}'_{\gamma', \geq u}\subseteq \mathcal{L}'_{\gamma'}
  \setminus    \mathcal{L}'_{\gamma', \geq u}=  \mathcal{L}'_{\gamma',< u} $}
	\\&\leq f_\epsilon(u).
  \end{align*}
The inclusion    $\phi(\mathcal{L}_{\gamma,\geq \delta})\subseteq   \mathcal{L}'_{\gamma'} $ is due to the fact that $\phi$ is $\delta$-suited.
\end{proof}
\begin{lemma}
\label{le:omegabound}
  For a function $f:[0,1)\to [0,\infty)$, let $J_f\coloneqq \{ \mathcal{X}=(\gamma,\mathcal{L})\in \Conf_0: \forall \epsilon>0,\ \omega_\epsilon(\mathcal{L}))\leq f(\epsilon)\}$.
  For $r>0$ and $\mathcal{X}\in\Conf$, recall $B'_{r}(\mathcal{X})$ be the ball in $\mathcal{R}_0$ with radius $r$ centred at $\mathcal{X}$ for the distance $\rho(\gamma,\gamma')+d( \mathcal{L},\mathcal{L}'  )$ of \eqref{eq:defdprime}, and let $\tilde{B}_r(\mathcal{X})$ as in Lemma~\ref{le:Tbound}.
  Then,  for all $\mathcal{X}\in\Conf$,
  for all continuous function $f$ with $f(0)=0$, for all $\epsilon>0$, there exists $\delta>0$ (which depends on $\epsilon$, $f$, and $\mathcal{X}$)
  such that $B'_{\delta}(\mathcal{X})\cap J_f\subseteq \tilde{B}_\epsilon(\mathcal{X})$.
\end{lemma}
\begin{proof}
  Fix $\epsilon>0$. By the equicontinuity condition~\eqref{it:unifcont}, there exists $\delta_1>0$ such that $\omega_{\delta_1}(\mathcal{L})<2\epsilon/3$. By continuity of $f$ and since $f(0)=0$, there exists $\delta_2>0$ such that $f(\delta_2)<2\epsilon/3$. Let $\delta\coloneqq \min(\epsilon/3,\delta_1,\delta_2)$.
  Let $t\in[0,1]$
  and $\mathcal{X}'=(\gamma',\mathcal{L}')\in B'_{\delta}(\mathcal{X})\cap J_f$, and let $\phi$ be a $\delta$-isomorphism.

  First assume there exists $\ell \in \mathcal{L}_{\geq \delta}$ such that  $\omega_t(\mathcal{L})=\omega_t(\ell )$.
  Let then $\ell'=\phi(\ell)$. Then,
  \[\omega_t(\mathcal{L}')\geq \omega_t(\ell')\geq \omega_t(\ell)-2 d_\infty(\ell,\ell')> \omega_t(\mathcal{L})-2 \delta\geq \omega_t(\mathcal{L})-2\epsilon/3.\]
  If such an $\ell$ does not exist,  then $\omega_t(\mathcal{L})=\omega_t(\mathcal{L}_{< \delta})$, and we deduce
   \[\omega_t(\mathcal{L})-\omega_t(\mathcal{L}')\leq\omega_t(\mathcal{L})  =\omega_t(\mathcal{L}_{< \delta})=\omega_\delta(\mathcal{L}_{<\delta})\leq \omega_\delta(\mathcal{L})<2\epsilon/3.
   \]
  In both case we deduce $\omega_t(\mathcal{L})-\omega_t(\mathcal{L}')\leq \epsilon$. For the reverse inequality, let
  first assume there exists $\ell' \in \mathcal{L}'_{\geq \delta}$ such that  $\omega_t(\mathcal{L}')=\omega_t(\ell')$ and let $\ell=\phi^{-1}(\ell')$. Similarly as above we see $\omega_t(\mathcal{L})\geq \omega_t(\mathcal{L}')-\epsilon$. If such an $\ell'$ does not exist, then
  $\omega_t(\mathcal{L}')=\omega_t(\mathcal{L}'_{< \delta})$, and we deduce
   \[\omega_t(\mathcal{L}')-\omega_t(\mathcal{L})\leq\omega_t(\mathcal{L}')  =\omega_t(\mathcal{L}'_{< \delta})=\omega_\delta(\mathcal{L}'_{<\delta})\leq \omega_\delta(\mathcal{L}')\leq f(\delta)\leq 2\epsilon/3.\]

  Thus, $|\omega_t(\mathcal{L}')-\omega_t(\mathcal{L})|\leq 2\epsilon/3$. Since $\rho(\gamma,\gamma')+d( \mathcal{L},\mathcal{L}'  )< \delta\leq \epsilon/3$, we deduce indeed that $\mathcal{X}'\in \tilde{B}_\epsilon(\mathcal{X})$.
\end{proof}

\subsection{Continuity of the hitting times, local equicontinuity of \texorpdfstring{$\Xi$}{X}}

\begin{lemma}
\label{le:Tcont}
  For all $\mathcal{X}=(\gamma,\mathcal{L})\in \Conf$, for all $\epsilon>0$, there exists $\delta>0$ such that for all
  $\mathcal{X}'\in \Conf_0$ and $\delta$-strongly suited isomorphism $\phi$ from $\mathcal{X}$ to $\mathcal{X}'$,
  for all $\ell\in \mathcal{L}_{\gamma,\geq \delta}$,
  \[|T_\ell(\mathcal{X})-T_{\phi(\ell)}(\mathcal{X}')|\leq \epsilon. \]
\end{lemma}
The idea behind this lemma is that,
when $\mathcal{X}'$ converges to $\mathcal{X}$, intersections between path and loops cannot disappear in the limit, which leads to the inequality $\liminf T_{\ell}(\mathcal{X}')\geq T_{\phi(\ell)}(\mathcal{X})$. However, since we also have a control over the total time (namely $t_{\cX'} \to t_{\cX}$) the previous inequality must in fact be sharp: any strict inequality would result in $t_{\cX}$ being noticeably different from $t_{\cX'}$. 
\begin{proof}
  First, remark  \[ \delta\cdot \#   \mathcal{L}_{\gamma,\geq \delta} \leq \sum_{\ell\in \mathcal{L}_\gamma} (t_\ell \wedge \delta)\underset{\delta\to 0}\longrightarrow 0,\]
  by dominated convergence theorem, since $t_{\mathcal{X}}<\infty $. Thus, for all $\epsilon>0$, for all $\delta>0$ small enough, it holds
  \[T_{\leq\delta}(\mathcal{X})+ \frac{\delta}{2}\cdot \#   \mathcal{L}_{\gamma,\geq \delta} + \delta+f_\delta(\delta) \leq \epsilon,\] where $f_\delta$ is the function from the second point in Lemma~\ref{le:Tbound}.
  Let $\mathcal{X}'\in \Conf_0$, and  $\phi$ be an $\delta$-suited isomorphism. By Lemma~\ref{le:Tbound}, $T_{\leq\delta}(\mathcal{X}')\leq f_\delta(\delta)$ and $T_{\leq\delta}(\mathcal{X})\leq f_\delta(\delta)$.

  On the one hand, for any $\ell'\prec \ell$ with  $t_{\ell'}\geq \delta $, it holds $\phi(\ell')\in \mathcal{L}'_{\prec \phi(\ell)} $, thus
  \begin{align*}
  T_\ell(\mathcal{X})-T_{\phi(\ell)}(\mathcal{X}')&= \sum_{\tilde{\ell}\in \mathcal{L}_{\prec \ell,< \delta} } t_{\tilde{\ell}}
  + \sum_{\tilde{\ell}\in \mathcal{L}_{\prec \ell,\geq \delta} } t_{\tilde{\ell}}-
  \sum_{\ell'\in \mathcal{L}'_{\prec \phi(\ell)}   } t_{\ell'}\\
  &\leq T_{\leq \delta}(\mathcal{X})+  \sum_{\tilde{\ell}\prec \ell, t_{\tilde{\ell}}> \delta} (t_{\tilde{\ell}}-t_{\phi(\tilde{\ell})})\\
  & \leq T_{\leq \delta}(\mathcal{X})+ \frac{\delta}{2} \# \mathcal{L}_{\gamma,> \delta}.
  \end{align*}

  On the other hand,
  \begin{align*}
  T_{\phi(\ell)}(\mathcal{X}')-T_\ell(\mathcal{X}) &=
  t_{\mathcal{X}'}-t_{\mathcal{X}}-
  \sum_{\tilde{\ell}\in \mathcal{L}_{\succeq \ell} } t_{\tilde{\ell}}
  +\sum_{\ell'\in \mathcal{L}'_{\succeq \phi(\ell),\geq \delta}  } t_{\ell'}
  +
  \sum_{\ell'\in \mathcal{L}'_{\succeq \phi(\ell),\leq \delta}   } t_{\ell'}
  \\
  &\leq \delta +  \sum_{\ell'\in \mathcal{L}'_{\succeq \phi(\ell), \geq \delta} } (t_{\ell'}-t_{\phi^{-1}(\ell')})+f_\delta(\delta) \qquad \leq \delta + \frac{\delta}{2} \# \mathcal{L}_{\gamma,\leq \delta}+f_\delta(\delta).
  \end{align*}
  Thus $|T_\ell(\mathcal{X})-T_{\phi(\ell)}(\mathcal{X}')|\leq \delta +  T_{\leq \delta}(\mathcal{X})+ \frac{\delta}{2} \# \mathcal{L}_{\gamma,\leq \delta}+f_\delta(\delta)\leq
  \epsilon$.
\end{proof}

\begin{lemma}
  \label{le:omega1bound}
  Recall the definition of $\omega^1$ from~\eqref{eq:def:omegadelta1}.

  Let $\mathcal{X}\in \Conf$. Then, for all $\epsilon>0$, there exists $\delta>0$ such that for all $\mathcal{X}'\in \Conf_0$, if there exists a $\delta$-strongly suited isomorphism $\phi$  from $\mathcal{X}$ to $\mathcal{X}'$, then $\omega^1_\delta(\mathcal{X}')<\epsilon$.
\end{lemma}
\begin{proof}
  Let $\mathcal{X}\in \Conf$ and $\epsilon>0$. Let $\delta_1>0$ such that $\omega^1_{3 \delta_1/2}(\mathcal{X})<\epsilon/2$, which exists by density of $\sigma_{\mathcal{L}}$.
  Let $\delta \in(0, \delta_1)$ such that for all $\ell\in \mathcal{L}_{\gamma,\geq \delta_1}$, for all $(\gamma',\ell')$ at distance less than $\delta$ from $(\gamma,\ell)$, it holds $|\sigma_{\gamma',\ell'}-\sigma_{\gamma,\ell}|<\epsilon/2$. The existence of $\delta$ follows from the fact $(\gamma,\ell)\in \mathscr{L}_{reg}$ for all $\ell \in \mathcal{L}$ (by definition of $\Conf$)
  and the continuity on $\mathscr{L}_{reg}$ of the function $(\gamma',\ell')\in \mathscr{S}\times\mathscr{L}\mapsto \sigma_{\gamma',\ell'}$  (Theorem~\ref{th:rerootingcontinuity}).

  Then, for all $\mathcal{X}'\in \Conf_0$ such that there exists a $\delta$-strongly suited isomorphism $\phi$  from $\mathcal{X}$ to $\mathcal{X}'$,
  let $(\ell'_i)_{i\in \{1,\dots, k\}}$ be an increasing enumeration of $\mathcal{L}'_{\gamma,\geq \delta}$ and $\ell_i=\phi^{-1}(\ell'_i)$. Since $\mathcal{L}_{\gamma,\geq 3\delta/2}\subseteq \phi^{-1}(\mathcal{L}'_{\gamma',\geq \delta})=\{ \ell_i: i\in \{1,\dots, k\} \}$,
  \[
  \omega_\delta^1(\gamma',\mathcal{L}')= \max \{ \sigma_{\gamma',\ell'_{i+1}}- \sigma_{\gamma',\ell'_{i+1}}\}\leq \epsilon/2+
  \max \{ \sigma_{\gamma,\ell_{i+1}}- \sigma_{\gamma,\ell_{i+1}}\}< \epsilon/2+ \omega_{3\delta/2}^1(\gamma,\mathcal{L})<\epsilon.\qedhere
  \]
\end{proof}

\begin{proposition}
  \label{prop:equicont}
  For all $\mathcal{X}\in \Conf$, for all $\epsilon>0$, there exists $\delta>0$ such that  for all $\lambda>0$ and
  $\mathcal{X}'=(\gamma',\mathcal{L}')\in \Conf_0$ and $\delta$-strongly suited isomorphism $\phi$ from $\mathcal{X}$ to $\mathcal{X}'$, for all $B\in \mathcal{B}_{\gamma',\mathcal{L}'}$,
  both $\omega_\delta(\gamma'\circ \sigma'_{\mathcal{X}',\lambda})$ and $\omega_\delta(\Att(\gamma',\mathcal{L}',\lambda,B))$ are smaller than $\epsilon$.
\end{proposition}
\begin{proof}
Let  $\mathcal{X}=(\gamma,\mathcal{L})\in \Conf$ and $\epsilon>0$.
Let $\delta_0>0$ such that $\omega_{\delta_0}(\gamma)<\epsilon/6$. Let $\delta_1>0$ such that for all
$\delta_1$-strongly suited isomorphism $\phi$ from $\mathcal{X}$ to $\mathcal{X}'$, $\omega^1_{\delta_1}(\mathcal{X}')<\eta$, the existence of which follows from Lemma~\ref{le:omega1bound}. Let $\delta\leq \min(\epsilon/8, \delta_1 )$ such that
$\rho(\gamma,\gamma')\leq \delta\implies d_\infty(\gamma,\gamma')<\epsilon/6$ and $\omega_{2\delta}(\mathcal{L})<\epsilon/8$.

Then, if there exists a $\delta$-strongly suited isomorphism $\phi$ from $\mathcal{X}$ to $\mathcal{X}'=(\gamma',\mathcal{L}')$,
\begin{equation}
\label{eq:temp:omegasmall}
\omega_\delta(\gamma'\circ \sigma_{\mathcal{X}',\lambda} )\leq \omega_{\omega_\delta(\sigma_{\mathcal{X}',\lambda}) }(\gamma')
\leq \omega_{\omega^1_\delta( \mathcal{X}')    }(\gamma')
\leq \omega_{\omega^1_\delta( \mathcal{X}')    }(\gamma)+2 d_\infty(\gamma,\gamma')\leq \omega_{\delta_0}(\gamma)+\epsilon/3\leq \epsilon/2,
\end{equation}
where the second inequality comes from Lemma~\ref{le:sigmaIsCont} (Equation~\eqref{eq:def:omegadelta1}).

Furthermore, $\omega^l_\delta(\mathcal{L}')\leq \omega_{2\delta}(\mathcal{L}')\leq \omega_{2\delta}(\mathcal{L})+ d_{\mathcal{R}_0}(\mathcal{X},\mathcal{X}' )\leq
\epsilon/4$. Combining this with~\eqref{eq:temp:omegasmall} and Proposition~\ref{prop:extension}, we deduce $\omega_\delta(\Att(\gamma',\mathcal{L}',\lambda,B))<\epsilon$, which concludes the proof.
\end{proof}

\subsection{Continuity of the attachment map}
\label{sub:cont2}
\begin{theorem}
\label{th:continuityAtt}

	The map $\Att$ is continuous on $\Conf\times\{0\}$, in the following sense:

  For all $\mathcal{X}\in \Conf$, for all $\epsilon>0$, there exists $\delta>0$ such that
  \[
  \forall (\mathcal{X}',\lambda)\in \Conf', \qquad d_{\mathcal{R}_0}(\mathcal{X},\mathcal{X}')\leq \delta \mbox{ and }\lambda\leq \delta\implies \forall B\in \mathcal{B}_{\mathcal{X}'} , \ d_\infty(\Att(\mathcal{X}',\lambda, B) ,\Att(\mathcal{X})   )<\epsilon.
  \]
\end{theorem}
\begin{remark}
  It should be possible to extend Proposition~\ref{prop:extension} to the larger space $\Conf_0\times [0,\infty)$ (i.e. to remove the condition~\eqref{it:dense} when $\lambda=0$) at the cost of replacing ``continuous" with ``càdlàg" in the result. Then, the continuity property in Theorem~\ref{th:continuityAtt} should also extend to the space of pairs which satisfy~\eqref{it:inj} and~\eqref{it:reg} but not necessarily~\eqref{it:dense}, provided we replace the uniform topology on $\mathcal{C}([0,\infty),\mathbb{R}^2)$
  with an appropriate Skorokhod-type topology on the space of càdlàg functions. In summary:~\eqref{it:inj} ensures uniqueness of the path obtained by attachment,~\eqref{it:dense} ensures continuity of this path when we follow $\gamma$ at infinite pace, and~\eqref{it:reg} ensures that the attachment map itself is continuous.


\end{remark}

\begin{proof}[Proof of Theorem~\ref{th:continuityAtt}]
Let $\mathcal{X}=(\gamma,\mathcal{L})\in \Conf$, $X=\Att(\mathcal{X})$, and $\epsilon>0$.
By Lemma~\ref{le:strongsuited}, it suffices to show that there exists $\delta>0$ such that for all $(\gamma',\mathcal{L}',\lambda)\in \Conf'$ with $\lambda<\delta$, if there exists a $\delta$-strongly suited isomorphism $\phi$ from $\mathcal{X}$ to $\mathcal{X}'\coloneqq (\gamma',\mathcal{L}')$, then, for all  $B\in \mathcal{B}_{\gamma',\mathcal{L}'}$ , $ d_\infty(\Att(\gamma',\mathcal{L}',\lambda, B) ,\Att(\gamma,\mathcal{L})   )<\epsilon$.

Let $\delta_0\in(0,\epsilon/4)$ such that $\omega_{6 \delta_0}(\mathcal{L})<\epsilon/8$ and such that for all $\delta_0$-strongly suited isomorphism $\phi$ from $\mathcal{X}$ to $\mathcal{X}'$,
$\omega_{3\delta_0}(X')\leq \epsilon/4$, where $X'=\Att(\mathcal{X}', \lambda,B)$ (in particular, $\omega_{3\delta_0}(X)\leq \epsilon/4$ since we can take $\mathcal{X}'\coloneqq\mathcal{X}$ and $\phi\coloneqq\operatorname{id}$). The existence of $\delta_0$ follows from the equicontinuity condition and Proposition~\ref{prop:equicont}.

Let $\delta_1\in(0,\delta_0)$ such that for all $t\in [0,t_{\mathcal{X}}=t_X]$, there exists $\ell \in \mathcal{L}_{\gamma,\geq \delta_1}$ and $s\in(T_\ell,T_\ell+t_\ell)$ such that $|s-t|\leq \delta_0$. The existence of $\delta_1$ follows from the density of $\mathcal{T}$ in $[0,t_\gamma]$ (Lemma~\ref{le:dense}).

Let $\delta\in(0,\delta_1)$ such that $\delta<\delta_0/(2t_\gamma)$ and for all $\delta$-strongly suited isomorphism $\phi$ from $\mathcal{X}$ to $\mathcal{X}'$, for all $\ell\in \mathcal{L}_{\gamma,\geq \delta_1}$, it holds
$|T_\ell-T_{\phi(\ell)}|\leq \delta_0$ and $|\theta_\ell-\theta_{\phi(\ell)}|\leq \delta_0$. This is possible by Lemma~\ref{le:Tcont} and Theorem~\ref{th:rerootingcontinuity} (continuity of $(\gamma,\ell)\mapsto \theta_{\gamma,\ell}$).

Then, for  all $\delta$-strongly suited isomorphism $\phi$ from $\mathcal{X}$ to $\mathcal{X}'$,
for all $t\in [0,t_X]$, let  $\ell \in \mathcal{L}_{\gamma,\geq \delta_1}$ and $s\in(T_\ell,T_\ell+t_\ell)$ such that $|s-t|\leq \delta_0$, let $\ell'\coloneqq \phi(\ell)$ and let
\[ s'=(s-T_\ell+T_{\ell'})\wedge (T_{\ell'}+t_{\ell'})\in [T_{\ell'},T_{\ell'}+ t_{\ell'}].\]
First, remark
\[|X(t)-X(s)|\leq \omega_{\delta_0}(X)\leq \epsilon/4.\]
Furthermore,  $X(s)={\ell}( s-T_\ell -\theta_{\gamma,\ell} +\mathbbm{1}_{ s-T_\ell\leq \theta_{\gamma,\ell}  } t_\ell
)$ and
$X'(s')=\ell'(s'- T_{\ell'} -\theta_{\gamma',\ell'}  +\mathbbm{1}_{ s'-T_{\ell'}\leq \theta_{\gamma',\ell'}  } t_{\ell'}   )  $.
Let  $\tilde{\ell}$ be the $t_\ell$-periodic extension of $\ell$, and observe that
\[
X(s)=\ell( s-T_\ell -\theta_{\gamma,\ell} +\mathbbm{1}_{ s-T_\ell\leq \theta_{\gamma,\ell}  } t_\ell
)=\tilde{\ell}( s-T_\ell -\theta_{\gamma,\ell} +\mathbbm{1}_{ s'-T'_{\ell'}\leq \theta_{\gamma',\ell'}  } t_\ell
).
\]
Thus,
\begin{align*}
|X(s)-X'(s')|
&\leq \omega^l_{ |s'-s| +|\theta_{\gamma,\ell}-\theta_{\gamma',\ell'}| +|t_\ell-t_{\ell'}| }(\ell)+\|\ell'-\tilde{\ell}\|_{\infty,[0,t_{\ell'}] }\\
&\leq 2 \omega^l_{ |s'-s|+|\theta_{\gamma,\ell}-\theta_{\gamma',\ell'}| +|t_\ell-t_{\ell'}| }(\ell)+d_\infty(\ell',\ell)\leq   2\omega_{4\delta_0+2\delta }(\mathcal{L}) +\delta\leq \epsilon/2.
\end{align*}
Furthermore, $|s-s'|\leq \delta_0+\delta/2<2\delta_0$:
either $s'=s-T_\ell+T_{\ell'}$, in which case $|s'-s|=| T_{\ell'}-T_\ell|\leq \delta_0 $, or
$s'=T_{\ell'}+t_{\ell'}$, which happens only if $T_{\ell'}+t_{\ell'}\leq s-T_\ell+T_{\ell'}$, i.e. $s\geq T_\ell+t_{\ell'}$. In this case, $s\in[T_\ell+t_{\ell'},T_\ell+t_{\ell}]$, thus $s-s'\in [T_{\ell}-T_{\ell'},T_{\ell}-T_{\ell'}+t_\ell-t_{\ell'}]$, thus $|s-s'|\leq | T_{\ell'}-T_\ell|+| t_{\ell'}-t_\ell|\leq \delta_0+\delta/2$.

It follows that $|t-s'|\leq 3\delta_0$. Since $\omega_{3\delta_0}(X')\leq \epsilon/4$,
\[ |X'(s')-X'(t)|\leq  \epsilon/4,\qquad \mbox{thus}  \qquad   |X(t)-X'(t)|\leq  \epsilon.\]
Thus,
$\|X-X'\|_{\infty,[0,t_X]}\leq \epsilon$. Since $X'$ is uniformly continuous in the vicinity of $X$ and $t_{X'}=t_{\mathcal{X}'}$ converges toward $t_X=t_{\mathcal{X}}$ as $(\gamma',\mathcal{X}',\lambda)$ converges toward $(\gamma,\mathcal{X},0)$ in $\mathcal{R}'$, this is sufficient to conclude.
\end{proof}

\section{Probabilistic aspects: \texorpdfstring{$(\gamma,\mathcal{L})\in \Conf$}{(gamma,L) in R} and fixed lattice approximation}
\label{sec:probabilistic1}
\subsection{Notations and preliminary properties}
In the following, we fix a proper simply connected planar domain $D$ with $0\in D$. We recall that $\mu^{\mathcal{L}}$ is the probability distribution of $\mathcal{L}$ a rooted and oriented Brownian loop soup with intensity $1$ on $D$. The whole-plane Brownian loop soup is written $\mathcal{L}^{\mathbb{R}^2}$. We recall also that $\mu^{\gamma}$ is the probability distribution of $\gamma$ a radial $SLE_2$ from the boundary of $D$ to $0$, run backward in time, and $\mathcal{X}\coloneqq (\gamma,\mathcal{L})\sim \mu^{\gamma}\otimes \mu^{\mathcal{L}}$.

We let $\mathbb{L}^{\mathbb{Z}^2}$ be a rooted and oriented random walk loop soup with intensity $1$ on the square lattice. For an integer $n$, we let $\mathbb{L}^{nD}$ be the subset of $\mathbb{L}^{\mathbb{Z}^2}$ which consists of loops whose edges are all contained  inside $nD$. It should be understood that for $n\neq n'$, the sets $\mathbb{L}^{nD}$ and $\mathbb{L}^{n'D}$ are not issued from the same whole-lattice loop soup $\mathbb{L}^{\mathbb{Z}^2}$,
but instead by two copies of it, possibly coupled in a non-trivial way (especially if we wish to consider convergence in probability after rescaling, as $n\to\infty$).
We also let $\gamma^{nD}$ be a loop-erased random walk on the square lattice, from $0$ to $\partial D$. It is always implicitly assumed that $\gamma^{nD}$ and $\mathbb{L}^{nD}$ are independent from each other.
Both $\gamma^{nD}$ and elements of $\mathbb{L}^{nD}$ can be considered either as continuous curves, parametrised at unit speed, or as the finite sequence of their values at integer times. Given an integer $n$, we will consider the scaling functions
\[
\begin{array}{cccc} \phi^*_{n}:& \mathscr{S} &\to & \mathscr{S}, \\ &\mathsf{S}&\mapsto &  n^{-1}\mathsf{S}( c_* n^{\frac{5}{4}}  \ \cdot \ )\end{array} \qquad
\begin{array}{cccc}\phi_n:& \mathcal{C}([0,\cdot],\mathbb{R}^2)&\to & \mathcal{C}([0,\cdot],\mathbb{R}^2) \\ &\mathsf{L}&\mapsto &  n^{-1}\mathsf{L}(2 n^2  \ \cdot \   )\end{array},
\]
where $c_*>0$ is a universal constant which is defined by the property that $\phi^*_n(\gamma^{nD})$ converges in distribution in $\rho$-distance toward $\gamma$ (recall that here $\gamma$ is parametrised by its Minkowski content or ``natural parametrisation'').
Then, we define \[ \mathcal{X}^n\coloneqq ( \phi^*_n(\gamma^{nD}), \phi_n( \mathbb{L}^{nD}  )   ) \qquad \mbox{and} \qquad \lambda_n\coloneqq \frac{c_*}{2}n^{-\frac{3}{4}}.\] Remark that $\lambda_n\to 0$ as $n\to \infty$.

Given $\mathcal{X}=(\gamma,\mathcal{L})\in \mathscr{S}\times \mathscr{L}^{lf}$ and $a>0$, let $a\cdot \mathcal{X}\coloneqq(a \gamma, \{ a \ell: \ell\in \mathcal{L}\})$. Then,
there is a canonical bijection between $\mathcal{B}_{\mathcal{X}}$ and  $\mathcal{B}_{a\cdot \mathcal{X}}$, and if we identify $B\in \mathcal{B}_{\mathcal{X}}$
with the corresponding element in $ \mathcal{B}_{a \cdot\mathcal{X}}$, the map $\Att$ commutes with multiplication by $a$, i.e.
\begin{equation}
\Att(a\cdot \mathcal{X}, \lambda, B    )=a \Att(\mathcal{X},\lambda, B    ).
\end{equation}
Furthermore, for $c>0$, if $S_c: \mathcal{C}([0,\cdot],\mathbb{R}^2 )\to \mathcal{C}([0,\cdot],\mathbb{R}^2 )$ is the operator $f\mapsto (t\mapsto f(c t ))$,
and $S_c(\mathcal{X})\coloneqq(S_c(\gamma),S_c(\mathcal{L})\coloneqq \{S_c(\ell):\ell\in \mathcal{L}\})$, there is also a canonical bijection between $\mathcal{B}_{\mathcal{X}}$ and  $\mathcal{B}_{S_c (\mathcal{X}) }$, and for this identification the map $\Att$ also commutes with $S_c$:
\begin{equation}
\Att(S_c(\mathcal{X}), \lambda, B    )=S_c ( \Att(\mathcal{X},\lambda, B    ) ).
\end{equation}
Furthermore,
\begin{equation}
\Att( S_c(\gamma), \mathcal{L},\lambda,B )= \Att( \gamma, \mathcal{L},c^{-1}  \lambda,B ).
\end{equation}
Of course the operators of scaling in space and in time commute with each other: for $a>0$ and $\lambda>0$, $S_c(a\cdot \mathcal{X})=a\cdot S_c( \mathcal{X})$, and the identification between $\mathcal{B}_{\mathcal{X}}$ and $\mathcal{B}_{S_c(a\cdot \mathcal{X})}=\mathcal{B}_{a\cdot S_c( \mathcal{X})}$ does not depend on the order in which we perform the two operations.

It follows from these considerations, and the choice of $\lambda_n$, that we have the equality
\begin{equation}
\label{eq:scaling}
\phi_n( \Att({\gamma}^{nD},{\mathbb{L}}^{nD},1,B  )) =
  \Att(\phi^*_n(\gamma^{nD}),\phi_n(\mathbb{L}^{nD}),\lambda_n,B ).
\end{equation}

For lattice path, the procedure of chronological loop-erasure is well-defined. We let $LE(X)$ be the loop erasure of a lattice path $X$. In particular,
\begin{equation}
LE(\Att(\mathsf{S}, \mathbb{L},1,B)  )=\mathsf{S}.
\end{equation}

We let $\lambda$ be the two-dimensional Lebesgue measure. For $X\subset D$, we let $X^\delta$ be the $\delta$-thickening of $X$, that is $X^\delta\coloneqq\{x: d(x,X)\leq \delta\}$. The notion of dimension will always refer to upper Minkowski dimension, e.g. a subset of $\mathbb{R}^ 2$ has dimension $d$ if
\[
\limsup \frac{\ln\ \lambda(X^\epsilon)  }{\ln\ \epsilon }= 2-d.
\]
In particular, $\mu^\gamma$-almost surely, $\Range(\gamma)$ has dimension $5/4$.

\subsection{Summary of the strategy}
In this section, we  will prove the two following results.
\begin{proposition}
\label{prop:conf}
  Let $\gamma\in \mathscr{S}$ deterministic and assume that $\Range(\gamma)$ has dimension $d<2$. Then, $ \mu^{\mathcal{L}}$-almost surely, $(\gamma,\mathcal{L})\in \Conf$.

  In particular, $\mu^\gamma\otimes \mu^{\mathcal{L}}$-almost surely, $(\gamma,\mathcal{L})\in \Conf$.
\end{proposition}
In other words, the conditions imposed by the space $\Conf$, albeit sufficient restrictive to ensure continuity of the attachment map, are sufficiently generic to hold with probability $1$.

\begin{proposition}
  \label{prop:RW=LERW+RWLS}
  Let $\Lambda$ be a sublattice of the square lattice on $\mathbb{Z}^2$, which contains $0$.
  Let $\gamma$ be a loop-erased random walk started from $0$ and stopped when it exits $\Lambda$ (by edge). Let $\mathbb{L}^{\mathbb{Z}^2}$ be a random walk loop soup independent from $\gamma$, and
  $\mathbb{L}$ be the restriction of $\mathbb{L}^{\mathbb{Z}^2}$ to loops contained (by edges) in $\Lambda$.
  Then , the set $\mathcal{B}_{\gamma,\mathbb{L}}$ is finite. Let $B\in \mathcal{B}_{\gamma,\mathbb{L}}$ be chosen uniformly at random, independently from $(\gamma, \mathbb{L})$
  conditionally on $\mathcal{B}_{\gamma, \mathbb{L}}$.

  Then, $\Att(\gamma,\mathbb{L},1,B)$ is distributed as a simple random walk started from $0$ and stopped when it exits $\Lambda$ (by edge).
\end{proposition}
In the next section, we will prove the following.
\begin{proposition}
  \label{prop:converg}
  As $n\to \infty$,
  $\mathcal{X}^n$ converges in distribution in $d_{\Conf_0}$-distance toward $\mathcal{X}$.
\end{proposition}
To prove this, we will rely on known results ensuring the convergence of $\phi_n^*(\gamma^{nD})$ toward $\gamma$ (for the distance $\rho$) and the convergence of $\phi_n(\mathbb{L}^{nD})$ toward $\mathcal{L}$ (for the distance $d$). We will have to prove, then, convergence of $\mathcal{X}^n$ toward $\mathcal{X}$, in the distance $d_{\Conf_0}$, and in particular that $t_{\mathcal{X}^n}$ converges toward $t_{\mathcal{X}}$ (which is the difficult part).
%
%

Assuming these three results holds, we can easily deduce the following.
\begin{theorem}
  \label{th:main}
  There exist couplings of $(\gamma^{nD},\mathbb{L}^{nD},S^{nD}\coloneqq \Att(\gamma^{nD},\mathbb{L}^{nD},1,B^n ))$ over the positive integers $n$, with
$\gamma^{nD}$, $\mathbb{L}^{nD}$ and $B^n$ distributed as above, such that the following holds.
  \begin{itemize}
  \item For all $n\in \mathbb{N}$, $S^{nD}$ is a simple random walk started from $0$ and stopped when it first reaches $\partial D$.
  \item As $n\to \infty$
  $(\phi^*_n(\gamma^{nD}),\phi_n(\mathbb{L}^{nD}) )$ converges almost surely toward $(\gamma,\mathcal{L})\sim \mu^\gamma\otimes \mu^{\mathcal{L}}$ for the topology of $d_{\Conf_0}$. In particular,
  $\phi^*_n(\gamma^{nD})$ converges almost surely toward $\gamma$ and
  $\phi_n(\mathbb{L}^{nD})$ converges almost surely toward $\mathcal{L}$ (for the respective distances $\rho$ and $d$ defined in~\eqref{eq:def:dist})
  \item As $n\to \infty$, $\phi_n(S^{nD})$
  converges almost surely for the topology of $d_\infty$.  The limit $W$ is distributed as a Brownian motion started from $0$ and stopped when it first reaches $\partial D$, and $W$ is almost surely equal to $\Att(\gamma, \mathcal{L})$.
  \end{itemize}
  In particular, if $(\gamma, \mathcal{L})\sim \mu^\gamma\otimes \mu^{\mathcal{L}}$, then
  $\Att(\gamma, \mathcal{L})$ is distributed as a Brownian motion started from $0$ and stopped when it first reaches $\partial D$.
\end{theorem}
\begin{proof}[Proof of Theorem~\ref{th:main} assuming Theorems 
\ref{prop:conf},~\ref{prop:RW=LERW+RWLS}  ,\ref{prop:converg} ]
  Let $\gamma^{nD}$, $\mathbb{L}^{nD}$, $B^n$ as above. By Proposition~\ref{prop:converg}, the couple $(\gamma^{nD},\mathbb{L}^{nD})$ converges in distribution, for the distance $d_{\Conf_0}$, toward $(\gamma,\mathcal{L})$. By Skorokhod's almost sure representation theorem, there exists a probability space on which the $(\gamma^{nD},\mathbb{L}^{nD})$
  are simultaneously defined for all $n\geq 1$, as well as $\gamma$ and $\mathcal{L}$, and such that $(\gamma^{nD},\mathbb{L}^{nD})\to (\gamma,\mathcal{L})$ almost surely as $n\to \infty$ (for  $d_{\Conf_0}$).
  Up to enlarging this probability space, we can assume it also contains the random element $B\in\mathcal{B}_{\gamma^{nD}, \mathbb{L}^{nD}} $. By
  Proposition~\ref{prop:RW=LERW+RWLS}, $S^n$ is then a simple random walk on the lattice, started from $0$ and stopped when it first exits $nD$. By~\eqref{eq:scaling}, the rescaled and reparametrised random walk $\phi_n(S^n)$ is equal to $\Att(\phi_n^*(\gamma^{nD}),\phi_n(\mathbb{L}^{nD}),\lambda_n,B^n)$, with $\lambda_n\underset{n\to \infty}\longrightarrow 0$.

  By Proposition~\ref{prop:conf}, almost surely, $ (\gamma, \mathcal{L}) \in \Conf$, so that $W\coloneqq \Att(\gamma, \mathcal{L})=\Att(\gamma, \mathcal{L},0)$ is almost surely well-defined. Since the restricted map $ X:(\gamma, \mathcal{L}) \in \Conf\mapsto  \Att(\gamma, \mathcal{L},0)$ is continuous, it is measurable. Thus, $W$ is a genuine random variable.

  Since, almost surely, $ (\phi_n^*(\gamma^{nD}),\phi_n(\mathbb{L}^{nD}),\lambda_n)\longrightarrow (\gamma,\mathcal{L},0)\in \Conf\times \{0\}$, 
  Theorem~\ref{th:continuityAtt} ensures that, almost surely, $\phi_n(S^n)=\Att(\phi_n^*(\gamma^{nD}),\phi_n(\mathbb{L}^{nD}),\lambda_n,B^n)$ converges almost surely for $d_\infty$ toward $W\coloneqq \Att(\gamma,\mathcal{L})=W$.

  Since $S^n$ is a simple random walk, Donsker's theorem ensures that $\phi_n(S^n)$ converges in distribution for $d_\infty$ toward a Brownian motion started from $0$ and stopped when it first exits $D$. By unicity of the limit in distribution, $W$ is distributed as a Brownian motion started from $0$ and stopped when it first exits $D$.
\end{proof}
\subsection{Properties~\texorpdfstring{\eqref{it:tmax}}{(i)} to~\texorpdfstring{\eqref{it:dense}}{(iv)} for the Brownian loop soup}

In this section, we prove in Lemma~\ref{le:item:tmax},~\ref{le:item:unifcont},~\ref{le:item:inj},~\ref{le:item:dense}, that for any simple path $\gamma\in \mathscr{S}$
such that $\Range(\gamma)$ has dimension less than $2$, and such that $\gamma((0,t_\gamma))\subset D$, for $\mathcal{L}$ a Brownian loop soup on
$D$, the pair $(\gamma,\mathcal{L})$ almost surely satisfies the properties~\eqref{it:tmax},~\eqref{it:unifcont},~\eqref{it:inj}, and~\eqref{it:dense}. The last property~\eqref{it:reg} will require a bit of extra work and is relegated to the next subsection.
As remarked in \cite{LawlerWerner}, the finiteness of the total time $t_{\gamma,\mathcal{L}}$ follows from the fact that $\gamma$ as dimension less than~$2$. We include a short proof for completeness.
\begin{lemma}
    \label{le:tech1}
    Let $X\subset\mathbb{R}^2$ be a set with dimension $d<2$.
    Then,  as $t\to 0$,
    \[
    \int_D \mathbb{P}_{t,x,x}(\Range(\ell)\cap X\neq \emptyset  ) \d x
    \leq t^{1-d/2 +o(1)   }.
    \]
\end{lemma}
\begin{proof}
    Let $d'>d$, so that $\lambda(X^{\epsilon})\epsilon^{d'-2}$ remains bounded as $\epsilon \to 0$.   Choose $\epsilon>0$, and let $\delta=\delta(t)=t^{\frac{1}{2}-\epsilon}$. We decompose the integral over $D$ into $X^\delta \cup ( D\setminus X^\delta )$.
    On the one hand, we have
    \[ \int_{ X^\delta} \mathbb{P}_{t,x,x}(\Range(\ell)\cap X\neq \emptyset ) \d x
    \leq \lambda(X^\delta)=O(\delta^{2-d' } ) = o(t^{ 1-d'/2  -2\epsilon } ). \]

    On the other hand, for $x\notin  \Range(\gamma)^\delta$, the probability $\mathbb{P}_{t,x,x}( \Range(\ell)\cap X\neq \emptyset)$ is smaller than
    \[\mathbb{P}_{t,x,x}(\exists s\in [0,t]: |\ell_s-x|\geq \delta )=
    \mathbb{P}_{1,0,0}(\exists s\in [0,1]: |\ell_s|\geq t^{-\epsilon}  )
    .\]
    Since the maximum displacement of a unit Brownian bridge in one dimension has Gaussian tails, we deduce that as $t\to 0$:
    \[ \exists c'>0, C'<\infty:
    \mathbb{P}_{t,x,x}(\Range(\ell)\cap X\neq \emptyset)\leq C' \exp({-c' t^{ -2 \epsilon} }).
    \]
    Since $\epsilon$ is arbitrary and $d'$ is arbitrary close to $d$, this concludes the proof.
\end{proof}

\begin{lemma}
  \label{le:item:tmax}
  Almost surely for all $\epsilon>0$, $\mathcal{L}_{\geq \epsilon}$ is finite.
  For all deterministic $\gamma\in \mathcal{C}([0,\cdot],\mathbb{R}^2)$ such that $\Range(\gamma)$ has dimension $d<2$, almost surely, $t_{\gamma,\mathcal{L}}<\infty$ and $\mathbb{E}[ t_{\gamma,\mathcal{L}}]<\infty $.
\end{lemma}
\begin{proof}
  The first property follows directly from the fact that the intensity measure of $\mathcal{L}_{\geq \epsilon}$ is finite.

  For the second property, notice that by Campbell's theorem for Poisson processes (see e.g. \cite{Kingman}),
  \[
  \mathbb{E}[ t_{\gamma,\mathcal{L}}]= \int_0^\infty \int_D \frac{t }{2\pi t^2} \mathbb{P}_{t,x,x}(\Range(\ell)\subset D,  \Range(\ell)\cap \Range(\gamma)\neq \emptyset   ) \d x\d t,
  \]
  where the right-hand side is finite if and only $t_{\gamma,\mathcal{L}}$ is almost surely finite and in $L^1(\mathbb{P})$.
  The extra factor $t$ in the numerator comes from the lengths $t_\ell$ we are summing.

  By Lemma~\ref{le:tech1},
  we deduce, as $t\to 0$,
  \[
  \int_D \frac{t }{2\pi t^2} \mathbb{P}_{t,x,x}(\Range(\ell)\cap \Range(\gamma)\neq \emptyset ) \d x \leq t^{-d/2+o(1)},
  \]
  which is integrable in a neighbourhood of $0$ since $d<2$.
  Near $+\infty$, the integrability follows from the exponential decay of $\mathbb{P}_{t,x,x}(\Range(\ell)\subset D)$.
  Thus, $  \mathbb{E}[ t_{\gamma,\mathcal{L}}  ]<\infty$ and in particular $t_{\gamma,\mathcal{L}}$ is almost surely finite.
\end{proof}

\begin{lemma}
  \label{le:item:unifcont}
  Almost surely, the set $\mathcal{L}$ is equicontinuous.
\end{lemma}
\begin{proof}
  Fix $\epsilon>0$.
  As $t$ goes to $0$, $\mathbb{P}_{t,x,x}( \|X-x\|_\infty\geq  \frac{\epsilon}{2} )=
  \mathbb{P}_{1,0,0}( \|X\|_\infty\geq  \frac{\epsilon}{2 \sqrt{t}})
  $ goes to $0$ exponentially fast in $t$: there exist positive constants $c,C$ (which depend on $\epsilon$) such that for all $t\leq 1$,
  \[\mathbb{P}_{t,x,x}( \|X-x\|_\infty\geq  \frac{\epsilon}{2} ) \leq C e^{- c t^{-1} }.\]

  For $\delta\in (0,1]$, define the event
  $E_\delta\coloneqq \{ \forall \ell \in \mathcal{L}_{\leq \delta } ,   \|X-x\|_\infty<  \frac{\epsilon}{2} \}$.
  Then,
  \begin{align*}
  \mathbb{P}(E_\delta^c )
  \leq
  \mathbb{E}[ \# \{  \ell \in \mathcal{L}_{\leq \delta } : \|X-x\|_\infty\geq  \frac{\epsilon}{2}   ) \}]
&  \leq \int_0^{\delta}  \frac{|D| }{2\pi t^2}  \mathbb{P}_{t,x,x}( \|X-x\|_\infty\geq  \frac{\epsilon}{2} ) \d t\\
&  \leq
  C |D| \int_0^{\delta} \frac{e^{-ct^{-1} }}{2\pi t^2} \d t\underset{\delta\to 0}\longrightarrow 0.
  \end{align*}
  Thus $\mathbb{P}(\bigcup_{\delta>0} E_\delta)=1$, hence almost surely there exists $\delta_1$ (random) such that the event $E_{\delta_1}$ holds. On this event, for all $\ell\in  \mathcal{L}_{\leq \delta_1 } $, and for all $\delta\in(0,+\infty)$, it holds $\omega_\ell(\delta)<\epsilon$.

  For each $\ell\in \mathcal{L}_{> \delta_1}$, let $\delta_\ell>0$ such that
  $\omega_{\delta_\ell}(\ell)\leq \epsilon$, which exists by continuity of $\ell$. Since $\mathcal{L}_{>\delta_1}$ is finite, $\delta_2\coloneqq \min \{\delta_\ell : \ \ell \in \mathcal{L}_{> \delta_1} \}$ is well-defined and positive. Then, $\delta\coloneqq \min(\delta_1, \delta_2)$ satisfies $\omega_\mathcal{L}(\delta)\leq \epsilon$ as desired.
\end{proof}

In order to consider the property~\eqref{it:inj}, given $\gamma\in \mathscr{S}$, let $\nu_{t,x,y}$ denote the law of $\sigma_{\gamma,\ell}$ where $\ell$ is distributed according to $\P_{t,x,y}$. We set arbitrarily $\sigma_{\gamma,\ell}=+\infty$ in the event that $\ell$ does not intersect $\gamma$.
Likewise let $\nu_{t,x}$ denote the law of $\sigma_{\gamma,\ell}$ when $\ell$ is distributed according to $\P_{t,x}$.
\begin{lemma}
    \label{le:noAtom}
    For all $\gamma\in \mathscr{S}$, for all
    $x\in \mathbb{R}^2\setminus \Range(\gamma)$ and $t>0$, the distributions $\nu_{t,x}$ and $\nu_{t,x,x}$ have no atoms on $[0,t_\gamma ]$.
\end{lemma}
\begin{proof}
    Let us start with $\nu_{t,x}$. Fix $x\in \mathbb{R}^2\setminus \Range(\gamma)$. For any $\sigma\in[0,t_\gamma ]$,
    \[
    \nu_{t,x}(\{\sigma\})\leq \mathbb{P}_{t,x}(\exists u\in[0,t]: W_u=\gamma(\sigma)).
    \]
    By point polarity of Brownian motion in two dimensions, the right hand side is zero unless $\gamma(\sigma)=x$, which is prevented by the condition $x \notin \Range(\gamma)$.

    The same argument holds for $\nu_{t,x,x}$, where we remark that the point polarity for a Brownian bridge follows from the point polarity of standard Brownian motion: indeed, the restriction of a Brownian bridge of duration $t$ to $[0, t/2]$ is absolutely continuous with respect to that of an ordinary Brownian motion, and Brownian bridges are furthermore reversible (thereby showing that a given point cannot also be hit with positive over the interval $[t/2, t]$).
\end{proof}
\begin{lemma}
\label{le:item:inj}
    Let $\gamma\in \mathscr{S}$ such that $\lambda(\Range(\gamma))=0 $. Then,
    $\mathcal{L}$-almost surely,
    it holds that
    \begin{enumerate}
        \item For all $\ell \in \mathcal{L}_\gamma$, $\Theta_{\gamma,\ell}$ is reduced to a single element.
        \item For all $\ell,\ell' \in \mathcal{L}_\gamma$, $x_{\gamma,\ell}\neq x_{\gamma,\ell'}$.
    \end{enumerate}
\end{lemma}
\begin{proof}
    First, we have
    \begin{align*}
        &\mathbb{P}(\exists \ell \in \mathcal{L}, \exists s_1,s_2: s_1<s_2, \ \ell(s_1)=\ell(s_2)=x_{\gamma,\ell} ) \\
        &\leq \mathbb{E}[ \# \{ \ell \in \mathcal{L}: \exists s_1<s_2:  \ell(s_1)=\ell(s_2)=x_{\gamma,\ell} \}   ]\\
        &= \int_0^\infty \int_D \frac{1}{2 \pi t^2} \mathbb{P}_{t,x,x}( \ell \cap \gamma\neq \emptyset ,
        \exists s_1<s_2:  \ell(s_1)=\ell(s_2)=x_{\gamma,\ell}
        ) \d x \d t \\
        &\leq
        \int_0^\infty \int_D \frac{1}{2 \pi t^2}
        \sum_{s\in \mathbb{Q}}
        \mathbb{P}_{t,x,x}( \ell \cap \gamma\neq \emptyset ,
        \exists s_1<s< s_2 :  \ell(s_1)=\ell(s_2)=x_{\gamma,\ell}
        )\d x \d t \\
        &\leq
        \int_0^\infty \int_D \frac{1}{2 \pi t^2}
        \sum_{s\in \mathbb{Q}}
        \mathbb{P}_{t,x,x}( \theta^-_\ell<s \text{ and }
        \exists s_2>s :  \ell(s_2)=x_{\gamma,\ell}
        )\d x \d t .
    \end{align*}
    Thus, it suffices to show that for all $0<s<t$, for almost all $x\in D$,
    \begin{equation}\label{eq:doublepoints}
        \mathbb{P}_{t,x,x}( \theta^-_{\gamma,\ell}<s \text{ and }
    \exists s_2>s :  \ell(s_2)=x_{\gamma,\ell}
    )=0.
        \end{equation}
    As we will see, this follows from the fact that points are polar for a Brownian bridge and the Markov property of Brownian bridge at time $s$. Care is needed because the event $\theta^-_{\gamma,\ell} <s$ is not
    measurable with respect to the $\sigma$-algebra $\cF_s$ generated by the Brownian bridge up to time $s$.

    We have defined $x_{\gamma,\ell}$ only when $\ell$ is a loop, but the definition extends without modification to all $\ell\in \mathcal{C}([0,\cdot],\mathbb{R}^2)$ which intersects $\gamma$.
    On the event $\{ \theta^-_{\gamma,\ell} < s \}$, it holds $x_{\gamma,\ell} = x_{\gamma,\ell_{[0,s]} }$. Furthermore, the position $x_{\gamma,\ell[0,s]}$ is measurable with respect to $\cF_s$.

    Therefore,
    \begin{align*}
        \P_{t,x,x} ( \theta_{\gamma,\ell}^-<s \text{ and }
    \exists s_2>s :  \ell(s_2)=x_{\gamma,\ell} | \cF_s) & \leq\mathbb{P}_{t,x,x} ( \exists s_2 > s: \ell (s_2) = x_{\ell[0,s]} |\cF_s)\\
    & \leq \sup_{y,z}  \P_{t-s, y,x} ( \exists s_2 >0 : b_{s_2} = z).
    \end{align*}
  As the right hand side is zero, so is the unconditional probability~\eqref{eq:doublepoints}. This concludes the proof of the first point.

  For the second point, remark
    \begin{align}
        &\mathbb{P}(\exists \ell\neq \ell'  \in \mathcal{L}_\gamma: x_{\gamma,\ell}=x_{\gamma,\ell'} ) \nonumber\\
        &\leq \mathbb{E}[ \#
        \{ \ell, \ell'  \in \mathcal{L}_\gamma: \ell\neq \ell',  x_{\gamma,\ell}=x_{\gamma,\ell'} ) \}]\nonumber\\
        &= \int_0^\infty \int_D \int_0^\infty \int_D  \frac{1}{4 \pi^2 t^2 s^2}
        \nu_{t,x,x}\otimes \nu_{s,y,y} (\{(\sigma,\sigma): \sigma\in [0,t_\gamma] \}) \d x \d t \d y \d s  , \label{eq:temp:iszero}
    \end{align}
    which is $0$ by Fubini theorem and Lemma~\ref{le:noAtom}, using the fact $\lambda(\Range(\gamma))=0$.
\end{proof}

\begin{lemma}
\label{le:item:dense}
  Assume $\gamma\in \mathscr{S}$ and $\gamma(0,t_\gamma)\subset D$.
  Almost surely, the set $\{ \sigma_{\gamma,\ell} : \ell\in \mathcal{L}\}$ is dense on $[0,t_\gamma]$.
\end{lemma}
\begin{proof}
  Let $0< \sigma_1< \sigma_2 < t_{\gamma,\ell}$. Let $\sigma\in(\sigma_1,\sigma_2)$ and $x=\gamma(\sigma)$.
  First remark that for $\epsilon_0>0$ small enough, $B_{\epsilon_0}(x)\cap \Range(\gamma)$ is contained inside $\gamma((\sigma_1,\sigma_2) )$. Indeed, if we proceed by contradiction and assume this is not the case, let $\epsilon_n $ a sequence decreasing toward $0$ and $x_n\in (B_{\epsilon_n}(x)\cap \Range(\gamma))\setminus \gamma((\sigma_1,\sigma_2) ) $. Then $x_n$ converges toward $x$. Since $\gamma$ is a
  homeomorphism into its image, $\sigma_n\coloneqq \gamma^{-1}(x_n)$ converges toward $\gamma^{-1}(x)=\sigma$. Thus, for $n$ large enough, $\sigma_n\in (\sigma_1,\sigma_2)$, thus $x_n=\gamma(\sigma_n)\in \gamma((\sigma_1,\sigma_2) )$,   which is absurd.

  Let $\epsilon_1>0$ such that $\omega_{\epsilon_1}(\mathcal{L})<\epsilon_0/2$, which exists by Lemma~\ref{le:item:unifcont}.
  Then, for all
  $\ell \in \mathcal{L}_{\leq \epsilon_1}$,
  \[\sup\{ |\ell_t-\ell_s|: s,t\in [0,t_\ell]\}=\omega_{\epsilon_1}(\ell)<\epsilon_0/2.\] In particular, if there exists $t$ such that $\ell(t)\in B_{\epsilon_0/2}(x)\cap \Range(\gamma)$, then $x_{\gamma,\ell}\in  B_{\epsilon_0}(x)\cap \Range(\gamma)\subseteq  \gamma((\sigma_1,\sigma_2) )$, hence
  $\sigma_{\gamma,\ell}\in   (\sigma_1,\sigma_2)$.
  Thus, we are reduced to find a loop $\ell \in \mathcal{L}_{\leq \min(\epsilon_0,\epsilon_1)}$ such that
  $B_{\epsilon_0/2}(x)\cap \Range(\ell) \cap \Range(\gamma)\neq \emptyset $.

  Let $\epsilon_2<\min d( x, \gamma(\sigma_1)  ), d( x, \gamma(\sigma_2)  )$,
$\epsilon_3=\min \{d(x,y): x\in \gamma[\sigma_1,\sigma_2], y \in \mathbb{R}^2\setminus D $
   $\epsilon\coloneqq \min(\epsilon_0,\epsilon_1,\epsilon_2,\epsilon_3)$, and $n_0$ be an integer such that $2^{-n_0}<\epsilon/2$.
  For all $n\geq n_0$, define the annulus \[C_n\coloneqq B_{2^{-n}}(x )\setminus B_{2^{-n-1}}(x),\]
  and let $\partial^1 C_n$ and $\partial^2 C_n$ be the two connected components of its boundary.
  By the intermediate value theorem and since $2^{-n}<\epsilon_2$, there exists $\sigma_3,\sigma_4\in (\sigma,\sigma_2)$ such that $\gamma(\sigma_3)\in \partial^1 C_n$
  and $\gamma(\sigma_4)\in \partial^2 C_n$.
  For topological reason, any loop $\ell\in \mathscr{L}$ which remains inside $C_n$ and is non-contractible on $C_n$ must intersect $\gamma[\sigma_3,\sigma_4]$, in particular it must intersect $\gamma$. Furthermore, since $\epsilon<\epsilon_3$, any loop in the whole-plane Brownian loop soup $\mathcal{L}^{\mathbb{R}^2}$ which remains inside $C_n$ is also contained in $D$.
  Thus, for all $n$,
  \begin{align*} \mathbb{P}&(\exists
  \ell \in \mathcal{L}_{\leq \epsilon} :B_{\epsilon_0/2}(x)\cap \Range(\ell) \cap \Range(\gamma)\neq \emptyset
  )\\
  &\geq \mathbb{P}
  (
  \exists
  \ell \in \mathcal{L}^{\mathbb{R}^2}_{\leq \epsilon} : \Range(\ell)\subseteq C_n, \ \ell \mbox{ is non-contractible on $C_n$})\\
  &\geq \mathbb{P}
  \Big(
  \bigcup_{n=n_0}^\infty
  \{
  \exists
  \ell \in \mathcal{L}^{\mathbb{R}^2}: t_\ell\in (2^{-2n-2}, 2^{-2n}],\ \Range(\ell)\subseteq C_n, \ \ell \mbox{ is non-contractible on $C_n$} \} )
  \Big)
  \end{align*}
  By the Poisson property, the events
  \[ E_n \coloneqq
  \{
  \exists
  \ell \in \mathcal{L}^{\mathbb{R}^2} : t_\ell\in (2^{-2n-2}, 2^{-2n}],\ \Range(\ell)\subseteq C_n, \ \ell \mbox{ is non-contractible on $C_n$} \}
  \]
  are globally independent. By scaling property of the Brownian loop soup, $\mathbb{P}(E_n)$ does not depend on $n$. Since it is positive, $\sum_{n=n_0}^\infty \mathbb{P}(E_n)=\infty$. By the second Borel-Cantelli lemma, $\mathbb{P}(\bigcup_{n=n_0}^\infty E_{n})=1$, which concludes the proof.
\end{proof}

\subsection{Property~\texorpdfstring{\eqref{it:reg}}{(v)} for the Brownian loop soup}
Remember, for $\gamma\in \mathscr{S}$ and $x\in \gamma((0,t_\gamma))$ the sets $L(x)$, $R(x)$ we constructed in Section~\ref{sec:regPairs}, which split a vicinity of $x$ into a part `left side' and a `right side' of $\Range(\gamma)$.
\begin{lemma}
    \label{le:tech:cas0}
    Let $\gamma\in \mathscr{S}$.
    Let $x\in \gamma((0,t_\gamma))$, $\epsilon>0$, and $W$ be a planar Brownian motion started from $x$. Then, almost surely, there exists  $u,v\in [0,\epsilon]$ such that $W_u\in L(x)$ and $W_v\in   R(x)$.
\end{lemma}
\begin{proof}
    By symmetry at the vicinity of $x$, it suffices to prove the existence of $v$.
    Let $(K_{n})_{n\in \mathbb{N}}$ be a compact exhaustion of $R(x)$, i.e. $K_{n}$ is compact, $K_{n}\subseteq K_{n+1}$ for all $n$, and $\bigcup_n K_{n}=R(x)$. Set
    $\pi:z\mapsto x+|z-x|$, and $\Pi_{n}=\pi(K_{n})=\{ x+|z-x|: z\in K_{n}\}  $.
    Let $T_{n}$ be the first time when $W$ hits $K_{n}$,
    $T'_{n}$ be the first time when $W$ hits $\Pi_{n}$,
    and $S_\delta$ be the first time when $W$ exits the ball $B_\delta(x)$. Since the increasing union
    $\bigcup_n  K_{n}$ is equal to $R(x)$, the increasing union $\Pi\coloneqq \bigcup_n  \Pi_{n}$ is equal to $\pi(R(x))=x+(0,a)$ for some $a>0$.

    For arbitrary $s>0$, the planar Brownian motion started from $x$ hits $\Pi$ before $s$ with probability $1$, and it follows that for any given stopping time $S$ which is almost surely non-vanishing,
    \begin{align*}
    \mathbb{P}_x( \exists n : T'_{n} \leq S)
    &= \mathbb{P}_x( \exists t \leq S: W_t\in \Pi)
    \geq \lim_{s\to 0}  \mathbb{P}_x( \exists t \leq s : W_t\in \Pi, s\leq S)\\
   & = \lim_{s\to 0}  \mathbb{P}_x( s\leq S)=\mathbb{P}(S>0)
    = 1.\end{align*}

    By Beurling's projection principle (see e.g. \cite[Eq. (2)]{Werner_Beurling}),
    \[
    \mathbb{P}(T_{n}\leq S_\delta)\geq  \mathbb{P}(T'_{n}\leq S_\delta).
    \]
    Taking the increasing union over $n$, we deduce that for an arbitrary $\delta>0$,
    \[
    \mathbb{P}( \exists n:  T_{n}\leq S_\delta)
    =\lim_{n \to \infty} \mathbb{P}(  T_{n}\leq S_\delta)
    \geq \lim_{n \to \infty}  \mathbb{P}( T'_{n}\leq S_\delta).  )
    =  \mathbb{P}( \exists n: T'_{n}\leq S_\delta).  )
    =1.
    \]
    Thus,
    \[
    \mathbb{P}( \exists n:  T_{n}\leq \epsilon )
    \geq
    \mathbb{P}( \exists n,\delta:  T_{n}\leq S_\delta \leq \epsilon )
    =
    \mathbb{P}( \exists \delta:  S_\delta \leq \epsilon )
    = \lim_{\delta\to 0}
    \mathbb{P}(   S_\delta \leq \epsilon )=1. \qedhere
    \]
\end{proof}

\begin{lemma}
    \label{le:tech:casT1}
    Let $\gamma\in \mathscr{S}$.
    Let $W$ be a Brownian loop, with given starting point $x\in \mathbb{R}^2$ and duration~$t$, and let $\epsilon>0$. Let $T$ be a stopping time such that almost surely in the event $T<+\infty$, it holds $W_T\in \gamma(0,t_\gamma)$. Then, there exist stopping times  $U,V\in [T,T+\epsilon]\cup \{+\infty\}$ such that almost surely in the even $T<+\infty$, it holds $U<+\infty$, $V<+\infty$, $W_U\in  L(W_T)$ and $W_V\in R(W_T)$.
\end{lemma}
\begin{proof}
    Fix $\delta>0$ and consider the event $E_\delta: T< t-\delta$, which increases as $\delta$ decreases toward~$0$. Since $x\notin \Range(\gamma)$ and $W$ is continuous, $\mathbb{P}(T=+\infty \text{ or }E_\delta)\underset{\delta \to 0}\longrightarrow 1$, so it suffices to prove the result conditionally on $E_\delta$. Since the restriction of $W$ to $[0, t-\delta]$ is absolutely continuous with respect to the distribution of a Brownian motion, it suffices to prove the result when $W$ is a Brownian motion started from $x$, with duration $t'=t-\delta$. It then follows from the strong Markov property of $W$ and Lemma~\ref{le:tech:cas0}.
\end{proof}

\begin{proposition}
    \label{prop:tech:casT}
    Let $\gamma\in \mathscr{S}$ and assume $\lambda(\Range(\gamma))=0$.
    Let $W$ be a Brownian loop, with given starting point $x\notin \Range(\gamma)$ and duration $t$. Then, almost surely on the event that $W$ hits $\gamma$, it holds $W$ is regular for $\gamma$.
\end{proposition}
\begin{proof}
    As we have already proved~\ref{item:prop1} in Lemma~\ref{le:item:inj}, it suffices to prove~\ref{item:prop2}, i.e. that almost surely, for all $\epsilon>0$, there exists $U,V\in [\theta_{\gamma,W}-\epsilon, \theta_{\gamma,W}+\epsilon]$ such that $W_U\in L(W_{\theta_{\gamma,W}})$ and $W_V\in R(W_{\theta_{\gamma,W}})$. By symmetry it suffices to prove the existence of $V$.

    Remember $\rhos(x)=d(x, \{ \gamma(0), \gamma(t_\gamma)\})$, and $\sigma^\pm_\gamma(x)$ defined in Subsection~\ref{sec:regPairs} as the last entry time of $\gamma$ inside $B_x(\rhos(x))$ before it reaches $x$ (resp. first exit time after it reaches $x$). For all $x\in \gamma((0,t_\gamma))$ and $y\in\gamma((\sigma^-_\gamma(x), \sigma^+_\gamma(x)))$,  \[R(y)\cap B_x(\rhos(x))\subset  R(x).\]

    For $x\in\gamma((0,t_\gamma))$, let $\delta_x\in (0,\rhos(x)]$ such that $B_x(\delta_x)\cap \Range(\gamma)\subset \gamma( (\sigma^-_\gamma(x), \sigma^+_\gamma(x))   )$ (the existence of which we already established in the proof of Lemma~\ref{le:item:dense}).
    Let $\epsilon_0$ such that $\omega_W (\epsilon_0)< \delta_{x_{\gamma,W}} $. Then, for all $s,t\in [\theta_{\gamma,W}-\epsilon_0,\theta_{\gamma,W}+\epsilon_0 ]$ such that $ W_t\in \gamma((0,t_\gamma))$ and $s\in R(W_t)$, it holds $|W_t-x_{\gamma,W}|< \delta_{x_{\gamma,W}} $, thus $W_t\in \gamma( (\sigma^-_\gamma(x), \sigma^+_\gamma(x))   )$, and $|W_s-x_{\gamma,W}|< \delta_{x_{\gamma,W}}\leq \rhos(x_{\gamma,W})$,
    hence
\begin{equation} \label{eq:temp:flipRpoint} W_s\in  R(W_t)\cap B_{x_{\gamma,W}}(\rhos(x_{\gamma,W}))\subset  R(x_{\gamma,W}).\end{equation}

    Fix $\epsilon\in(0, \epsilon_0]$, and set $j_0$ such that $j_0\epsilon\leq t <(j_0+1)\epsilon$. For all $j\in \{0,\dots, j_0\}$, define $T_j= \inf \{ s\in [j\epsilon, (j+1\epsilon)): W_s\in \gamma((0,t_\gamma))\}$ (where $\inf \emptyset=+\infty$), which clearly is a stopping time for all $j\in \{0,\dots, j_0\}$. Thus, we can apply Lemma~\ref{le:tech:casT1}, and we deduce that almost surely, for all $j$, there exists a stopping time  $V_j\in [T_j,T_j+\epsilon]\cup \{+\infty\}$ such that almost surely on the even $T_j<+\infty$, $W_{V_j}\in  R(V_j)$. Define also $V_{+\infty}=+\infty$, in order to deal with the even when $W$ does not hit $\gamma$. Let $J$ be the unique random index such that
    $\theta_{\gamma,W}\in [J\epsilon,(J+1)\epsilon)$ (or $J=+\infty$ in the event that $W$ does not hit $\gamma$). Then $V_{J}$ is a random variable, and in the event
    that $W$ hits $\gamma$ it holds that $V_{J}$ and $\theta_W $ both lie in $[J\epsilon , (J+1)\epsilon) $, hence $V_J \in [\theta_W-\epsilon,\theta_W+\epsilon] $. Furthermore,  $W_{V_J}\in  R( W_{T_j}  )$, and it follows from~\eqref{eq:temp:flipRpoint} that $W_{V_J}\in R(W_{\theta_{\gamma,W}})$.
    Thus, the random  time $V\coloneqq V_J$ satisfies the desired properties.
\end{proof}


\begin{corollary}
  \label{coro:item:reg}
  Let $\gamma\in \mathscr{S}$ and assume $\lambda(\Range(\gamma))=0$. Almost surely, for all $\ell\in \mathcal{L}_\gamma$, it holds $\ell\in \mathcal{L}_{reg,\gamma}$.
\end{corollary}
\begin{proof}
    \begin{align*}
    \mathbb{P}( \exists \ell \in \mathcal{L}_\gamma: \ell \notin \mathcal{L}_{reg,\gamma}   )
    &\leq \mathbb{E}[ \# \{ \ell \in \mathcal{L}_\gamma: \ell \notin \mathcal{L}_{reg,\gamma}   \}]\\
    &= \int_0^\infty \int_D p^D_{t,x,x} \mathbb{P}_{t,x,x}(  W \in \mathcal{L}_\gamma,W \notin \mathcal{L}_{reg,\gamma}  )	\d x \d t\\
    &=0 \qquad \text{by Proposition~\ref{prop:tech:casT} .}\qedhere
    \end{align*}
\end{proof}

\begin{proof}[Proof of Proposition~\ref{prop:conf}]
  For $\gamma$ deterministic,  it is the combination of Lemma~\ref{le:item:tmax},~\ref{le:item:unifcont},~\ref{le:item:inj},~\ref{le:item:dense}  and Corollary~\ref{coro:item:reg} (the fact $\lambda(\Range(\gamma))=0$ follows from $\dim(\gamma)<2$).

  For $\gamma\sim\nu^\gamma$, it follows from the fact that almost surely, $\gamma\in \mathscr{S}$ and the dimension of $\gamma$ is $5/4<2$ (see e.g. Theorem 1.1 of \cite{LawlerRezaei} for a stronger result), that $\nu^\gamma$-almost surely, $\nu^\mathcal{L}$-almost surely, $(\gamma,\mathcal{L})\in \Conf$. By Fubini theorem,
  $\nu^\gamma\otimes \nu^\mathcal{L}$-almost surely, $(\gamma,\mathcal{L})\in \Conf$.
\end{proof}

\subsection{Attaching a random walk loop soup to a loop erased random walk}
In this subsection, we prove Proposition~\ref{prop:RW=LERW+RWLS}, i.e. that for $\Lambda\subset \mathbb{Z}^2$, a loop erased random walk $\gamma$ from $0$ stopped when it exits $\Lambda$ (by edges), a random walk loop soup $\mathbb{L}$, and a uniform choice of tie-break $B$, the curve $\Att(\gamma,\mathbb{L},1,B)$ is a simple random walk stopped when it exits $\Lambda$ (by edges). Remark it is essentially an already known result; the key argument being~\cite[Proposition~9.4.1]{LawlerLimic}.
\begin{proof}[Proof of Proposition~\ref{prop:RW=LERW+RWLS}]
  Since we work with paths on the lattice parametrised at unit speed, we identify them with the list of their vertices, i.e. we identify a path $X$ with $(X(0), \dots, X(t_X))$. For $a,b$ neighbours in $\mathbb{Z}^2$, let $[a,b]$ be the path with duration $1$ given by $[a,b](0)=a$ and $[a,b](1)=b$. For two paths $X,Y$ such that $X(t_X)=Y(0)$, we define the concatenation $X\oplus Y$ by $t_{X\oplus Y}=t_X+t_Y$, $X\oplus Y(t)=X(t)$ for $t\leq t_X$, and $X\oplus Y(t)=Y(t-t_X)$ for $t\geq t_X$.
  For a lattice path $\eta$, recall $LE(\eta)$ is its loop erasure, as defined e.g. in  \cite{LawlerLimic}. Let $S$ be a simple random walk started from $0$ and stopped when it first exits $\Lambda$.
  Thus, we need to show that for any deterministic path $\eta$ from $0$ and stopped when it first exits $\Lambda$,
  \[
  \mathbb{P}(S=\eta)=\mathbb{P}( \Att(\gamma,\mathbb{L},1,B) =\eta ).
  \]
  By definition of the loop-erased random walk, $\gamma$ is equal in distribution to $LE(S)$. Thus, it suffices to prove that for any path $\eta$ from $0$, stopped when it first exits $\Lambda$,
  \[
  \mathbb{P}( S=\eta |LE(S)=\eta')= \mathbb{P}(\Att(\gamma,\mathbb{L},1,B)=\eta|
  \gamma= \eta' ), \qquad \eta'\coloneqq LE(\eta),
  \]
	i.e. that
	\[
	\mathbb{P}( S=\eta |LE(S)=\eta')= \mathbb{P}(\Att(\eta',\mathbb{L},1,B)=\eta ).
	\]
	For $j\in \{0, \dots, t_{\eta'}-1\}$, let $\omega_j$ be the loops (with eventually $t_{\omega_j}=0$) uniquely defined by the property that
  $\omega_i$ does not intersect $\eta_{[0,i-1]}$ and
  \[
  \eta= \omega_0\oplus [\eta'_0,\eta'_1]\oplus \omega_1\oplus[\eta'_1,\eta'_2] \oplus \dots \oplus \omega_{t_{\eta'}-1}\oplus [\eta'_{t_{\eta'}-1}, \eta'_{t_{\eta'}}].
  \]
  Then,
  \[
  \mathbb{P}( S=\eta)= 4^{-t_{\eta} }=4^{-t_{\eta'}} \prod_{j=0}^{t_{\eta'}-1 } 4^{-t_{\omega_j}},
  \]
  Since this holds for any $\eta$, it follows
  \[
  \mathbb{P}( S=\eta| LE(S)=\eta')= Z_{\eta'} \prod_{j=0}^{t_{\eta'}-1 } 4^{-t_{\omega_j}},
  \]
  for some constant $Z_{\eta'}$ which depends only on $\eta'$ and $\Lambda$.

  On the other hand, let $(L_j)_{j\in \{0,\dots t_{\eta'}-1\}}$ be the loops uniquely defined by the property that
  $L_i$ does not intersect $\eta'_{[0,i-1]}$ and
  \[ \Att(\eta',\mathbb{L},1,B)= L_0\oplus [\eta'_0,\eta'_1]\oplus L_1\oplus [\eta'_1,\eta'_2]\oplus \dots \oplus L_{k_{\eta'}-1}\oplus [\eta'_{t_{\eta'}-1}, \eta'_{t_{\eta'}}]. \]
  Thus,
  \[\mathbb{P}( \Att(\eta',\mathbb{L},1,B)=\eta)
  = \mathbb{P}( \forall j<t_{\eta'}, L_j=\omega_j).\]
  By the Poisson property of $\mathbb{L}$ and by \cite[Proposition 9.4.1]{LawlerLimic},
  \[ \mathbb{P}( \forall j<t_{\eta'},  L_j=\omega_j)
  =\prod_{j=0}^{ t_{\eta'}-1  }\mathbb{P}( L_j=\omega_j) =\prod_{j=0}^{ t_{\eta'}-1  } Z_j 4^{- t_{\omega_j}  },
  \]
  for some constants $Z_j$ which depends only on $j$, $\eta'$, and $\Lambda$.
  We deduce for any given $\eta'$, for all $\eta$ such that $LE(\eta)=\eta'$, the probabilities
  $\mathbb{P}( S=\eta| LE(S)=\eta')$ and $\mathbb{P}( \Att(\eta',\mathbb{L},1,B)=\eta) $ are proportional to each other. Since both are probabilities, summing over $\{\eta: LE(\eta)=\eta'\}$, we deduce that the proportionality constants are equal, i.e. that $Z_{\eta'}=\prod_{j=0}^{ t_{\eta'}-1  } Z_j$. Thus,
  \[
  \mathbb{P}( S=\eta| LE(S)=\eta') = \mathbb{P}( \Att(\eta',\mathbb{L},1,B)=\eta).\qedhere
  \]
\end{proof}

\section{Probabilistic aspects: convergence \texorpdfstring{$n\to \infty$}{n to infinity}}
\label{sec:probabilistic2}

The goal of this section is to establish Proposition~\ref{prop:converg}.

By Theorem 1.1 in \cite{LawlerViklund}, $\gamma^{nD}$ converges in distribution for the distance $\rho$ toward $\gamma$. By Skorokhod's representation theorem, we can and we do assume that all the $(\gamma^{nD})_{n\geq 1}$ and $\gamma$ are defined in the same probability space $(\Omega,\mathcal{F},\mathbb{Q})$ in such a way that the convergence holds in the almost sure sense.

As a starting point to deal with the convergence of $\mathbb{L}$, we will use the following result,
where $\alpha\coloneqq\frac{5}{8}$, $\theta\in ( \frac{2}{3}, 1)$, and
 $\epsilon_n\coloneqq n^{-2}( \frac{n^\theta-2}{2}+\alpha) $.
\begin{theorem}[{\cite[Theorem 2.2]{SapozhnikovShiraishi}} with $d=2$ --weakened version]
   \label{th:bls}
   There exist a constant $C<\infty$ and, for each integer $n\geq 2$, a coupling between $\mathcal{L}^{\mathbb{R}^2}$ and $\mathbb{L}^{\mathbb{Z}^2}$ and
   an event $E_{n}$, such that $\mathbb{P}(E_{n})\underset{n\to \infty}\longrightarrow 1$, and such that on the event $E_n$,
   there exists a bijection $\psi_n$ between $\mathcal{L}^{R }_{\geq \epsilon_n }$ and $\mathbb{L}^{n R-\frac{1}{2}}_{\geq n^\theta}$,
   such that for all $\ell \in\mathcal{L}^{R }_{\geq  \epsilon_n }$,
   \begin{equation}
   | t_{\ell} - t_{\phi_n\circ \psi_n(\ell)}| \leq \alpha n^{-2} \qquad  \mbox{and} \qquad d_\infty(\ell, \phi_n\circ \psi_n(\ell))\leq C n^{-\frac{1}{4}} \log(n).
   \label{eq:dinfty}
   \end{equation}
\end{theorem}
\begin{remark}
    The result in \cite{SapozhnikovShiraishi} is stronger. In particular, the probability $\mathbb{P}(E_{n})$ is explicitly controlled and shown to decay as fast as $n^{-3\theta+2}$. One could probably also use the recent stronger result of Qian \cite{Qian}.
\end{remark}

In the following, we remove the subscript $n$ from the notations $ \psi_n,\phi_n,\phi^*_n$, as these are clear from the context.

\begin{remark}
  \label{rem:dis}
    In Theorem~\ref{th:bls}, rather than a single Brownian loop soup $\mathcal{L}^{\mathbb{R}^2}$ and a family of random walk loop soups converging toward $\mathcal{L}^{\mathbb{R}^2}$ after rescaling,
    we have a copy $\mathcal{L}^{\mathbb{R}^2}(n)$ of the Brownian loop soup for each integer $n$. However, the space on which $\mathcal{L}^{\mathbb{R}^2}$ and $\mathbb{L}^{\mathbb{Z}^2}$ are taking their values is a countable product of Polish spaces, hence a Polish space. This allows to use the disintegration theorem, from which we can construct a single measurable space on which the couplings provided by Theorem~\ref{th:bls} are all defined simultaneously, and such that $\mathcal{L}^{\mathbb{R}^2}(n)=\mathcal{L}^{\mathbb{R}^2}$ does not depend on $n$. We will assume this in the following, but only for notational convenience: it is possible not to use the disintegration theorem and to read this section and the next keeping indeed a copy of $\mathcal{L}^{\mathbb{R}^2}$ for each $n$. Then $\mathbb{P}$ must be replaced with $\mathbb{P}^n$ and convergences in probability must be replaced with convergences in distribution.
\end{remark}

Considering Theorem~\ref{th:bls} and Remark~\ref{rem:dis}, let $\mathbb{P}'$ under which are defined $\mathcal{L}^{\mathbb{R}^2}$ and, for each positive integer $n$, a copy of $\mathbb{L}^{\mathbb{Z}^2}$, such that~\eqref{eq:dinfty} holds. We set $\mathbb{P}=\mathbb{Q}\otimes \mathbb{P}'$,
 under which the random paths $\gamma$ and the $\gamma^{nD}$ are also defined.


\subsection{\texorpdfstring{$\delta$}{delta}-isomorphism between the Brownian and random walk loop soups}

Theorem~\ref{th:bls} tells us about an isomorphism between loops in a square box contained $D$. The following proposition ensures that we can restrict ourselves to the loops in $D$. Intuitively, this is related to the fact that all the loops that stay in $\bar D$  also remain in the interior of $D$ (as can be readily checked using the Riemann mapping theorem and the conformal invariance of the Brownian loop soup).

\begin{proposition}
    \label{prop:blsBis}
    Consider the couplings from Theorem~\ref{th:bls}.
    Let $\epsilon>0$. There exist events $E_{\epsilon,n}\subset E_{n}$ such that:
    \begin{itemize}
    \item $\mathbb{P}(E_{\epsilon,n})\underset{n\to \infty}\longrightarrow 1$ for all fixed $\epsilon>0$.
    \item For all $n$ sufficiently large that $\epsilon_n<\epsilon$, on $E_{\epsilon,n}$,
    \begin{equation}
        \label{eq:temp:inclu}
        \mathbb{L}^{nD}_{\geq   2 n^2 \epsilon+2\alpha }\subseteq \psi( \mathcal{L}_{\geq \epsilon}^{D}) \subseteq \mathbb{L}^{nD}.
    \end{equation}
    \end{itemize}
\end{proposition}
\begin{proof}
    Let $C$ be the constant from ~\eqref{eq:dinfty}, $\delta_n=C n^{-\frac{1}{4}} \log(n)$,   and let
    \[F_{\epsilon,n}\coloneqq \{ \exists \ell\in \mathcal{L}^{R}_{\geq \epsilon}: \ell\not\subset D, \ell \subset D^{\delta_n } \},\quad \text{where}\quad D^\delta\coloneqq\{x: d(x,D)\leq \delta \}, \]
    \[ G_{\epsilon,n}
    \coloneqq \{ \exists \ell\in \mathcal{L}^{D}_{\geq \epsilon}: \exists y \in \Range(\ell): d(y,\mathbb{R}^2 \setminus  D)\geq \delta_n \}, \qquad E_{\epsilon,n}=E_{n}\setminus ( F_{\epsilon,n} \cup G_{\epsilon,n} ).
    \]
    The events $F_{\epsilon,n}$ and $G_{\epsilon,n}$ are decreasing in $n$. Since $\mathcal{L}^{R}_{\geq \epsilon}$ is finite,
    \[
    F_{\epsilon,\infty}\coloneqq \bigcap_{n\geq 1} F_{\epsilon,n}= \{ \exists \ell\in \mathcal{L}^{R}_{\geq \epsilon}: \ell\not\subset D, \ell \subset \bar{D} \}.
    \]
    Since $D$ is a simply connected domain, $\mathbb{P}(F_{\epsilon,\infty})=0$:
    this can be proved as Lemma~\ref{le:tech:cas0} and Lemma~\ref{le:tech:casT1}, {\it mutatis mutandis}. The fact that the corresponding set $\Pi$ contains a whole interval $x+(0,a]$ can be proved by using Riemann mapping theorem. Indeed, this allows for all $x\in \partial D$, to exhibit $\gamma\in\mathcal{C}([0,1],\mathbb{R}^2)$ such that $\gamma(0)=x$ and $\gamma(t)\in \mathbb{R}^2\setminus \bar{D}$ for all $t>0$. Then $x+ \pi(\gamma (0,1])=x+ (0,a] \subseteq \Pi$.
    Thus, $\mathbb{P}(F_{\epsilon,n})\underset{n\to \infty}\longrightarrow 0$.

    Furthermore, since $\mathcal{L}^{D}_{\geq \epsilon}$ is finite and $\Range(\ell)$ is compact for all $\ell \in \mathcal{L}^{D}_{\geq \epsilon}$,
    \[
    \bigcap_{n\geq 1} G_{\epsilon,n} = \{ \exists \ell\in \mathcal{L}^{D}_{\geq \epsilon}: \exists y \in \Range(\ell): d(y,\mathbb{R}^2 \setminus  D )=0 \}=\emptyset.
    \]
    It only remains to prove the inclusions~\eqref{eq:temp:inclu} hold on $E_{\epsilon,n}$.

    Let $\mathsf{L}\in \mathbb{L}^{nD}_{\geq   2 n^2 \epsilon+2\alpha } $. Since
    $2 n^2 \epsilon+2\alpha\geq     2 n^2 \epsilon_n+2\alpha
    =n^\theta+1/2>n^\theta$ and $nD\subset [-(nR-\frac{1}{2}), nR-\frac{1}{2} ]^2$, on $E_{n}$, $\mathsf{L}\in
    \mathbb{L}^{nR-\frac{1}{2} }_{\geq  n^\theta }\subseteq
    \psi( \mathcal{L}_{\geq \epsilon_n}^{R})$. Thus,  $\ell\coloneqq \psi^{-1}(\mathsf{L})\in \mathcal{L}^{R}_{\geq \epsilon_n }$ is well-defined. In order to prove the first inclusion in~\eqref{eq:temp:inclu}, it suffices to prove $t_\ell\geq \epsilon$ and
    $\ell\subset D$.

    By~\eqref{eq:dinfty}, $ t_{\ell} \geq t_{\phi(\mathsf{L})} - \alpha n^{-2}=n^{-2}k_{\mathsf{L}}/2-\alpha n^{-2} \geq
    n^{-2}(     2 n^2 \epsilon_n+2\alpha   )/2-\alpha n^{-2}
    =\epsilon$.

    Since $\ell\in \mathcal{L}^R_{\geq \epsilon}$, on $E_n\setminus F_{\epsilon,n} $, either $\ell\subset D$ or $\ell\not\subset D^{\delta_n}$.

    Let us assume $\ell\not\subset D^{\delta_n}$, and let then $t$ such that $\ell_t\notin D^{\delta_n}$, i.e. $d(\ell_t, D)>\delta_n$.
    By~\eqref{eq:dinfty}, $d( \phi(\mathsf{L})_t   ,D)\geq d(\ell_t, D)- d(\phi(\mathsf{L})_t   ,\ell_t)>\delta_n-\delta_n=0$. Thus $\phi(\mathsf{L}) \not\subset D$, which contradicts the fact $\mathsf{L}\in \mathbb{L}^{nD}$. We deduce $\ell \subset D$, which proves the first inclusion in~\eqref{eq:temp:inclu}.

    For the second inclusion, let $\mathsf{L}\in \psi( \mathcal{L}_{\geq \epsilon}^{D})$, $\ell=\psi^{-1}(\mathsf{L})$. Then, on $E_n\setminus G_{\epsilon,n} $, by~\eqref{eq:dinfty}, for all $t$, $d(\phi(\mathsf{L})_t, \mathbb{R}^2 \setminus D)\geq  d(\ell_t, \mathbb{R}^2 \setminus D) -d( \phi(\mathsf{L})_t, \ell_t)>\delta_n -\delta_n=0$. Thus $\phi(\mathsf{L})_t \in D$, thus $\mathsf{L}\in\mathbb{L}^{nD}$, which concludes the proof.
\end{proof}
For all $\epsilon>0$, there exists $n_0= n_0(\epsilon)$ such that for all $n\geq n_0$, $\delta_n\coloneqq \max(\epsilon_n, 2(\alpha n^{-2}+Cn^{-\frac{1}{4}}\log(n)))$ is smaller than $\epsilon$. Then, on the event $E_{\epsilon,n}$, $\phi$ induces an $\epsilon$-isomorphism from $\mathcal{L}$ to $\phi(\mathbb{L}^{nD})$.
In particular,
\[
\mathbb{P}( d(  \mathcal{L},\phi(\mathbb{L}^{nD} ) )>\epsilon)\leq \mathbbm{1}_{n<n_0}+\mathbb{P}( E^c_{\epsilon,n} )\underset{n\to \infty}\longrightarrow 0.
\]
Thus, $\phi(\mathbb{L}^{nD}) $ converges in probability toward $\mathcal{L}$ for the distance $d$.
\begin{lemma}
\label{le:omegaequibound}
  Let $C$ be the constant from Theorem~\ref{th:bls}. For all $\epsilon\in(0,e^{-5}]$ and $n\geq n_0(\epsilon)$, on the event $E_{\epsilon,n}$ defined in Proposition~\ref{prop:blsBis},
  \begin{equation}
  \label{eq:omegabound}
  \omega_{\epsilon}( \phi(\mathbb{L}^{nD}) )\leq  f_{\mathcal{L}}(\epsilon)\coloneqq \max(\epsilon^\frac{1-\theta}{2-\theta},\epsilon^\frac{1}{5}, \omega_{\epsilon}(\mathcal{L})+\frac{8C}{5} |\log(\epsilon)| \epsilon^{\frac{1}{5}}).
  \end{equation}
  For $\epsilon\geq e^{-5}$, for all $ n\geq n_0(e^{-5})$, on the event $E_{\epsilon,n}$, we have instead \[ \omega_{\epsilon}( \phi(\mathbb{L}^{nD}) )\leq
\max(\epsilon^\frac{1-\theta}{2-\theta},\epsilon^\frac{1}{5}, \omega_{\epsilon}(\mathcal{L})+  8C e^{-1})
    .\]
\end{lemma}
\begin{proof}
For $\mathsf{L} \in \mathbb{L}^{n D}$, since $\mathsf{L}$ is $1$-Lipschitz,
\begin{equation}
\label{eq:temp:omegastupidbound}
\omega_\epsilon(\phi(\mathsf{L}))=n^{-1}\omega_{ \epsilon n^2}(\mathsf{L})\leq\min (  n^{-1}t_{\mathsf{L}} ,\epsilon n ).
\end{equation}
This allows to deduce the following.
\begin{itemize}
\item If $t_{\mathsf{L}}\leq n^{\theta}$ and  $n< \epsilon^{-1/(2-\theta)}$ , then $ \omega_\epsilon(\phi(\mathsf{L}))\leq \epsilon n \leq  \epsilon^{ \frac{1-\theta}{2-\theta}  } $.
\item If $t_{\mathsf{L}}\leq n^{\theta}$ and  $n\geq \epsilon^{-1/(2-\theta)}$ , then $ \omega_\epsilon(\phi(\mathsf{L}))\leq n^{-1}t_{\mathsf{L}}\leq n^{\theta-1}\leq \epsilon^{ \frac{1-\theta}{2-\theta}  } $.
\item If $t_{\mathsf{L}}> n^{\theta}$ and $n< \epsilon^{-\frac{4}{5}}$,  then $ \omega_\epsilon(\phi(\mathsf{L}))\leq \epsilon n <\epsilon^\frac{1}{5}$.
\item If $t_{\mathsf{L}}> n^{\theta}$ and $n\geq \epsilon^{-\frac{4}{5}}$:
On the event $E_{\epsilon,n}$, for $C$ the constant from~\eqref{eq:dinfty},
\[
\omega_\epsilon(\phi(\mathsf{L}))\leq \omega_\epsilon(\psi^{-1}(\mathsf{L}))+2 d_\infty(\psi^{-1}(\mathsf{L}), \phi(\mathsf{L}) )
 \leq \omega_\epsilon(\mathcal{L})+
 2Cn^{-\frac{1}{4}}\log(n)
\]
In the case $\epsilon\leq e^{-5}$: $\epsilon^{-\frac{4}{5}}\geq e^4$. The function
 $x\mapsto x^{-\frac{1}{4}}\log(x)$ is decreasing on $[e^4,\infty)$, thus $n^{-\frac{1}{4}}\log(n)\leq \epsilon^{\frac{1}{5}} \log(\epsilon^{-\frac{4}{5}}) $. It follows
$2Cn^{-\frac{1}{4}}\log(n)\leq 2Cn^{-\frac{1}{4}}\log(n)$,
which proves~\eqref{eq:omegabound}. In the case $\epsilon \geq e^{-5}$, we use instead the fact $n^{-\frac{1}{4}}\log(n)\leq 4e^{-1}$ for all $n>0$. \qedhere
\end{itemize}
\end{proof}

In the next section, we will prove (Theorem~\ref{th:tCV}) that under these coupling, $t_{\mathcal{X}^n}$ converges in probability toward $t_\mathcal{X}$. We now explain why this is sufficient to prove Proposition~\ref{prop:converg}
\begin{proof}[Proof of Proposition~\ref{prop:converg} assuming Theorem~\ref{th:tCV}]
  Recall on the probability space $\mathbb{P}$, the random objects $\gamma^{nD}$, $\gamma$, $\mathcal{L}$, $\mathbb{L}^{nD}$ are all defined. Let $\epsilon>0$. Since $\mathcal{X}\in \Conf$, by Lemma~\ref{le:omegabound}, there exists $\delta>0$ (random) such that
  $B'_\delta(\mathcal{X})\cap J_{f_{\mathcal{L}}}\subseteq \tilde{B}_{\epsilon/2}(\mathcal{X})$, with the notations from Lemma~\ref{le:omegabound} and where
  $f_{\mathcal{L}}$ is defined by~\eqref{eq:omegabound}. Remark the fact $f$ is continuous and vanishes at $0$ comes from the equicontinuity of $\mathcal{L}$ (Lemma~\ref{le:item:unifcont}).

  For $\delta'>0$, consider \[F_{\delta',n} \coloneqq \{ \rho(\gamma^{nD}, \gamma)\leq \delta'/2\}.
  \]
  By construction $\gamma^{nD}$ converges in probability toward $\gamma$. 
  Thus, for any deterministic $\delta'>0$, there exists $n_1(\delta')$ such that for all $n\geq n_1$, $\mathbb{P}(F_{\delta',n} )\geq 1-\epsilon/4$. By Proposition~\ref{prop:blsBis}, there exists $n_2\geq n_1$ such that for all $n\geq n_2$,
  $\mathbb{P}(E_{\delta',n}\cap F_{\delta',n} )\geq 1-\epsilon/2$.
  On $ E_{\delta',n} \cap F_{\delta',n}$, for $n\geq  n_2$,
  \[
  d'_{\Conf_0}( \mathcal{X}^n, \mathcal{X})\leq \delta' \qquad \mbox{(recall $d'_{\Conf_0}$ from~\eqref{eq:defdprime}) }
  \]
  On $E_{\delta',n} \cap F_{\delta',n}\cap \{  \delta\geq \delta' \}\cap\{ n\geq n_2\}$, we therefore have $\mathcal{X}^n\in B'_\delta(\mathcal{X})$. By Lemma~\ref{le:omegaequibound}, we also have $\mathcal{X}^n\in J_{f_{\mathcal{L}}} $, thus $\mathcal{X}^n\in  \tilde{B}_{\epsilon/2}(\mathcal{X})$.

  By  Theorem~\ref{th:tCV}, $t_{\mathcal{X}^n}$ converges  toward $t_{\mathcal{X}}$. Thus, there exists $n_3\geq n_2$ such that for all $n\geq n_3$,
  the event $G_{\epsilon,n}\coloneqq  \{ |t_{\mathcal{X}^n} -t_{\mathcal{X}} |\leq \epsilon/2\}$ has probability at least $1-\epsilon/2$.
  On the large probability event $E_{\delta',n} \cap F_{\delta',n}\cap \{  \delta\geq \delta' \}\cap\{ n\geq n_2\} \cap G_{\epsilon,n} $, we have $d_{\Conf_0}( \mathcal{X}^n,\mathcal{X})\leq \epsilon$.
  Thus,
  \[
  \mathbb{P}(d_{\Conf_0}( \mathcal{X}^n,\mathcal{X})\!\geq\!\epsilon)\leq \mathbb{P}( (E_{\delta',n}\cap F_{\delta',n} )^c   )
  +\mathbb{P}( G_{\epsilon,n}^c)+ \mathbb{P}( n\leq n_2)+\mathbb{P}( \delta<\delta')\leq \epsilon+ \mathbb{P}( n\leq n_2)+\mathbb{P}( \delta<\delta').
  \]
  Since $n_2$ is almost surely finite, for all $\delta'$,
  \[\limsup_{n\to \infty}
  \mathbb{P}(d_{\Conf_0}(\mathcal{X}^n,\mathcal{X})\geq\epsilon)\leq\epsilon+ \mathbb{P}( \delta<\delta').\]
  Since this is true for arbitrary deterministic $\delta'>0$ and since $\delta$ is almost surely positive,
  \[\limsup_{n\to \infty}
  \mathbb{P}(d_{\Conf_0}( \mathcal{X}^n,\mathcal{X})\geq \epsilon)\leq \epsilon,\]
  which concludes the proof.
\end{proof}

%
%
%
%

\subsection{Estimation on the total time spent on small loops}

\label{SS:timesmallloops}

The goal of this subsection is to establish Theorem~\ref{th:tCV}, which roughly states that the random walk does not spend a macroscopic amount of time on microscopic loops, and that we can quantify this uniformly over the mesh size. This is obtained by estimating the size of the neighbourhoods of the loop-erased random walk $\gamma^{nD}$: in the continuum, we know the size of the $\delta$-neighbourhood of $\gamma$ is of order $\delta^{3/4+o(1)}$, from which we may expect the cardinal of the $k$-neighborhood of $S$ on $nD$ should be of order $n^2(k/n)^{3/4+o(1)}$, for $k<n$. We only prove it is of order $O(n^2(k/n)^{\epsilon})$ (Corollary~\ref{coro:vol}), which is sufficient for us to conclude. We prove this by estimating the probability that a given point $x$ lies in the vicinity of $\gamma^{nD}$. A first estimation is provided in Lemma~\ref{le:probbound2}, and gives a better estimation if we restrict to points whose distance to $\partial D$ is bounded below. This is sufficient to deal with domains whose boundary satisfies a very mild regularity condition, but we will remove this additional requirement using Lemma~\ref{le:probbound1}: when $x$ is close to $\partial D$, it is unlikely to lie
the vicinity of the (rescaled) loop-erased random walk, because it is already unlikely to lie in the vicinity of the (rescaled) random walk, since the random walk is likely to be killed at $\partial D$ before it goes close to $x$. We start with two technical estimations concerning the simple random walk: a lower bound on the probability that it winds once around $x$ without ever approaching it (Lemma~\ref{le:unifp}), and an upper bound on the number of upcrossing events (Lemma~\ref{le:upcross}). We will use these respectively in
Lemma~\ref{le:probbound1} and Lemma~\ref{le:probbound2}.

\medskip

For $r>0$ let $B_r^{\mathbb{Z}^2}=\{ x\in \mathbb{Z}^2: |x|< r \} $, and let $\partial B^{\mathbb{Z}^2}_r=\{x\in B^{\mathbb{Z}^2}_r: \exists y \in  \mathbb{Z}^2 \setminus B^{\mathbb{Z}^2}_r: |x-y|=1$.
For a simple random walk $S$ on $\mathbb{Z}^2$ (not necessarily started from $0$), let $\zeta_r(S)=\inf\{j: |S(j)|< r \}$ be the first entry time of $S$ inside $B^{\mathbb{Z}^2}_r$.  Let also $\zeta_0(S)$ be the first time $S$ hits $0$.
For $0\leq s<t<\zeta_0(S)$ such that $S(s)=S(t)$, we say $S_{[s,t]}$ is non-contractible if the continuous loop obtain by piecewise-linear interpolation of $S_{[s,t]}$ is non-contractible on $\mathbb{R}^2\setminus \{0\}$.

%

\begin{lemma}
\label{le:unifp}
For $r>2$ and $ x \in \partial B^{\mathbb{Z}^2}_{r}$,
    let \[ p_{x,r}\coloneqq
     \mathbb{P}_x( \exists s,t: s \leq t\leq \zeta_{r/2}(S ):  S_{[s,t]} \text{ is non-contractible}  ) .
    \]
    Then,  $\inf\{ p_{x,r}:
    r> 2, x \in \partial B^{\mathbb{Z}^2}_{r} \}   >0$.
\end{lemma}
\begin{proof}
    Let $W$ be a planar Brownian motion started from $0$ with covariance matrix $I/2$. Let $E$ be the event
\[ \begin{array}{c} \zeta_{3/4}(W )\leq 1 \text{ and there exists $s\leq t\leq \zeta_{3/4}(W)$ such that $W_{[s,t]}$ is non-contractible on $\mathbb{R}^2\setminus\{0\}$, }\\ \text{$|W(s)-1|> 1/4$, and $|W(s)-W( \zeta_{3/4}(W )  )|>1/4 $,}
\end{array}
\]
    and $p_0=\mathbb{P}(E)>0$. On the event $E$, for any curve $\ell$ such that $\| \ell-W \|_{\infty,[0,1]}\leq 1/4$, their exist $s'\leq t' \leq  \zeta_{1/2}(\ell) $ such that $\ell(s')=\ell(t')$ and $\ell_{[s',t']}$ is non-contractible  on $\mathbb{R}^2\setminus\{0\}$.

    Extend $S$ to $[0, \zeta_0(S))$ by piecewise linear interpolation, and for a given $r$ and $x=|x| e^{i \phi}\in\partial B^{\mathbb{Z}^2}_{r} $, define $\tilde{S}:t\mapsto e^{-i \phi} S( 2 r^2 t )   /r$. As $r\to \infty$, $\tilde{S}$ converges locally uniformly in distribution toward $W$. Let $r_0$ be large enough so that for all $r\geq r_0$, there exists a coupling of $W$ with $S$ such that the event  $F:\{\|W-\tilde{S}\|_{\infty,[0,1]}\leq 1/4\}$ has probability at least $1-p_0/2$. Thus,
    \[\inf\{ p_{x,r}: r\in [r_0, \infty), x\in \partial B^{\mathbb{Z}^2}_{r}  \} \geq\mathbb{P}(E \cap F)
\geq p_0/2>0.
\]

  On the other hand,
  for all $x\in \partial B^n_{r}$ and $r>2$, it holds $r\leq |x|+1$ and $|x|>1$.
Since $p_{x,r}$ is decreasing in the $r$ variable,
\[\inf\{ p_{x,r}: r\in (2,r_0], x\in \partial B^n_{r}  \}
\geq \min\{ p_{x,|x|+1}: x\in \mathbb{Z}^2, |x|>1   \}.
\]
For any
$x\in  \mathbb{Z}^2$ with $|x|>1$, there exists at least one deterministic path $\eta$ from $x$ and $s<t<\xi_{(|x|+1)/2}(\eta)$ such that $\eta_{[s,t]}$ is non contractible, so that $p_{x,|x|+1}>0$, which concludes the proof.
\end{proof}

In the following, for $W$ either discrete or continuous, and for $0<a<b$, we define recursively the $(a,b)$-\emph{upcrossing times} $\tau_i=\tau_i(a,b,W)$ and $\sigma_i=\sigma_i(a,b,W)$:
	\[
\tau_0=\sigma_0=0,
\tau_{i+1}\coloneqq \inf\{ j\geq \sigma_i: |W(j)|\leq a   \},
\sigma_{i+1}\coloneqq \inf\{ j\geq \tau_{i+1}: |W(j)|\geq b \},
\]
and we then define  the number of $(a,b)$-upcrossing of $W$ before the time $T$
\[N_{a,b,T}(W)\coloneqq \max\{ i : \tau_i(a,b)<T \}.\]
For $r>0$, we also define the exit time $\xi_r = \inf\{ j: |W(j)|\geq r \}$.
\begin{lemma}
    \label{le:upcross}
    Let $S$ be a simple random walk on $\mathbb{Z}^2$.	Then, there exists $C,C'$ such that for all $k,n\geq 2$ with $ C\log(n)\leq k\leq n/2$, for all $x\in \mathbb{Z}^2$,
    \[
    \mathbb{E}_x[ N_{k,2k,\xi_n}(S)   ] \leq C' \log_2( n/k).
    \]
\end{lemma}

\begin{proof} 
This can be proved either using a strong (KMT) coupling or directly. For instance, by Lemma~6.6.7 in \cite{LawlerLimic}, there exists $c>0$ such that if $\rho_k = \inf \{j\ge 0:| W(j) |\le k \}$ then we have 
$$
\P_x ( \xi_n < \rho_k) \ge p_{k,n}:=c (\log (k/n))^{-1}
$$
provided that $n\ge 2k$, uniformly over points $x\in \Z^2$ which are adjacent to $B(0,2k)$. Applying the strong Markov property iteratively, we deduce that 
$\mathbb{P}_x[ N_{k,2k,\xi_n}(S) \ge m  ] \leq C (1-p_{k,n})^m,$ and thus $ \mathbb{E}_x[ N_{k,2k,\xi_n}(S)]\le p_{k,n}^{-1},$ as desired.
\end{proof}

The next lemma shows that a simple random walk is polynomially unlikely to visit a point close to the boundary (thus this applies also to a loop-erased random walk).

\begin{lemma}
    \label{le:probbound1}
    There exist $C<\infty$ and $q>0$ such that for all positive integers $n,k$, for all $x\in n D$,
    \[ \mathbb{P}( d(x,S) \leq k)\leq C  d(x,0)^{-q} \max(d(x,n \partial  D), k)^q,
    \]
    where $S$ is the simple random walk on $\mathbb{Z}^2$ started from $0$ and stopped when it exits $n D$.
\end{lemma}
\begin{proof}
    Up to replacing $C$ with $\max(C,1 )$, we can assume $d(x,n \partial D)<d(x,0)$ and $k<d(x,0)$. We treat separately the cases $k< d(x,n \partial D)$ and $k\geq d(x,n \partial D)$.

    Let \[j_0\coloneqq \lfloor \ln_2(d(x,0)/d(x,\partial D))\rfloor,\] i.e. the maximal integer such that $2^{j_0}d(x,\partial D) \leq d(x,0) $, which is non-negative. For $j\in \{1,\dots, j_0-1\}$, consider the concentric disks \[ \mathbb{D}_j\coloneqq \{ y\in \mathbb{Z}^2 : |y-x|\in [  2^{j_0-j-1}d(x,n \partial D),  2^{j_0-j} d(x,n \partial D) )  \}, \]
    and let $\tau'_j$ the first time when $S$ enters $\mathbb{D}_j$, which form an increasing sequence of stopping times.

    Consider the event $E_j\subset\{\tau'_j<\infty\}$ that there exists $s,t$ with $\tau'_j\leq s <t\leq \tau'_{j+1}$ such that $S_{[s,t]}$ is non-contractible around $x$. By Lemma~\ref{le:unifp}, the probability that this event happens is bounded below by a positive $p$ which depends on no parameter. Furthermore, for this event to happen, $S$ must exits $nD$ : let us proceed by contradiction and assume the opposite (the reader may convince themselves by drawing that we can always find a path which remains outside of $D$ and connects the two boundary components of $\mathbb{D}_j$).  Let $D'$ be the complementary of the unbounded path-connected component delimited by $S_{[s,t]}$. Since $nD$ is simply connected and contains $S_{[s,t]}$, $D'\subset nD$. Since $S_{[s,t]}$ is non-contractible on $\mathbb{D}_j$, $D'$ contains the closed ball $\bar{B}$ centred at $x$ with radius $r=2^{j_0-j-1}d(x,n\partial D)\geq d(x,n \partial D)$. By compactness, there exists $z \in n \partial D$ such that $d(x,z)=d(x,n \partial D)$. In particular $z\notin n D$, but $z\in \bar{B}\subseteq D'\subset nD$, hence the contradiction.

    Thus, for $k< d(x,n \partial D)$, setting  $q= \ln_2 (p^{-1})$ and $C=p^{-2}$, we get
    \begin{align*}
    \mathbb{P}\big( d(x, S)\leq  k\big)&\leq
    \mathbb{P}\big( d(x,S)\leq  d(x,n \partial D)\big)
    \leq \mathbb{P} \big( \forall j\in \{1,\dots, j_0-1\} , (E_j)^c \big)\\
    &\leq \prod_{j=1}^{j_0-1} \mathbb{P} \big(  (E_j)^c\big|  \tau'_j<\infty \big) \leq p^{j_0-1}
    \leq C d(x,n \partial D)^q  d(x, 0)^{-q}.
    \end{align*}

    In the case $k\geq d(x,n \partial D)$, consider instead
    $j_1\coloneqq \lfloor \ln_2(d(x,0)/k )\rfloor$.
    Since $k<d(x,0)$ and $k>d(x,n \partial D)$, it holds $j_0\in\{0,\dots, j_1\}$.
    Then
    \begin{align*}
    \mathbb{P}\big( d(x,S)\leq  k\big)
    & \leq \mathbb{P} \big( \forall j\in \{1,\dots, j_1-1\} , (E_j)^c \big)
    \leq \prod_{j=1}^{j_1-1} \mathbb{P} \big(  (E_j)^c\big|  \tau_j<\infty \big) \\ &\leq p^{j_1-1}
    \leq C k^q  d(x, 0)^{-q},
    \end{align*}
    which concludes the proof.
\end{proof}

We now show that a loop-erased random walk is polynomially unlikely to visit a point in the interior of a domain. Here, bounding crudely the range of the LERW by the range of the walk is not sufficient (as the random walk visits all but a \emph{logarithmic} fraction of points). Instead, we must rely on the fact that random walks create loops in annuli at all scales which erase previous portions of the path. The result below could be obtained by applying Proposition 4.11 in \cite{BLR}, but we include a proof for completeness.

\begin{lemma}
    \label{le:probbound2}
    There exist $C',C''<\infty$ such that for any positive integers $k,n\geq 2$ with $C'\log(n)\leq k\leq n/2$, for all $x\in n D\setminus\{0\}$,
    \[
    \mathbb{P}( d(x,\gamma^{nD})\leq k  )\leq C'' \log_2(n/k) \, k^\frac{1}{2}\,  d(x, n \partial D \cup \{0\})^{-\frac{1}{2}} .\]
\end{lemma}
\begin{proof}
    Up to replacing $C''$ with $\max(C'',2^{3/2} )$, we can assume $4 k+4<d(x, n\partial D \cup \{0\})$.

    Let $S$ be a simple random walk on $\mathbb{Z}^2$ started from $0$, $\tau$ be its first exit time from $nD$, and let $\gamma$ be the loop-erasure of $S_{[0,\tau]}$, which is distributed as $\gamma^{nD}$.

    On the event $\{d(x,\gamma )\leq k\}$, there exists at least one couple $(j_0,s)$ such that $j_0\leq \tau $, $S_{j_0}=\gamma_s$, and $|S_{j_0}-x|\leq k$ (take e.g. $s$ such that $d(x,\gamma)=d(x,\gamma_s)$). Among the possible choices for $j_0$, let us choose the maximal one.

    We claim that $S_{[j_0+1,\tau]} \cap LE(S_{[0,j_0]})=\emptyset$.
    We set $s''$ the length of $LE(S_{[0,j_0]})$
    and we proceed by contradiction, assuming there exists $j_1\in [j_0+1,\tau]$, and $s'\in [0,s'']$ such that $S_{j_1}=  LE(S_{[0,j_0]})_{s'}$, and let $j_1$ be minimal with this property.
    Then,
    \[LE(S_{[0,j_0]})_{s'}=S(j_1)\neq
    S_{j_0}=LE(S_{[0,j_0]})_{s''}
    \]
The first equality is by definition of $(j_1,s')$, the inequality is by maximality of $j_0$ (recall $j_1>j_0$), and the last equality is due to the fact that loop-erasing a curve preserves its endpoints.
In particular, $s'\neq s''$, hence $s'<s''$. Since
$LE(S_{[0,j_0]})$ is injective and $ LE(S_{[0,j_0]})_{s''}=\gamma_s$, we deduce that $\gamma_s\notin LE(S_{[0,j_0]})_{[0,s']}=LE(S_{[0,j_1]})$. Furthermore, the maximality of $j_0$ ensures that there exists no $j\geq j_1$ such that $S_j=\gamma_s$. Recall loop-erasure is chronological: for arbitrary $\mathsf{P}, \mathsf{P}'$ such that $\mathsf{P}\oplus \mathsf{P}'$ is well-defined, $LE(\mathsf{P}\oplus \mathsf{P}')= LE( LE(\mathsf{P})\oplus \mathsf{P}')$.
Thus,
\begin{align*}
\gamma_s\notin LE(S_{[0,j_1]}) \cup S_{[j_1,\tau]}
=\Range(  LE(S_{[0,j_1]}) \otimes S_{[j_1,\tau]})
&\supseteq
\Range( LE(  LE(S_{[0,j_1]}) \otimes S_{[j_1,\tau]}) )\\
&=\Range(LE(S_{[0,\tau]})) =\Range(\gamma),
\end{align*}
hence the contradiction.

\smallskip

    \begin{figure}[ht]
        \begin{center}
        \includegraphics{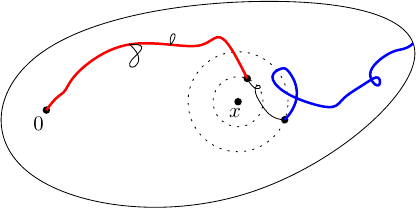}
        \end{center}
        \caption{The fatter parts of the trajectory ($ LE(S)_{[0,s]}$ in red, $S_{[\sigma_{M_0}, \tau ]} $ in blue) do not intersect each other.}
    \end{figure}

    Remark now that $j_0$ is not a stopping time for $S$, so we cannot apply the Markov property at this time, which we now overcome by using upcrossing times.   Let $\tau_i$, $\sigma_i$ and $N_{k,2k,\tau^{D_n}}$ be the $(k,2k)$-upcrossing times and number of upcrossing times of $S-x$.
    Since $2 k\leq  d(x,n \partial D)$, we have the equivalence $(\tau_i<\tau\iff \sigma_i<\tau) $.

Let $m_0$ be the unique random integer such that $\tau_{m_0}\leq j_0<\tau_{m_0+1}$. Since $|S_{j_0}-x|\leq k$, we have $\tau_{m_0}\leq j_0<\sigma_{m_0}<\tau_{m_0+1}$ (if we had $\sigma_{m_0}\leq j_0$, there would be an upcrossing in between $\tau_{m_0}$ and $\sigma_{m_0}$, followed by a downcrossing in between $\sigma_{m_0}$ and $j_0$, hence we would have $\tau_{m_0+1}\leq j_0$). Furthermore, on the event
$\{d(x,\gamma )\leq k\}$, $\tau_{m_0}\leq j_0<\tau$, hence $\sigma_{m_0}<\tau$.

  Since  $LE(S_{[0,j_0]})\cap S_{[j_0+1,\tau]} =\emptyset$,
  $LE(S_{[0,j]} )_{[0,j_0]}=LE(S_{[0,j_0]})$
  for all $j\geq j_0$.
  For a simple curve $\eta$, let $\pi(\eta)$ be the restriction of $\eta$ up to the first time $j$ such that $|\eta_j-x|\leq k$.
    Thus,
  \begin{align*}
j_0\leq \tau&\implies \sigma_{m_0}<\tau,  LE(S_{[0,\sigma_{m_0} ]})_{[0,\tau_{m_0}]}\cap S_{[\sigma_{m_0}+1,\tau]} =\emptyset\\
&\implies \exists m\in \mathbb{N}: \sigma_m<\tau,
   LE(S_{[0,\sigma_{m} ]})_{[0,\tau_{m}]}\cap S_{[\sigma_{m}+1,\tau]} =\emptyset\\
&\implies    \exists m\in \mathbb{N}: \sigma_m<\tau,
   \pi(LE(S_{[0,\sigma_{m} ]})) \cap S_{[\sigma_{m}+1,\tau]} =\emptyset.
  \end{align*}
 For any given $m$, we can now apply the Markov property at the stopping time $\sigma_m$:
    \begin{align*}
    \mathbb{P}(   d(x,\gamma)\leq k  )    &\leq \sum_{m=0}^\infty \mathbb{P}( \pi(LE(S_{[0,\sigma_{m} ]}))\cap S_{[\sigma_{m}+1,\tau]} =\emptyset )\\
    &\leq \sum_{m=0}^\infty \sum_{ \eta } \sup_y \mathbb{P}( \sigma_m\leq  \infty,\  \pi(LE(S_{[0,\sigma_{m} ]}) =\eta)\ \mathbb{P}_{y}(  S'( [1, \tau ] )\cap  \eta =\emptyset ) \\
    &\leq \Big(\sum_{m=0}^\infty \mathbb{P}( \sigma_m\leq  \infty)\Big)  \sup_{y,\eta} \mathbb{P}_{y}(  S'_{[1, \tau' ]} )\cap  \eta =\emptyset )\\
    &=	\mathbb{E}[N_{k,2k,\tau} ] \sup_{y,\eta} \mathbb{P}_{y}(  S'_{[0, \tau' ]} )\cap \eta =\emptyset  )	,
    \end{align*}
    where $\eta$ varies over the set of simple paths from $0$ on $\mathbb{Z}^2$, stopped at
    the first time $j$ such that $|\eta_j-x|\leq k$, $y$ varies over $\{ y\in n D: |y-x|\in [2k,2k +1)  \}$, $S'$ is a random walk started from $y$ under $\mathbb{P}_y$, and $\tau'$ is its exit time from $nD$.
    In these sets, the distance between $y$ and $\eta$ is smaller than $3k$, and the exit time $\tau'$ is bigger than
    \[
    \xi\coloneqq \inf\{j: |S'(j)-y|\geq  d(x,n \partial D \cup \{0\})- 2k-1    \},
    \]
    and it follows from the discrete Beurling estimate \cite[Theorem 6.8.1]{LawlerLimic} that there exists a universal constant $c>0$ such that
    \[
    \mathbb{P}_y(  S'_{[0, \tau' ]}\cap  \eta =\emptyset  )\leq c \sqrt{\frac{k  }{ d(x,n \partial D \cup \{0\})- 2k -1}}\leq \frac{c}{2} \sqrt{\frac{k  }{d(x,n \partial D \cup \{0\}) }}.
    \]
    Since $k\in [C\log(n), n/2]$, it follows from Lemma~\ref{le:upcross} and monotonicity of $t\mapsto N_{k,2k,t}$ that $\mathbb{E}[N_{k,2k,\tau} ]\leq C'\log(n/k)$ for a constant $C'$ which only depends  on the diameter of $D$. Thus,
    \[
    \mathbb{P}( d(x,\gamma)\leq k)\leq \frac{cC'}{2} \log(n/k) \sqrt{\frac{k}{d(x,n \partial D \cup \{0\}) }  }. \qedhere
    \]
\end{proof}

We can combine both previous estimates in a single inequality, as follows. 

\begin{proposition}
    \label{prop:probbound}
    There exist $C_0,C'<\infty$ and $\nu\in(0,1/2)$ such that for any integers $n\geq 2$ and $k\geq C'  \log(n)$, for all $x\in n D\setminus\{0\}$,
    \[
    \mathbb{P}_0( d(x,\gamma^{nD})\leq k )\leq C_0 d(x, 0)^{-\nu}  k^\nu \max(1, \log(n/k )).\]
\end{proposition}
\begin{proof}
    Up to replacing $C_0$ with $\max(C_0, 2^\nu \operatorname{Diam}(D)^{\nu})$, we can assume $k\leq n/2$.
    Let $C,C',C'',q$ be the constants from Lemma~\ref{le:probbound1} and~\ref{le:probbound2}. Recall that for $k\in [C'\log(n),n/2]$, we get
    \begin{equation}
    \label{eq:probbound1}
    \mathbb{P}( d(x,\gamma^{nD})\leq k )\leq \mathbb{P}( d(x,S) \leq k)\leq C  d(x,0)^{-q} d(x,n \partial D)^q,\end{equation}
    and
    \begin{equation}
    \label{eq:probbound2}
    \mathbb{P}( d(x,\gamma^{nD})\leq k )\leq C'' \log_2(n/k ) k^{\frac{1}{2}} d(x, n \partial D )^{-\frac{1}{2}}.
    \end{equation}
    Let $\theta\coloneqq q/(q+1/2)\in(0,1)$ and $\nu\coloneqq(1-\theta)q=\theta/2<1/2$. We discuss four cases, depending on the relative order between $k$, $ d(x,n \partial D)$, and $d(x,0)$.
    \begin{itemize}
    \item Case $d(x,0)\leq d(x,n \partial D)$. Since $2\nu<1$, using~\eqref{eq:probbound2}, we get
    \begin{align*}
        \mathbb{P}( d(x,\gamma^{nD})\leq k )
    \leq \mathbb{P}( d(x,\gamma^{nD})\leq k )^{2\nu}
    &\leq C''^{2\nu} \log(n/k)^{2\nu} k^\nu d(x, n\partial D )^{-\nu}\\
    &\leq C''^{2\nu} \max(1,\log(n/k)) k^\nu d(x, 0 )^{-\nu}.
    \end{align*}

  \item   Case $ d(x,n \partial D)\leq d(x,0)\leq k$. We simply use
    \[\mathbb{P}( d(x,\gamma^{nD})\leq k )\leq 1\leq  d(x,0)^{-\nu} k^\nu.
    \]

    \item Case $ d(x,n \partial D)\leq k \leq d(x,0)$. Since $\nu=(1-\theta)q\leq q$,\eqref{eq:probbound1} readily gives
    \[	 \mathbb{P}( d(x,\gamma^{nD}) \leq k)\leq C  d(x,0)^{-q} k^q\leq
    C  d(x,0)^{-\nu} k^\nu. \]

    \item Case $k\leq d(x,n \partial D)\leq d(x,0)$.
    We interpolate between~\eqref{eq:probbound1} and~\eqref{eq:probbound2} :
    \begin{align*}
    \mathbb{P}( d(x,\gamma^{nD})\leq k )
    &=\mathbb{P}( d(x,\gamma^{nD})\leq k )^{1-\theta}\mathbb{P}( d(x,\gamma^{nD})\leq k )^{\theta}\\
    &\leq C^{1-\theta} C''^\theta \log_2(n/k)^\theta
    k^{\theta/2} d(x,0)^{-(1-\theta) q} d(x, n \partial D)^{-\theta/2}   d(x,\partial D)^{(1-\theta) q} \\
    &=
    C^{1-\theta} C''^\theta \log_2(n/k)^\theta k^{\nu} d(x,0)^{\nu}<    C^{1-\theta} C''^\theta \log_2(n/k)k^{\nu} d(x,0)^{\nu}.
    \end{align*}
    \end{itemize}
    We conclude by taking $ C_0=\max( 1, C,C''  )$.
\end{proof}

\begin{corollary}
    \label{coro:vol}
    There exists $C',C_1<\infty$ and $\nu>0$ such that for all positive integers $n$ and $k\geq C' \log(n)$,
    \[
\mathbb{E}[\#\{ x\in n D: d(x,\gamma^{nD})\leq k  \}] \leq C_1  k^\nu n^{2-\nu}.
    \]
\end{corollary}
\begin{proof}
  Eventually replacing $C_1$ with $\max(C_1,2^\nu \lambda(D)) $, we can assume $k<n/2$. Let \[ C_\nu\coloneqq   \sup_{n\geq 1} \ \ n^{\nu-2} \hspace{-0.2cm}\sum_{x\in n D\setminus\{0\} } d(x,0)^{-\nu}, \]
  which is finite for all $\nu\in(0,2)$. Indeed, let $K$ be such that $D\subset [-K,K]^2$. Then,
    \begin{align*}
    \sum_{x\in n D\setminus\{0\} } d(x,0)^{-\nu}
    &\leq
    \sum_{x\in n D\setminus\{0\}}  \int_{x+[-1/2,1/2]^2  }    (|z|-\frac{1}{\sqrt{2}} )^{-\nu} \mathrm{d} z\\
    &  \leq \sum_{x\in n D\setminus\{0\}} \int_{x+[-1/2,1/2]^2  }  (\frac{2-\sqrt{2}}{2})^{-\nu} |z|^{-\nu }\mathrm{d} z\quad \text{( $r-\frac{1}{\sqrt{2}} \geq (1-\frac{1}{\sqrt{2}}) r$ for $r\geq 1$)}  \\
    &    \leq (1-\frac{1}{\sqrt{2}})^{-\nu} \int_{ [-n K+1/2,nK+1/2]^2 }|z|^{-\nu }\mathrm{d} z \\
    & = (1-\frac{1}{\sqrt{2}})^{-\nu} n^{2-\nu} \int_{[-K-1/(2n),K+1/(2n)]^2} |u|^{-\nu } \d u.
    \end{align*}
    The last integral is bounded by its value for $n=1$, which is finite for $\nu<2$, and it follows $C_\nu<\infty$.

    Let $\nu\in(0,\frac{1}{2})$ and $C_0$ be given by Proposition ~\ref{prop:probbound}. We get
    \begin{align*}
    \mathbb{E}[\#\{ x\in n D: d(x,\gamma^{nD})\leq k  \}]
    & = \sum_{x \in n D} \mathbb{P}(   d(x,\gamma^{nD})\leq k  )\\
    &\leq 1+
    C_0  \log_2(n/k) k^\nu \sum_{x \in n D\setminus \{0\} }  d(x,0)^{-\nu} \\
    & \leq 1+C_0 \log_2(n/k)  k^\nu n^{2-\nu} C_\nu \\
    & \leq C_1  (k/n)^{\nu'} n^2,
    \end{align*}
    for an arbitrary $\nu'\in (0,\nu)$, and for $C_1$ large enough.
\end{proof}
\begin{remark}
    For $K$ compactly contained in $D$, we can deduce from Lemma~\ref{le:probbound2} the bound
    \[
    n^{-2}\mathbb{E}[\#\{ x\in n K: d(x,\gamma^{nD})\leq k \}] \leq C|\log(k/n)| (k/n)^{\frac{1}{2}}.
    \]
   One should expect the optimal exponent the left-hand side is in fact of order $(k/n)^{\frac{3}{4}+o(1)}$. Proposition~\ref{prop:probbound} gives a worse bound, but allows us to replace $K$ with $D$.
\end{remark}

\begin{proposition}
    \label{prop:vol2}
    There exists $C<\infty$ and $\nu>0$ such that for all positive integers $n$ and $k$,
    \[
    \mathbb{E}\Big[ \#   \{   \mathsf{L}\in \mathbb{L}^{nD}_{\gamma^{nD}}: t_\mathsf{L}=k \}  \Big]
    \leq C n^{2-\nu}   k^{-2}
    \max(\sqrt{k},   \log(n))^\nu .
    \]
    As a consequence, (recalling that $T_{\leq \delta} (\mathcal{X}^n)$ denotes the total time spent by $\mathcal{X}^n$ on loops of duration less than $\delta n^2$), there exists $C<\infty$ such that for all $n\geq 1$ and all $\delta>0$,
    \begin{equation}
    \label{eq:temp:Ebound}
    \mathbb{E}[T_{\leq \delta}(\mathcal{X}^n)]
    %
    \leq C ((\log\log n) (\log n)^\nu n^{-\nu}+ \delta^{\frac{\nu}{2}}).
    \end{equation}
\end{proposition}
\begin{proof}
    We assume $k$ even during the proof, for otherwise the left-hand side is equal to $0$.
    Let $\mathsf{L}$ be distributed as a simple random walk loop on $\mathbb{Z}^2$ started from $0$ with $k$ steps under $\mathbb{P}_{k}$, and let
    \[r_\mathsf{L}\coloneqq \max \{ |\mathsf{L}(t)-\mathsf{L}(0)|:   t\in \{0,\dots, k\} \}.\]

    The following estimate on the tail of $r_\mathsf{L}$ is classical, and we include it here for completeness.

\begin{lemma}\label{L:ex_hoeffding} There exists a universal constant $C>0$ such that 
    \begin{equation} 
    \label{eq:temp:hoef}
    \mathbb{P}_k( r_{\mathsf{L}} \geq R \sqrt{k} ) \leq C \exp(-R^2).
    \end{equation}
\end{lemma}

\begin{proof}[Proof of Lemma \ref{L:ex_hoeffding}]
    Let $(X_n, n\ge 0)$ denote simple random walk on $\mathbb{Z}^2$ starting from $X_0 = 0$, and let $M_n = 
    \max \{|X_j|:0
    \le j \le n \}$.  
    By symmetry and reversibility of a random walk bridge of length $k$, it suffices to prove that 
    \begin{equation}\label{eq:AH}
    \mathbb{P}( M_{k/2} \geq R \sqrt{k} | X_k = 0) \leq (C/2) \exp(-R^2).
    \end{equation}
    (Recall that $k$ is assumed even.) 
    Now, let $j$ denote $k/2$ if $k/2$ is even, and $k/2+1$ otherwise. 
    Applying the Markov property, and letting $p_n(x,y)$ denote the transition probability for random walk on $\mathbb{Z}^2$, 
    \begin{align*}
    \mathbb{P}( M_{k/2} \geq R \sqrt{k} ; X_k = 0) 
    &\le \mathbb{P}( M_j 
    \geq  R \sqrt{k} ; X_k = 0) \\
    & \leq \mathbb{E}( 1_{
    \{M_{j} \ge R\sqrt{k}\}} p_{k-j}(X_j,0))  .
    \end{align*}
    Now, using Cauchy--Schwarz and the reversibility of simple random walk on $
    \mathbb{Z}^2$ it is well known that for any $x\in \mathbb{Z}^2$ with the same parity as 0 we have $p_{2n}(x,0) \le p_{2n} (0,0)$. This also trivially holds when $x$ does not have the same parity as 0. Hence (since we have chosen $j$ even, hence $k-j$ is also even), we deduce that 
    \begin{align}\label{eq:tailAH}
    \mathbb{P}( M_{k/2} \geq R \sqrt{k} ; X_k = 0)  &\le p_{k-j} (0,0) \mathbb{P}( M_j 
    \geq  R \sqrt{k} )
    \end{align}
    By the maximal Azuma--Hoeffding inequality,
    $$
    \mathbb{P}( M_j 
    \geq  R \sqrt{k} )\le \exp ( - 2\tfrac{R^2k}{j}) \le \exp ( - R^2).
    $$
    Also, by the local central limit theorem, there exists a universal constant $C>0$ such that $p_{k-j}(0,0) \le C p_k(0,0)$. Plugging into \eqref{eq:tailAH} we get
    \[
    \mathbb{P}( M_{k/2} \geq R \sqrt{k} ; X_k = 0)  \le C \exp(-R^2) p_k(0,0).
    \]
    Dividing by $
    P(X_k =0) = p_k(0,0)$ to get the conditional probability, this proves \eqref{eq:AH} (changing the value of $C$ by a factor two) as desired. 
\end{proof}

    Now let us return to the proof of Proposition \ref{prop:vol2}.    For any $k\geq 1$, 
    
    \begin{equation}
    \label{eq:temp:mass}
    \mathbb{E}[\# \{ \mathsf{L}\in \mathbb{L}^{\mathbb{Z}^2}: t_\mathsf{L}=k, \mathsf{L}(0)=0 \} ] =\frac{4^{-k}}{k}\binom{k}{k/2}^2 \leq Ck^{-2}
    \end{equation}
    for a universal constant $C$.
    Let $\nu, C',C_1$ be the constants from Corollary~\ref{coro:vol}, let $\tilde C$ be the constant from Proposition \ref{prop:vol2} and let $n\geq 2$.
    For $R$ a non-negative integer, apply Corollary~\ref{coro:vol} to $k_R\coloneqq\lceil \max((R+1)\sqrt{k},  C' \log(n) )\rceil $:
    \begin{equation}
    \label{eq:temp:card}
    \mathbb{E}[\#\{ x\in n D: d(x,\gamma^{nD})\leq k_R  \}] \leq C_1  n^{2-\nu} k_R^{\nu}.
    \end{equation}
    For $\mathsf{L}$ rooted at $x$ to intersect $\gamma^{nD}$, it must hold $r_\mathsf{L}\geq d(x,\gamma^{nD})$. Let $\mathbb{E}^*$ denote the conditional expectation on the loop soup, given the path $\gamma^{nD}$ (i.e., we fix the realisation of $\gamma^{nD}$ and we take the expectation just over the loop soup $\mathbb{L}^{\mathbb{Z}^2}$. 
    By independence between $\mathbb{L}^{\mathbb{Z}^2}$ and $\gamma^{nD}$ and translation invariance of $\mathbb{L}^{\mathbb{Z}^2}$, for all $\delta\in(0,1/2]$, 
    \begin{align*}
    \mathbb{E}^*\Big[ &\#   \{   \mathsf{L}\in \mathbb{L}^{nD}_{\gamma^{nD}}: t_\mathsf{L}=k \}  \Big]
    = \sum_{R=0}^\infty
    \mathbb{E}^*\big[\# \{\mathsf{L}\in  \mathbb{L}^{nD}_{\gamma^{nD} } : t_\mathsf{L}=k, r_{\mathsf{L}}\in [R\sqrt{k}, (R+1)\sqrt{k}   )   \} \big]\\
    &	\leq \sum_{R=0}^\infty
    \mathbb{E}^*\big[\# \{\mathsf{L}\in  \mathbb{L}^{\mathbb{Z}^2} :
    \mathsf{L}(0)\in nD, \
    d(\mathsf{L}(0),\gamma^{nD})\leq (R+1)\sqrt{k}, \\
    &\hspace{6cm}
    t_\mathsf{L}=k, r_{\mathsf{L}}\in [R\sqrt{k}, (R+1)\sqrt{k}   )   \} \big]\\
    &	\leq \sum_{R=0}^\infty  \#
    \{  x\in n D : d(x,\gamma^{nD} )\leq  (R+1) \sqrt{k}  \}
    \\
    &\hspace{2cm}\cdot \mathbb{E}^*\big[\# \{L\in \mathbb{L}^{\mathbb{Z}^2}: L(0)=0, t_\mathsf{L}=k, r_{\mathsf{L}}\in [R\sqrt{k}, (R+1)\sqrt{k}  \} \big]\\
    &	\leq \sum_{R=0}^\infty  \#
    \big\{  x\in n D : d(x,\gamma^{nD} )\leq  (R+1) \sqrt{k}  \big\}
    \big(Ck^{-2} \big) \big( \tilde C \exp(-R^2) \big)\quad \text{by (\ref{eq:temp:mass},\ref{eq:temp:hoef})}
    \end{align*}
    Taking now the expectation with respect to $\gamma^{nD}$ and applying \eqref{eq:temp:card}, we get
    \begin{align*}
    \mathbb{E}\Big[ \#   \{   \mathsf{L}\in \mathbb{L}^{nD}_{\gamma^{nD}}: t_\mathsf{L}=k \}  \Big]
    &	\leq \sum_{R=0}^\infty
    \big( C_1  n^{2-\nu} k_R^{\nu} \big) \big(Ck^{-2} \big) 
    \big(\tilde C \exp(-R^2)\big) \\&=\tilde C CC_1 k^{-2}n^{2-\nu}  \sum_{R=0}^\infty  \exp(-R^2) k_R^{\nu}.
    \end{align*}

    Since $k_R\leq 2 (R+1)(\max(\sqrt{k},\log(n) ))$ and $\sum_{R=0}^\infty   \exp(-R^2) (R+1)^\nu<\infty$, we deduce the first statement in Proposition \ref{prop:vol2}.
    
    Summing over $k=1$ to $\lceil \log(n)^2 \rceil$, and from $k=\lceil \log(n)^2 \rceil$ to $\delta n^2$ for an arbitrary $\delta>0$,
    we deduce that there exists a constant $C'$ such that for all $n\geq 1$ and all $\delta>0$,
    \[
    \mathbb{E}[T_{\leq \delta}(\mathcal{X}^n)]
    =n^{-2}
    \sum_{k=1}^{\lfloor \delta n^2 \rfloor}  k \mathbb{E}[\# \{  \mathsf{L}\in \mathbb{L}^{nD}_{\gamma^{nD}}: t_{\mathsf{L}}=k  \}]
    \leq C' ((\log\log n) (\log n)^\nu n^{-\nu}+ \delta^{\frac{\nu}{2}}).
    \]
    This proves \eqref{eq:temp:Ebound}.
\end{proof}

%
%
%
%

\begin{theorem}
\label{th:tCV}
  As $n\to\infty$, $t_{\mathcal{X}^n}$ converges to $t_{\mathcal{X}}$ in probability.
\end{theorem}
\begin{proof}
  Recall that $\mathcal{X} = (\gamma, \mathcal{L})$, where $\gamma$ is a radial SLE$_2$ curve. Thus $\gamma$ has almost surely finite $d$-dimensional Minkowski content for some $d<2$ (in fact one may take $d = 5/4$ \cite{LawlerRezaei}, but this is not relevant here).   By Lemma \ref{le:item:tmax}, when we take the expectation $\mathbb{E}'$ over only the loop soup,   
  $\mathbb{E}'[t_{\mathcal{X}} ]< \infty$.

Consider the event $F_k$ such that $t_{\mathcal{X}}\le k$. We first note since $t_{\mathcal{X}}< \infty$ almost surely, we have $\mathbb{P}( F_k) \to 1$ as $k \to \infty$. 

    Let $\epsilon>0$. Fix $k$ such that $\mathbb{P}(F_k) \ge 1-\epsilon/2$. 
  Let $f(\epsilon)=\mathbb{E}[\mathbbm{1}_{F_k} T_\epsilon(\mathcal{X})]$. 
  Since $t_{\mathcal{X}} < \infty$ a.s., we have $T_{\leq \epsilon}( \mathcal{X}) \to 0 $ as $\epsilon \to 0$, almost surely and hence by dominated convergence,
   $f(\epsilon)\to 0$ as $\epsilon \to 0$ as well. Let $\delta>0$ such that $C\delta^{\frac{\nu}{2}}+f(\delta)<\epsilon^2/2$, for $C$ from~\eqref{eq:temp:Ebound}.
  Let $\eta\in(0,\delta)$. Let $I_{\eta, n}$ denote the event that there is an $\eta$-isomorphism from $\mathcal{X}$ which is $\delta$-suited (recall Lemma~\ref{le:suited}). We have $\limsup_{\eta \to 0} \limsup_{n\to \infty} \P( I_{\eta, n}^c) = 0$ by combining Lemma~\ref{le:suited} and Proposition~\ref{prop:blsBis}. In particular, we can choose $\eta>0$ sufficiently small (in particular $\eta< \delta$) such that $\limsup_{n\to \infty} \mathbb{P}(I_{\eta, n}^c) \le \epsilon/2$. 

For all $n\geq 0$ such that $\delta\geq \epsilon_n$, since $I_{\eta, n} \subset E_{\delta , n} \subset E_n$ where $E_{\delta, n}$ is the good event from Proposition~\ref{prop:blsBis} and $E_n$ the good event from Theorem \ref{th:bls},
  \begin{align*}
  \mathbb{E}[ |t_{\mathcal{X}^n}-t_{\mathcal{X}}|\mathbbm{1}_{I_{\eta,n} \cap F_k} ]
  &\leq \mathbb{E}[T_{ \leq \delta}(\mathcal{X}_n)| ]+ \mathbb{E}[T_{\leq  \delta}(\mathcal{X})|1_{F_k} ]
  +\mathbb{E}[ \mathbbm{1}_{E_{n}} \sum_{\ell\in \mathcal{L}_{\gamma,\geq \delta}} |t_\ell-t_{\phi\circ \psi(\ell)}|    ]\\
  &\leq  C (\log\log n (\log n)^\nu n^{-\nu}+ \delta^{\frac{\nu}{2}}) +f(\delta)+ \alpha n^{-2}\mathbb{E}[  \#  \mathcal{L}_{\gamma,\geq \delta}].
  \end{align*}


  Thus, by Markov's inequality
  \begin{align*}
  \mathbb{P}(  |t_{\mathcal{X}^n}-t_{\mathcal{X}}| \geq \epsilon  )&\leq \mathbb{P}(I_{\eta,n}^c)+ \mathbb{P}(F_k^c) + \epsilon^{-1}(C (\log\log n (\log n)^\nu n^{-\nu}+ \delta^{\frac{\nu}{2}}) +f(\delta)+ \alpha n^{-2}\mathbb{E}[  \#  \mathcal{L}_{\gamma,\geq \delta}])\\
  &\leq  \mathbb{P}(I_{\eta,n}^c)+\epsilon/2+ C\epsilon^{-1} \log\log n (\log n)^\nu n^{-\nu}+ \epsilon^{-1} \alpha n^{-2}\mathbb{E}[  \#  \mathcal{L}_{\gamma,\geq \delta}]
  \end{align*}
  Hence 
  \[
\limsup_{n\to \infty} \mathbb{P}(  |t_{\mathcal{X}^n}-t_{\mathcal{X}}| \geq \epsilon  ) \le \epsilon/2 + \epsilon /2 
 = \epsilon. \]
 Since $\epsilon$ is arbitrary, this concludes the proof.
\end{proof}

\section{Open problems}
\label{sec:open}
\begin{itemize}
\item (\textbf{Measurability of $\gamma$.}) Naturally, our proof does not give any kind of uniqueness, and in particular, does not answer the question of whether one can ``erase the loops of planar Brownian motion'' and obtain an SLE$_2$ path. More precisely, one can ask the following question: is it possible to add a loop soup $\mathcal{L}$ to an SLE$_2$ path $\gamma$ to obtain a Brownian motion $W$ such that $\gamma$ (and also $\mathcal{L}_\gamma$) are a measurable function of $W$ in this coupling? 

\item (\textbf{Almost sure (non-)injectivity.})  A natural starting point for the above would be to check if the map $\Att$ is injective, In fact, one can show that this is not the case. This leaves open the following question: does $\Att$ almost surely has a unique pre-image? We believe this is also not the case, more precisely, almost surely there are infinitely many pre-images.

\item (\textbf{Dimension 3.}) It would be particularly interesting to prove a similar result in three dimensions. In this case, the (deterministic) continuity of the attachment map $\Att$ fails under the provided conditions, because the set of pairs $(\gamma,\ell)$ such that $\gamma$ intersects $\ell$ is no longer open (a small deformation can make $\ell$ detach from $\gamma$). However, it is natural to believe that for the scaling limit $\gamma$ of LERW, there is a kind of almost sure continuity:
``typical'' small perturbations should continue to intersect $\gamma$. Ideas of this type were in fact already used
 in the proof of \cite{SapozhnikovShiraishi}. 

\item (\textbf{SLE$_2$ loop analogue.}) Is there a way of describing a unit Brownian loop as a version of an SLE$_2$ loop, decorated by an independent Brownian loop soup?

\item (\textbf{Other values of $\kappa$}.) Our construction shows that we can attach a loop soup of arbitrary intensity $\lambda$ to a chordal SLE$_\kappa$ for other values of $\kappa$, to produce a continuous path $W^{\kappa, \lambda}$. It is also known that when the intensity of the loop soup is given by \[\lambda = \lambda(\kappa) = \frac{(8-3\kappa)(6-\kappa)}{2\kappa}\]
then the resulting hull (i.e. after filling in the holes) gives a restriction measure of exponent $\alpha = (6-\kappa)/ (2\kappa)$. But what is the law of the resulting continuous trajectory $W^{\kappa, \lambda}$ in that case? Does it correspond to a Brownian motion with 
``partially erased'' loops when $\kappa \ge 2$?

\item (\textbf{Continuity in rough path topologies.}) With a slight adaptation of Proposition~\ref{prop:extension}, we can deduce the $p$-variation estimate
\[\|X\|_{p-var}\leq C_p (\| \gamma\|_{p-var}+\sum_{\ell\in \mathcal{L}_\gamma} \|\ell\|_{p-var})\]
for an explicit constant $C_p$.
Can we prove some continuity of $\Att$ for  $\alpha$-H\"older or $p$-variation distances? What about rough path topologies? It may be the case that continuity in  these topologies holds, but
only along cusps on $\Conf'$, i.e. for a given $\mathcal{X}\in \Conf$,
consider only the $(\mathcal{X}',\lambda)$ where $\lambda\geq C d_{\Conf_0}(\mathcal{X}, \mathcal{X}')^\beta$, for a well-chosen $\beta$ and fixed $C$, and with $d_{\Conf_0}$ adapted to the appropriate norms.
\end{itemize}

\section*{Acknowledgments}

N.B.  acknowledges the support from the Austrian Science Fund (FWF) grants 10.55776/F1002 on ``Discrete random structures: enumeration and scaling limits" and 10.55776/PAT1878824 on ``Random Conformal Fields''. 

I.S. acknowledges the support of a Junior Research Fellowship (on gauge theory and Gaussian multiplicative chaos) from the Erwin Schrödinger Institute which made this collaboration possible. He also thanks the University of Vienna for its hospitality during this stay. Parts of this work were completed while he was at University of Warwick (with support from EPSRC grant EP/W006227/1) and at ENS Lyon (with support from Labex Milyon). Both these institutes and financial support are gratefully acknowledged.

We thank Adrien Kassel and Xinyin Li for some useful discussions.

\bibliographystyle{alpha}
\bibliography{bib.bib}

\end{document}